\documentclass[12pt]{article}
\usepackage[utf8]{inputenc}
\usepackage{footmisc}
\usepackage{verbatim}
\usepackage{fancyhdr}
\usepackage{graphicx}
\usepackage{epstopdf}
\usepackage{amsmath}
\usepackage{float}
\usepackage{amsmath}
\allowdisplaybreaks[4]
\usepackage{amsthm}
\usepackage{amssymb}
\usepackage{bm}
\usepackage{eucal}
\usepackage{color}
\usepackage[table]{xcolor}
\usepackage[figureright]{rotating}
\usepackage[titletoc]{appendix}
\usepackage[figureright]{rotating}
\usepackage{multirow}
\usepackage{subfigure}
\usepackage{tikz}
\usetikzlibrary{shapes.geometric}
\usetikzlibrary{plotmarks}
\usepackage{pgfplots}
\usepgfplotslibrary{groupplots}

\graphicspath{ {Figures/} }

\theoremstyle{plain}
\newtheorem{theorem}{Theorem}[section]

\theoremstyle{plainNoItalics}

\newtheorem{lemma}[theorem]{Lemma}

\newtheorem{remark}[theorem]{Remark}

\numberwithin{equation}{section}

\renewcommand{\div}{\nabla \cdot}
\newcommand{\R}{\mathbb{R}}
\newcommand{\F}{\mathbf{F}}
\newcommand{\n}{\mathbf{n}}

\newcommand{\ino}{\int_{\Omega}}
\newcommand{\V}{\mathcal{V}_h^r}
\newcommand{\VE}{\mathcal{V}_{h,i}}
\newcommand{\Vi}{\mathcal{V}_{h,i}^r}
\newcommand{\Vn}{\mathcal{V}_{h}^{n,r}}

\newcommand{\rM}{\gamma_M}
\newcommand{\rA}{\gamma_A}

\newcommand{\mt}{\mathcal{T}}
\newcommand{\me}{\mathcal{E}}
\newcommand{\mn}{\mathcal{N}}
\newcommand{\mf}{\mathcal{F}}

\newcommand{\Fhi}{\mathcal{F}_{h,i}}

\allowdisplaybreaks[4]
\begin{document}

\title{High order  discontinuous cut finite element methods for linear hyperbolic conservation laws with an interface\thanks{Research was supported by the Swedish Research Council Grants No. 2018-05262, No. 2018-05279 and the Wallenberg Academy Fellowship KAW 2019.0190. }
}
\author{Pei Fu\footnotemark[1], 
Thomas Frachon\footnotemark[2], 
Gunilla Kreiss\footnotemark[1], 
Sara Zahedi\footnotemark[2] 
}
\footnotetext[1]{Division of Scientific Computing, Department of Information Technology, Uppsala University, SE-75105 Uppsala, Sweden. E-mail: pei.fu@it.uu.se, gunilla.kreiss@it.uu.se.}
\footnotetext[2]{Department of Mathematics, KTH Royal Institute of Technology, SE-10044 Stockholm, Sweden. Email: frachon@kth.se, sara.zahedi@math.kth.se.}
\date{}
\maketitle

\begin{abstract}
We develop a family of cut finite element methods of different orders based on the discontinuous Galerkin framework, for hyperbolic conservation laws with stationary interfaces in both one and two space dimensions, and for moving interfaces in one space dimension. Interface conditions are imposed weakly and so that both conservation and stability are ensured. A CutFEM with discontinuous elements in space is developed and coupled to standard explicit time stepping schemes for linear advection problems and the acoustic wave problem with stationary interfaces. In the case of moving interfaces,  we propose a space-time CutFEM based on discontinuous elements both in space and time for linear advection problems.  We show that the proposed CutFEM are conservative and energy stable. For the stationary interface case an a priori error estimate is proven. Numerical computations in both one and two space dimensions support the analysis, and in addition demonstrate that the proposed methods have the expected accuracy.


\end{abstract}

\section{Introduction}
A finite element method (FEM) that uses discontinuous piecewise polynomial spaces as trial and test spaces is commonly called a discontinuous Galerkin (DG) method. Already in 1973 a DG method was introduced for the neutron transport  equation \cite{Reed1973Triangular}. Later for nonlinear time dependent hyperbolic conservation laws DG discretizations in space were coupled to Runge-Kutta time discretizations and limiters see e.g. \cite{ShuDG2,ShuDG3,ShuDG5}. The approach was shown to work well, retaining high order accuracy, conservation and other important properties, and has become very popular. For more details on DG methods, we refer to \cite{hesthaven2007nodal,shu2009discontinuous}.

Most finite element methods require the mesh to be aligned to boundaries and material interfaces and to achieve the full potential accuracy the mesh quality needs to be high. This type of requirements can be problematic for problems posed on complicated geometries or with material interfaces, especially when the geometry is evolving.  For time-dependent problems the time step may be severely restricted. To overcome these problems, approaches using fixed background meshes, unfitted to boundaries and interfaces, are of great interest. However, in a naive approach small cut elements will cause problems,  including ill-conditioned linear systems and severe time-step restrictions. Various techniques have been introduced to handle such difficulties.  One approach often used in unfitted methods based on DG is cell merging or agglomeration techniques where new elements of sufficient size are created by merging small cut elements with their neighbours \cite{Johansson2013,kummer2017extended,modisette2010toward,muller2017high,Qin2013}. A  common technique in connection with Cut Finite Element Methods (CutFEM) is to add ghost penalty stabilization terms in the weak form \cite{burman2010ghost,burman2012fictitious}.  In CutFEM the physical domain is embedded into a computational domain equipped with a quasi-uniform mesh. Elements that have an intersection with the domain of interest define the active mesh and associated to that is a finite dimensional function space and a weak form that together define the numerical scheme \cite{massing2014stabilized,hansbo2014cut,frachon2019cut,sticko2016stabilized}. Interface and boundary conditions are typically imposed weakly. For hyperbolic problems, CutFEM based on discontinuous piecewise polynomial spaces and ghost penalty stabilization has been developed, e.g. see \cite{gurkan2020stabilized}  where a time independent linear advection-reaction problem is considered and see \cite{fu2021high} for time dependent nonlinear conservation laws. We also refer to the recent work \cite{engwer2020stabilized} where a DG method for time dependent linear advection problems is developed with a stabilization of small elements that is designed to restore proper domains of dependence.

The focus of this paper is conservation at material interfaces. We consider time dependent linear hyperbolic conservation laws with discontinuous coefficients in the flux at a stationary or moving interface. We assume the problem has a structure  such that requiring conservation yields sufficient conditions at the interface for well-posedness. Typically the interface condition requires the solution to be discontinuous at the interface. Such problems can model for example wave propagation in materials where the wave speed changes abruptly at a material interface.

The first main result is an extension of the family of high order CutFEM with ghost penalty stabilization in~\cite{fu2021high} to problems with stationary interfaces. In the new method the solution is built from separate solutions on the two sides of the interface, and coupled through the interface condition, which is imposed weakly through penalties in the weak form. We show how to choose these penalties such that both conservation and stability is ensured, independently of how the interface cuts the elements. Numerical results demonstrate that conservation can indeed be lost with other choices. We note that our proposed CutFEM scheme is locally conservative in elements away from the interface and in the patch of elements involving the interface. The stability result is based on a semi-discrete energy analysis, which generalises the stability result in \cite{la2016well}, where a high order finite difference methodology with a grid  aligned with the interface is analyzed. We apply our method to the scalar advection equation in one and two space dimensions,  and to an acoustic system in one space dimension, but the proposed method can be applied to other hyperbolic systems with similar structure.

A second result is a space-time CutFEM for the case of a moving interface. We use a framework similar  to that
proposed in \cite{hansbo2016cut,zahedi2017space,frachon2019cut}, but here we use discontinuous elements both in space and time. The interface condition is imposed weakly as above, with the same restriction on penalties imposed by conservation. Our analysis as well as our numerical results show that the choice of weak form is important for achieving conservation in the discrete setting. Since the space-time formulation corresponds to an implicit time discretization, the method is computationally more demanding than the proposed method for the case of a stationary interface. However, we also demonstrate a strategy where the space-time CutFEM is restricted to the interface region, and coupled to a standard DG method with explicit time discretization in other parts of the domain.

The paper is organized as follows. In Section 2,  the model problem is given.  In Section 3,  we consider a stationary interface, propose a discontinuous cut finite element discretization in space and perform a stability analysis, and an a priori error estimate is given for the scalar problems. Numerical examples  show that the proposed  method has the expected convergence rate, is conservative, and allows for similar time steps as a corresponding standard DG method.  In Section 4 we consider a moving interface and propose a space-time CutFEM. The stability of the semi-discrete scheme is analyzed and we  present some examples to show that the method can simulate the moving interface problem with expected accuracy and with conservation.  In section 4.6, we formulate a locally implicit CutFEM. In Section~\ref{sec:2d}, we extend our scheme to the advection equation in two space dimensions with a material interface. Finally, in Section~\ref{sec:conclusion} we conclude.

\section{Model problem}
Let $x_\Gamma(t)$ be an interface that separates the domain $\Omega=[x_L, x_R]$
into two subdomains $\Omega_1=[x_L,x_\Gamma(t)]$ and $\Omega_2=[x_\Gamma(t),x_R]$. Consider the hyperbolic conservation law
\begin{alignat}{2}
&u_{t}+F(u)_{x}=0, && \quad {x \in \Omega_1\cup \Omega_2,  t>0,}  \label{eq:model} \\
&{u(x,0)=f(x),} &&  \quad   {x \in \Omega_1\cup \Omega_2,} \label{eq:initialcond} \\
&{[F(u)]_\Gamma-x_\Gamma'(t)[u]_\Gamma=0,} &&   \quad {
 t\geq0,} \label{eq:interfacecond}
\end{alignat}
 with suitable boundary conditions at $x_L$ and $x_R$.  Here the unknown conservative variable is
\begin{equation}\label{eq:u}
u=\left\{\begin{array}{ll}
 u_1, & x \in\Omega_1(t), \\
 u_2, & x\in\Omega_2(t),
\end{array} \right.
\end{equation}
which may be discontinuous across the interface $x_\Gamma(t)$ with jump
\begin{equation}\label{eq:defjumpu}
[u]_\Gamma=u_2(x_\Gamma,t)-u_1(x_\Gamma,t).
\end{equation}
We assume that $u_i$, $i=1,2$ are continuous functions with sufficiently many continuous derivatives, and that a discontinuity in $u$ may exist only at the interface. The flux function is
\begin{equation}\label{eq:flux}
F(u)=\left\{\begin{array}{ll}
F_1(u_1)\equiv A_1 u_1, & x \in\Omega_1(t), \\
F_2(u_2)\equiv A_2 u_2, & x\in\Omega_2(t),
\end{array}\right.
\end{equation}
with $A_1$ and $A_2$ being either constant scalars or matrices. We will only consider problems where the interface condition~\eqref{eq:interfacecond} ensures that the problem is well-posed. For the scalar case this means that $A_1-x_\Gamma'(t)$ and $A_2-x_\Gamma'(t)$ have the same sign. For systems the number of positive and negative eigenvalues of $A_i-x_\Gamma'(t)I$, $i=1,2$ with $I$ being the identity matrix, must be the same on both sides of the interface, and  the eigenstructure of the matrices must be such that the interface condition determines entering characteristic variables in terms of exiting characteristic variables.

Also, note that
\begin{align*}
\frac{d}{dt} \int_{x_L}^{x_R} u dx&=\frac{d}{dt}\int_{x_L}^{x_\Gamma(t)} u_1 dx+\frac{d}{d t}\int_{x_\Gamma(t)}^{x_R} u_2 dx\notag\\
&=\int_{x_L}^{x_\Gamma(t)}\partial_t u_1 d x+\int_{x_\Gamma(t)}^{x_R}\partial_{t} u_2 d x+u_{1}(x_\Gamma,t) x_\Gamma^{\prime}(t)-u_2(x_\Gamma,t)x_\Gamma^{\prime}(t)\notag\\
&=-\int_{x_L}^{x_\Gamma(t)} (F_1(u_1))_x d x-\int^{x_R}_{x_\Gamma(t)}(F_2(u_2))_xd x- x_\Gamma^{\prime}(t)[u]_\Gamma\notag\\
&=F_1(u_{1}(x_L,t))-F_1(u_1(x_\Gamma,t))+F_2(u_2(x_\Gamma,t))-F_2(u_2(x_R,t))-x_\Gamma^{\prime}(t)[u]_\Gamma.
\end{align*}
When the interface condition~\eqref{eq:interfacecond} is satisfied, i.e.
$$F_2(u_2(x_\Gamma,t))-F_1(u_1(x_\Gamma,t))-x_\Gamma^{\prime}(u_2(x_\Gamma,t)-u_1(x_\Gamma,t))=0,$$ we have
\begin{align}
\frac{d}{dt}\int_{x_L}^{x_R} u(x,t) dx=F_1(u_1(x_L,t))-F_2(u_2(x_R,t)).
\end{align}
Thus,  condition~\eqref{eq:interfacecond} ensures conservation of $u$.
In this paper, we will consider both a stationary interface, $x_\Gamma'=0$,  and a moving interface.

Before we propose a DG scheme for the problem \eqref{eq:model}-\eqref{eq:interfacecond} we introduce some notations. For square integrable scalar real valued functions on a given domain $K$, the standard notation is used for the  inner product and the $L^2$-norm, namely,
\begin{align}\label{L2}
(v,w)_{K}:=\int_{K} vw \ dx, \quad \|v\|_K:=\sqrt{(v,v)_{K}}, \quad \forall v, w\in L^2(K),
\end{align}
and for square integrable vector real valued functions $v, w$ with $m$ components, each component in $L^2$, we will use the same notation, but now $vw$ means the standard dot product, i.e.
\begin{align}\label{L2:vector}
(v,w)_{K}:=\int_K v^Twdx,
\quad \|v\|_K:=\sqrt{(v,v)_{K}}, \quad \forall v, w\in {[L^2(K)]}^{m}.
\end{align}
Furthermore
\begin{align}\label{L2union}
(v,w)_{\Omega_1 \cup \Omega_2}:= \sum_{i=1}^2 (v,w)_{\Omega_i}.
\end{align}

\section{Stationary interface}
Consider \eqref{eq:model}-\eqref{eq:interfacecond} in the case of a stationary interface, that is with  $x_\Gamma'=0$ and  interface condition  $[F(u)]_\Gamma=0$.  In the following we define the mesh, the space, and the weak formulation for a cut finite element method based on the DG framework.

\subsection{Mesh and spaces}
\label{sec:notapro}
Let $\mt_h$ be a quasi-uniform partition of the domain $\Omega$ generated independently of the position of the interface and let
$\me_h$ denote the set containing the edges in this mesh.  The mesh consists of intervals $I_{j}=[x_{j-\frac{1}{2}},x_{j+\frac{1}{2}}], j=1,\cdots,N$ with length $\Delta x_j=x_{j+\frac{1}{2}}-x_{j-\frac{1}{2}}$, and $x_{L}=x_{\frac{1}{2}}<x_{\frac{3}{2}}<\cdots<x_\Gamma<\cdots<x_{N+\frac{1}{2}}=x_{R}$. The mesh size is $h=\max_{1\leq j\leq N} \Delta x_j$. See Fig.  \ref{fig:cutmiddle1} for an illustration.

\begin{figure}[!htp]
\begin{tikzpicture}[xscale=4]
\draw[-][draw=black, very thick] (0,0) -- (.4,0);
\draw[dotted][draw=black, very thick] (0.4,0) -- (0.8,0);
\draw[-][draw=black, very thick] (0.8,0) -- (1.2,0);
\draw[-][draw=red, very thick] (1.2,0) -- (1.35,0);
\node[red][above] at (1.28,0.2) {${\alpha_1 h}$};
\draw[-][draw=red, very thick] (1.35,0) -- (1.6,0);
\node[red][above] at (1.5,0.2) {${\alpha_2h}$};
\draw[-][draw=black, very thick] (1.6,0) -- (2.0,0);
\draw[dotted][draw=black, very thick] (2.0,0) -- (2.4,0);
\draw[-][draw=black, very thick] (2.4,0) -- (2.8,0);
\draw [thick] (0,-.2) node[below]{$x_L=x_{\frac{1}{2}}$} -- (0,0.2);
\draw [thick] (.4,-.2) node[below]{$x_{\frac{3}{2}}$} -- (0.4,0.2);
\draw [thick] (0.8,-.2) node[below]{$x_{J-\frac{3}{2}}$} -- (0.8,0.2);
\draw [thick] (1.2,-.2) node[below]{$x_{J-\frac{1}{2}}$} -- (1.2,0.2);
\draw[dotted] [blue][thick] (1.35,-0.5) node[below]{$x_{\Gamma}$} -- (1.35,0.5);
\draw [thick] (1.6,-.2) node[below]{$x_{J+\frac{1}{2}}$} -- (1.6,0.2);
\draw [thick] (2.0,-.2) node[below]{$x_{J+\frac{3}{2}}$} -- (2.0,0.2);
\draw [thick] (2.4,-.2) node[below]{$x_{N-\frac{1}{2}}$} -- (2.4,0.2);
\draw [thick] (2.8,-.2) node[below]{$x_R=x_{N+\frac{1}{2}}$} -- (2.8,0.2);
\end{tikzpicture}
\caption{A uniform partition of  $\Omega$ with mesh size $h=(x_R-x_L)/N$. In this case, the interface splits the cell $I_{j}$ into two cells of length $\alpha_1 h$ and $\alpha_2h$.}\label{fig:cutmiddle1}
\end{figure}
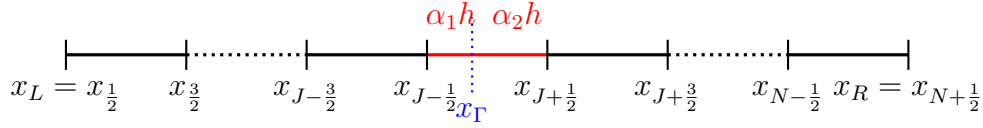
Define the following active meshes
\begin{equation}\label{eq: meshi}
  \mt_{h,i}=\left\{I_j \in \mt_h: I_j \cap \Omega_{i} \neq \emptyset\right\}, \quad  i=1,2,
\end{equation}
and the set of edges
\begin{equation}\label{eq: edgesi}
\me_{h,i}=\{e\in\me_h:  e\cap \Omega_i\neq \emptyset \}.
\end{equation}
Denote by $\Fhi$ the interior edge in $\me_{h,i}$ that belongs to a cut element, i.e.,
\begin{equation}
  \Fhi=\{e=I_j\cap I_k: I_j,I_k \in \mt_{h,i} \text{ and } x_\Gamma\in I_j \text{ or } I_k, \, j\neq k\}.
\end{equation}

Define the piecewise polynomial space
\begin{equation} \label{eq:BGspace}
\widetilde{\V}=\{v :  {v}|_{I_j} \in P^r(I_j), \,\forall I_j \in \mathcal{T}_h\},
\end{equation}
where $P^{r}(I_j)$ is the space of polynomials with degree at most $r$ on $I_j$. We note that if $v\in\widetilde{\V}$ is a vector, it means each of its component belongs to $P^r(I_j)$. Define the active finite element spaces
\begin{equation}
\Vi={\widetilde{\V}|_{\mt_{h,i}}}, \, i=1,2,
\end{equation}
and let $\V=\mathcal{V}_{h,1}^r \times \mathcal{V}_{h,2}^r$. Hence, with $v\in \V$ we mean $v=(v_1,v_2)$ with $v_i \in \Vi$.

For any $v(\cdot,t) \in \mathcal{V}_{h,i}^r$ at a fixed time $t\in[0,T]$, let $v^+ $ and $v^- $ denote the limit values of $v$ at $x$ from right and left, i.e.,
\begin{align}\label{eq: limitvx}
 v^-(x,t)=\lim\limits_{\epsilon\to 0^+}v(x-\epsilon,t), \quad v^+(x,t)=\lim\limits_{\epsilon\to 0^+}v(x+\epsilon,t).
 \end{align}
Define the average and the jump of the function $v$ at an edge $e\in\me_{h,i}$ by
\begin{align}\label{eq: jumpandmean}
 \{v\}_e=\frac{1}{2}(v^++v^-), \quad [v]_e=v^+-v^-.
 \end{align}
On the interface the average and jump of $v(\cdot,t)=(v_1(\cdot,t),v_2(\cdot,t))\in \V$ are defined by
\begin{align}\label{def:aver:jump:interface}
\{v\}_\Gamma=\frac{1}{2}(v_1(x_\Gamma,t)+v_2(x_\Gamma,t)),\quad [v]_\Gamma=v_2(x_\Gamma,t)-v_1(x_\Gamma,t).
\end{align}

\subsection{Weak formulation}
We now state a semi-discrete  weak formulation. For $t\in[0,T]$, find $u_h(\cdot,t)$ $\in \V$ such that
\begin{align}
\label{scheme:state:DG}
\left( (u_h(\cdot,t))_t,v_h \right)_{\Omega_1\cup\Omega_2}+\gamma_MJ_1((u_h(\cdot,t))_t,v_h)+A_h(u_h(\cdot,t),v_h)=0, \, t>0, \\
\label{scheme:initial}
(u_h(\cdot,0),v_h)_{\Omega_1\cup\Omega_2}+\gamma_MJ_1(u_h(\cdot,0),v_h)=(f(x),v_h)_{\Omega_1\cup\Omega_2},
\end{align}
for all $v_h\in\V$. Here
\begin{align}\label{scheme:Ah}
A_h(u_h,v_h)&=a_h(u_h,v_h)+\gamma_AJ_0(u_h,v_h),
\end{align}
with
\begin{align}\label{scheme:ah}
a_h(u_h,v_h)=&- (F(u_h),(v_h)_x)_{\Omega_1\cup\Omega_2} -\sum_{i=1}^{2}\sum_{e\in\me_{h,i}} \widehat{F}_e(u_h)[v_h]_e \nonumber \\
&-\left([F(u_h)v_h]_\Gamma+[F(u_h)]_\Gamma[\lambda v_h]_\Gamma\right),
\end{align}
and
\begin{equation}\label{stable0}
J_s(u_h, v_h)=\sum_{i=1}^2\sum_{e \in \Fhi} \sum_{k=0}^{r} \omega_kh^{2 k+s}\left[\partial^ku_{h,i}\right]_e\left[\partial^{k} v_{h,i}\right]_{e}.	
\end{equation}
The stabilization terms $J_s(u,v),s=0,1$ are added in order to have a stable scheme independently of how the interface cuts the background mesh. Otherwise, the mass matrix may be nearly singular, which can cause very severe time step restriction or ill-conditioning.
 The parameters in front of the stabilization terms, $\gamma_M,\gamma_A$ and $\omega_k$ are positive constants.  The choice is not unique and we choose $\omega_k=\frac{1}{(k!)^2(2k+1)}$ as in \cite{schoeder2020high}. The penalty parameter at the interface in \eqref{scheme:ah} is
 \begin{equation}\label{eq:lam}
\lambda=\left\{\begin{array}{ll}
 \lambda_1, & x \in\Omega_1(t), \\
 \lambda_2, & x\in\Omega_2(t).
\end{array} \right.
\end{equation} The choice of values will be discussed below.
We note that interfaces between elements and boundary edges are handled with the usual DG methodology, where we choose a single-valued function $\widehat{F}_e$ to approximate $F(u_h)$ on edge $e$.
In this paper,  the flux $\widehat{F}_e$  is chosen as
\begin{align}\label{eq:fluxe:state}
\widehat{F}_e(u_h)=\{F(u_h)\}_e-\frac{\lambda_e}{2}[u_h]_e,  \quad e\in\me_{h,i}.
\end{align}
Here, $\{F(u_h)\}_e=\frac{1}{2}\left(F(u_h^-)+F(u_h^+)\right)$, $[u_h]_e=u_h^+-u_h^-$, and $\lambda_e$ is an estimate of the largest absolute eigenvalue of the Jacobian $\frac{\partial F(u_h)}{\partial u_h}$ in the neighbourhood of edge $e$. This flux is known as the Lax-Friedrichs flux. In this paper we only consider the Lax-Friedrichs flux, but other monotone fluxes  are also possible.  At the boundaries of the domain, $x_L$ and $x_R$, we define the average and jump of a test functions $v_h$ as
 \begin{align}\label{eq: jumpandmean:bc}
 \{v_h\}_L=v_h^+, \,  \{v_h\}_R=v_h^-, \quad [v_h]_L=v_h^+,\, [v_h]_R=-v_h^-.
\end{align}
For scalar problems, we use the  inflow boundary condition $u(x,t)=g(t), x=x_L$ or $x_R$. To specify the corresponding values $u_h$ on $x_L$ and $x_R$, we set
 \begin{align}
 u_h^-(x_L,\cdot)=\left\{\begin{array}{ll}
g_{h}(\cdot), & \text{ if } a_i>0, \\
u_h^+(x_L,\cdot), & \text{ if } a_i<0,
\end{array} \right. \quad
 u_h^+(x_R,\cdot)=\left\{\begin{array}{ll}
g_{h}(\cdot), & \text{ if } a_i<0, \\
u_h^-(x_R,\cdot), & \text{ if } a_i>0.
\end{array} \right.
\end{align}
 Here $g_h$ is the approximation of the given boundary data $g$ at the inflow boundary. For the approximation of boundary conditions in systems see \cite{ShuDG3}.

To derive the weak formulation above we multiply equation \eqref{eq:model} by a test function $v\in \V$, integrate by parts, and enforce the interface condition \eqref{eq:interfacecond}. Since $v \in \V$ is discontinuous across interior edges, integration by parts results in jump terms  $[F(u)v]_e$ across interior edges.  Since $[u]_e=0$ these terms can be rewritten as
\begin{align}\label{eq:fluxone}
 [F(u)v]_e=\{F(u)\}_e[v]_e-\frac{\lambda_e}{2}[u]_e[v]_e=\widehat{F}_e(u)[v]_e.
\end{align}
Since the stabilization terms $J_s(u,v)$ vanish when $u$ is a sufficiently smooth exact solution, the proposed  formulation is consistent.

Taking test function $v_h=1$ in the scheme \eqref{scheme:state:DG}, we have
\begin{align}
\int_\Omega (u_h)_tdt &=\widehat{F}_L(u_h(x_L,t))-\widehat{F}_R(u_h(x_R,t))
+\left(\lambda_2-\lambda_1+1\right)[F(u_h)]_\Gamma.
\end{align}
Clearly, the proposed scheme is conservative only if \begin{equation}\label{eq:conservcondlamb}
\lambda_2-\lambda_1+1=0.
\end{equation}
If this condition is satisfied it follows that
\begin{align}
\frac{d}{dt}\int_\Omega u_h dx=\widehat{F}_L-\widehat{F}_R.
\end{align}
We note that our scheme is locally  conservative on the elements that along with their neighbors do not involve the interface, like $\{I_j\}_{j=1}^{J-2}$ and $\{I_{j}\}_{j=J+2}^{N}$ in Fig.  \ref{fig:conservationdomain}. The method is also locally conservative in subintervals containing elements which need to be stabilized, like the subinterval   $K_1=(I_{J-1}\cup I_{J})\cap\Omega_1$ and $K_2=(I_{J}\cup I_{J+1})\cap\Omega_2$ in Fig. \ref{fig:conservationdomain}.
For the scheme to be stable the penalty parameters need to satisfy additional requirements. See Theorem~\ref{thm:stabilityscalarp} for the scalar problem and Theorem~\ref{thm:stabilityacoustic} for the acoustic system, for choices that yield a conservative and energy stable interface treatment.

\begin{figure}[!htp]
\begin{tikzpicture}[xscale=4]
\draw[-][draw=black, very thick] (-0.2,0) -- (.4,0);
\draw[decorate, decoration={brace},yshift=1.5ex]  (0,0) -- node[above=0.4ex] {$I_{J-3}$}  (.4,0);
\draw [thick] (0,-.2) -- (0,0.25);
\draw [thick] (0.4,-.2) -- (0.4,0.25);
\draw[-][draw=black, very thick] (0.4,0) -- (0.8,0);
\draw [thick] (0.4,-.2) -- (0.4,0.25);
\draw [thick] (0.8,-.2) -- (0.8,0.25);
\draw[decorate, decoration={brace},yshift=1.5ex]   (0.4,0) -- node[above=0.4ex] {$I_{J-2}$}  (.8,0);
\draw[-][draw=black, very thick] (0.8,0) -- (1.2,0);
\draw [thick] (0.8,-.2)  -- (0.8,0.25);
\draw [thick][red] (1.2,-.2) -- (1.2,0.25);
\draw [thick][red] (1.2,-.3) node[below] {$\mf_{h,1}$} -- (1.2,0.25);
\draw[-][draw=black, very thick] (1.2,0) -- (1.35,0);
\draw[decorate, decoration={brace},yshift=1.5ex]   (0.8,0) -- node[above=0.4ex] {$K_1$}  (1.35,0);
\draw[-] [blue,very thick] (1.35,-0.5) node[below]{$x_{\Gamma}$} -- (1.35,0.5);
\draw[-][draw=black, very thick] (1.35,0) -- (1.6,0);
\draw[-][draw=black, very thick] (1.6,0) -- (2.0,0);
\draw[decorate, decoration={brace},yshift=1.5ex]   (1.35,0) -- node[above=0.4ex] {$K_2$}  (2.0,0);
\draw [thick][red] (1.6,-.3) node[below]{$\mf_{h,2}$} -- (1.6,0.25);
\draw[-][draw=black, very thick] (2.0,0) -- (2.4,0);
\draw [thick] (2.0,-.2) node[below]{$x_{J+\frac{3}{2}}$} -- (2.0,0.25);
\draw[decorate, decoration={brace},yshift=1.5ex]   (2.0,0) -- node[above=0.4ex] {$I_{J+2}$}  (2.4,0);
\draw [thick] (2.4,-.2) node[below]{$x_{J+\frac{5}{2}}$} -- (2.4,0.25);
\draw[-][draw=black, very thick] (2.4,0) -- (3.0,0);
\draw[decorate, decoration={brace},yshift=1.5ex]   (2.4,0) -- node[above=0.4ex] {$I_{J+3}$}  (2.8,0);
\draw [thick] (0,-.2) node[below]{$x_{J-\frac{7}{2}}$} -- (0,0.25);
\draw [thick] (.4,-.2) node[below]{$x_{J-\frac{5}{2}}$} -- (0.4,0.25);
\draw [thick] (0.8,-.2) node[below]{$x_{J-\frac{3}{2}}$} -- (0.8,0.25);
\draw [thick] (2.8,-.2) node[below]{$x_{J+\frac{7}{2}}$} -- (2.8,0.25);
\end{tikzpicture}
\caption{Illustration of intervals where the proposed scheme is locally conservative.}\label{fig:conservationdomain}
\end{figure}
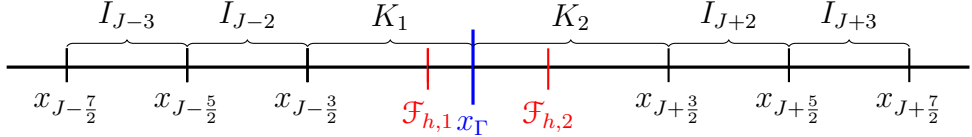

\subsection{Stability analysis} \label{sec:stability}
Here we will use the energy method to investigate how requiring stability of the proposed semi-discrete scheme will  restrict the choice of the penalty parameters $\lambda_1$ and $\lambda_2$ beyond \eqref{eq:conservcondlamb}.

\subsubsection{The scalar problem}\label{sec:scalar:state}
We consider the linear scalar problem, i.e.
\begin{equation}\label{eq:linear:state}
F(u)=au=\left\{\begin{array}{ll}
F_1(u)\equiv a_1u_1, & x \in \Omega_1, \\
F_2(u)\equiv a_2u_2, & x \in \Omega_2,
\end{array}\right.
\end{equation}
where $a_1,a_2$ are non-zero constants with the same sign.
 In the proposed scheme we use $\lambda_e=|a_i|$ in equation \eqref{eq:fluxe:state}.

Let $v_h=(u_{h,1},0)$ and $v_h=(0, u_{h,2})$ in \eqref{scheme:state:DG}, integrate
$(a_iu_{h,i},(u_{h,i})_x)_{\Omega_i}$, and take into account discontinuities of $u_{h,i}$ across the edges to get
\begin{align}
&\frac{1}{2}\frac{d}{dt} \left (\int_{\Omega_1} |u_{h,1}|^2 dx + \gamma_M J_1(u_{h,1},u_{h,1}) \right) +\sum_{e\in\me_{h,1}}\left(\frac{a_1}{2} [u_{h,1}^2]_e -\widehat{F_{e}}(u_{h,1})[u_{h,1}]_e\right) \notag  \\
&+(\frac{a_1}{2}-a_1\lambda_1)|u_{h,1}(x_\Gamma,t)|^2+\lambda_1a_2u_{h,2}(x_\Gamma,t)u_{h,1}(x_\Gamma,t) + \gamma_A J_0(u_{h,1},u_{h,1})=0, \notag\\
&\frac{1}{2}\frac{d}{dt}  \left ( \int_{\Omega_2} |u_{h,2}|^2 dx + \gamma_M J_1(u_{h,1},u_{h,1}) \right)
+\sum_{e\in\me_{h,2}}\left(\frac{a_2}{2}[u_{h,2}^2]_e-\widehat{F_{e}}(u_{h,2}) [u_{h,2}]_e\right)\notag \\
&-(a_2\lambda_2+\frac{a_2}{2})|u_{h,2}(x_\Gamma,t)|^2+
\lambda_2 a_1u_{h,2}(x_\Gamma,t)u_{h,1}(x_\Gamma,t)+ \gamma_A J_0(u_{h,2},u_{h,2})=0.
\label{energy}
\end{align}
By $\frac{a}{2}[u_{h}^2]_e=a\{ u_h \}_e[u_{h}]_e$ and the definition of the flux, equation \eqref{eq:fluxe:state}, we have
\begin{equation} \label{eq:edgeterms}
\frac{a_i}{2}[u_{h,i}^2]_e-\widehat{F_{e}}(u_{h,i})[u_{h,i}]_e=\frac{\lambda_e}{2} [u_{h,i}]^2_e,\quad i=1,2.
\end{equation}

Introducing a weighted energy $E_\eta$, where $\eta$ is a positive constant,
\begin{align}
E_\eta(t) = & \frac{1}{2} \left(\int_{\Omega_1} |u_{h,1}|^2 dx + \gamma_M J_1(u_{h,1},u_{h,1}) \right)\notag\\
& + \frac{\eta}{2} \left(\int_{\Omega_2} |u_{h,2}|^2 dx + \gamma_M J_1(u_{h,2},u_{h,2}) \right),
\label{eq:define:energy}
\end{align}
we have from \eqref{energy}, \eqref{eq:edgeterms} with $\lambda_e=|a_i|$ for $e\in \me_{h,i}$ that
\begin{align} \label{eq:energyA}
\frac{d}{dt} E_\eta=-\mathbf{u}_h^T S \mathbf{u}_h
&-\sum_{e\in\me_{h,1}}\frac{|a_1|}{2}[u_{h,1}]_e^2-\sum_{e\in\me_{h,2}}\frac{|a_2|\eta}{2}[u_{h,2}]_e^2 \nonumber \\
&-  \gamma_A J_0(u_{h,1},u_{h,1})-\eta \gamma_A J_0(u_{h,2},u_{h,2}).
\end{align}
Here
\begin{equation}\label{eq:matrixAs}
S=\left(
\begin{array}{cc}
   (\frac{1}{2}-\lambda_1)a_1 & \frac{a_2\lambda_1+ a_1\eta \lambda_2}{2}   \\
\frac{a_2\lambda_1+a_1\eta\lambda_2}{2}  & -(\lambda_2+\frac{1}{2})\eta a_2
   \end{array}
\right),\quad \mathbf{u}_h= \left(\begin{array}{c}u_{h,1}(x_\Gamma,t)\\u_{h,2}(x_\Gamma,t)\end{array}\right).
\end{equation}
If $S$ is positive semi-definite, then the energy is non-increasing with time. In  Appendix \ref{appendix:proof}, Lemma  \ref{lemma:matrixA} states conditions for $S$ to be positive semi-definite. We summarize the results for scalar problems in the following theorem.

\begin{theorem}\label{thm:stabilityscalarp}
Consider the discontinuous cut finite element method \eqref{scheme:state:DG} for the scalar problem \eqref{eq:model}-\eqref{eq:interfacecond} with $F(u)$ as in \eqref{eq:linear:state} with $x_\Gamma'(t)=0$ and stabilization parameters $\gamma_A\geq 0, \gamma_M\geq 0$. For penalty parameters $\lambda_1$ and $\lambda_2$, that satisfy  \eqref{eq:conservcondlamb} and
\begin{align}\label{eq:state:parametercondition}
\left\{\begin{array}{ll}
{\lambda_1\leq\frac{1}{2},\lambda_2\leq-\frac{1}{2},}&{\text{if } a_1>0,a_2>0,}\\
{\lambda_1\geq\frac{1}{2},\lambda_2\geq-\frac{1}{2},}&{\text{if } a_1<0,a_2<0,}\\
\end{array}\right.
\end{align}
the method is conservative and there exists a positive $\eta$ such that the energy defined in \eqref{eq:define:energy} does not grow with time.
\end{theorem}

\begin{remark}
The stability results in Theorem \ref{thm:stabilityscalarp} are derived assuming a conservative interface treatment with $\lambda_1$ and $\lambda_2$ satisfying   \eqref{eq:conservcondlamb}. Without the conservation condition \eqref{eq:conservcondlamb}, less restrictive stability results are possible. See Lemma \ref{lemma:matrix:stable} in Appendix A. In the right panel of Fig. \ref{exe:scalar:conservation} and Fig. \ref{fig:conservation2D}, we show that when the penalty parameters satisfy a stability condition but don't satisfy the conservation condition \eqref{eq:conservcondlamb},  the scheme produces large conservation errors.

\end{remark}

Based on the stability result we can  derive an a priori estimate for the scalar problems, if it has a sufficiently smooth solution.
\begin{theorem}\label{thm:scalarp:error}
Let $u(x,t)=u_i(x,t)$ for $x \in \Omega_i$, be the solution to problem \eqref{eq:model}-\eqref{eq:interfacecond} with $F(u)$ as in \eqref{eq:linear:state} and $x_\Gamma'(t)=0$. Assume $u_i$ is sufficiently smooth: $u_i, \,(u_i)_{t}  \in L^{\infty}\left([0,T] ; H^{r+1}({\Omega_i}) \right)$. Let $u_h=(u_{h,1}, u_{h,2}) $ with $u_{h,i}(\cdot,t)\in \Vi$  be the solution to the discontinuous cut finite element scheme \eqref{scheme:state:DG}. Then the following a priori error estimate hold
$$||u_1(\cdot,t)-u_{h,1}(\cdot,t)||^2_{\Omega_1}+||u_2(\cdot,t)-u_{h,2}(\cdot,t)||^2_{\Omega_2}\leq C h^{2r},\,t\in[0,T],$$
where $C$ is a constant which is independent of the mesh parameter $h$ and how the interface cuts the mesh.
\end{theorem}
The proof of this theorem is given in Appendix \ref{proofoferror}.
\subsubsection{The acoustic system}
We consider the acoustic system~\cite{piraux2001new}
\begin{alignat}{2} \label{System:acoustic:non1}
&\rho u_{t}+p_{x} =0, && \quad {x \in \Omega_1\cup \Omega_2,  t>0,} \nonumber \\
&p_{t}+\rho c^{2} u_{x}=0, && \quad {x \in \Omega_1\cup \Omega_2,  t>0,} \\
&[u]_\Gamma=0, \ [p]_\Gamma=0, &&   \quad {x=x_\Gamma(t), t\geq0,} \nonumber
\end{alignat}
with a stationary interface at $x=x_\Gamma$, i.e., $x'_\Gamma(t)=0$. Here, $u(x, t)$ is the velocity, $p(x, t)$ is the pressure,  $\rho(x)$  is the density, and $c(x)$ is the sound speed. The density and the sound speed are piecewise constant
$$
(\rho, c)=\left\{\begin{array}{ll}
\left(\rho_{1}, c_{1}\right) & \text {if } x \in \Omega_1, \\
\left(\rho_2, c_2\right) & \text {if } x \in \Omega_2.
\end{array}\right.
$$
By introducing $q=\frac{p}{\rho c^2}$ and $m=\rho u$, denoting strain and momentum, we can write problem \eqref{System:acoustic:non1} in the  conservative form as in equation \eqref{eq:model}-\eqref{eq:interfacecond},
\begin{alignat}{2}
&U_{t}+F(U)_{x}=0, && \quad {x \in \Omega_1\cup \Omega_2,  t>0,}    \label{System:acoustic:eq} \\
&{U(x,0)=f(x),} &&  \quad   {x \in \Omega_1\cup \Omega_2,} \label{eq:initialcondacoustic} \\
&[F(U)]_\Gamma=0, &&   \quad {x=x_\Gamma, t\geq0,}   \label{System:acoustic:interface}
\end{alignat}
with
\begin{equation}\label{eq:fluxacoustic}
F(U)=\left\{\begin{array}{ll}
A_1U_1, & x \in\Omega_1, \\
A_2U_2, & x \in\Omega_2,
\end{array}\right.
\end{equation}
where
\begin{equation}\label{eq:acousticsystem}
U_i=\left(\begin{array}{l}
m_i \\
q_i
\end{array}\right)  \textrm{ and }
A_i=\left(\begin{array}{cc}
0 &\rho_i c_i^{2} \\
 \frac{1}{\rho_i} & 0
\end{array}\right), \ i=1,2.
\end{equation}
The components of $U$, $m$ and $q$, are the conserved quantities, and are referred to as the conservative variables, while $u$ and $p$ are called the primitive variables.

We begin with showing an energy estimate in the continuous setting for  the physically motivated energy
\begin{equation}\label{acoustic energy}
E(t)=\frac{1}{2}\int_{\Omega_1\cup \Omega_2} U^TBU dx,\quad
B=\begin{cases}B_1, & x \in\Omega_1, \\B_2, & x \in\Omega_2, \end{cases},\quad B_i=\left(\begin{array}{cc}
  \frac{1}{\rho_i}&0 \\
 0&\rho_i c_i^2
\end{array}\right),\, i=1,2.
\end{equation}
Since $B$ is symmetric
\begin{equation}\frac{d}{dt}E=\int_{\Omega_1}U^T_1B_1\frac{dU_1}{dt}dx+\int_{\Omega_2}U^T_2B_2\frac{dU_2}{dt}dx.\end{equation}
Introduce $U_t=-(AU)_x$ and integrate by parts to get
\begin{align}
\frac{d}{dt}E=&-\frac{1}{2}\int_{\Omega_1\cup \Omega_2} U^T (A^TB^T-BA)U_x dx
+\frac{1}{2}\left(U_1^TA_1^TB_1^TU_1|_{x_L}^{x_\Gamma}+U_2^TA_2^TB_2^TU_2|_{x_\Gamma}^{x_R}\right)
\notag\\&
\equiv -ET+IT.
\end{align}
Note that $B_iA_i$ is symmetric and therefore the bulk terms vanish, and only terms  at the physical boundary, collected in
 $ET$, and terms at the interface, collected in $IT$, remain.
We can write the contributions at the interface as
\[IT=\frac{1}{2}\left(U_{1}^TA_1^TB_1^TU_{1}-U_{2}^TA_2^TB_2^TU_{2}\right)|_\Gamma
=\frac{1}{2}\left(U_{1}^T\left((A_2^TB_1)^T-A_1^TB_2\right)U_{2}\right)|_\Gamma.\]
The last equality follows using again that $B_iA_i$ is symmetric and the interface condition (\ref{System:acoustic:interface}), i.e., $U_{1}^TA_1^T=U_{2}^TA_2^T$. Note that
\[A_2^TB_1=\left(\begin{array}{cc}
0&  \frac{\rho_1c_1^2}{\rho_2} \\
 \frac{\rho_2 c_2^2}{\rho_1}&0
\end{array}\right),\quad
A_1^TB_2=\left(\begin{array}{cc}
0&  \frac{\rho_2c_2^2}{\rho_1} \\
 \frac{\rho_1 c_1^2}{\rho_2}&0
\end{array}\right).\]
Thus,
\begin{equation}\label{relation}
(A_2^TB_1)^T=A_1^TB_2,
\end{equation}
 and therefore the interface term vanishes, $IT=0$. We conclude that the interface term gives no contribution to energy growth or decay.

Next we study the stability of the proposed discontinuous CutFEM \eqref{scheme:state:DG} for the system \eqref{System:acoustic:eq}-\eqref{eq:acousticsystem}.  Guided by the energy result for the continuous system we define
\begin{equation}\label{eq:acousticenergyh}
E_h(t)=\frac{1}{2}\left(\int_{\Omega_1\cup \Omega_2} U_h^TBU_h dx + \rM J_1(U_{h},BU_h)\right),
\end{equation}
and take $v_h=BU_h$ in \eqref{scheme:state:DG} and get
\begin{align}\label{DG_energy}
\frac{d}{dt}E_h &= ((U_{h})_t,BU_h)_{\Omega_1\cup\Omega_2}
+ \rM J_1((U_{h})_t,BU_h) \notag\\
&=\int_{\Omega_1\cup\Omega_2}U_h^T B A (U_{h})_x \ dx
+\sum_{i=1}^{2}\sum_{e\in\me_{h,i}} \widehat{F}_e(U_h)^T[BU_h]_e \notag\\
&+([F(U_h)^TBU_h]_\Gamma+[F(U_h)^T]_\Gamma[\lambda BU_h]_\Gamma)
-\rA J_0(U_h,BU_h).
\end{align}
The bulk term can be integrated by parts
and therefore only interface and edge terms remain in the right hand side of (\ref{DG_energy}). Contribution from element edges when  the Lax-Friedrichs flux is used  are well-known, and therefore we only analyze the terms from the interface at $x_\Gamma$,
\begin{align}
IT = -&(F_1(\mathbf{U}_{h,1})+\lambda_1[F(\mathbf{U}_h)]_\Gamma)^TB_1\mathbf{U}_{h,1}
+(F_2(\mathbf{U}_{h,2})+\lambda_2[F(\mathbf{U}_h)]_\Gamma)^TB_2\mathbf{U}_{h,2} \nonumber \\
+&\frac{1}{2}\mathbf{U}_{h,1}^TA_1^TB_1\mathbf{U}_{h,1}-\frac{1}{2} \mathbf{U}_{h,2}^TA_2^TB_2\mathbf{U}_{h,2}
\nonumber \\
=&-(\frac{1}{2}-\lambda_1)\mathbf{U}_{h,1}^TA_1^TB_1\mathbf{U}_{h,1}+
(\frac{1}{2}+\lambda_2)\mathbf{U}_{h,2}^TA_2^TB_2\mathbf{U}_{h,2}
\nonumber \\
-&\mathbf{U}_{h,2}^T(\lambda_1A_2^TB_1+\lambda_2(A_1^TB_2)^T)\mathbf{U}_{h,1}.
\label{system_IT}
\end{align}
Here $\mathbf{U}_{h,1},\mathbf{U}_{h,2}$ denote the numerical solution values at the interface $x_\Gamma$.
We can write
$$IT=- \left(\begin{array}{c}
\mathbf{U}_{h,1} \\
\mathbf{U}_{h,2}
\end{array}\right)^T
\mathbf{S}\left(\begin{array}{c}
\mathbf{U}_{h,1}\\
\mathbf{U}_{h,2}
\end{array}\right),
\textrm{ with }
\mathbf{S}=\left(\begin{array}{cc}
(\frac{1}{2}-\lambda_1)A_1^TB_1&\frac{\lambda_1+\lambda_2}{2}A_1^TB_2\\
\frac{\lambda_1+\lambda_2}{2}A_2^TB_1&-(\frac{1}{2}+\lambda_2)A_2^TB_2
\end{array}\right).
$$
Note that by (\ref{relation}) the matrix $\mathbf{S}$ is symmetric.
To ensure that the interface terms do not cause energy growth, $\mathbf{S}$ needs to be positive semi-definite. The only possible choice for the penalty parameters that also satisfies the conservation condition \eqref{eq:conservcondlamb} is
$$\lambda_1=\frac{1}{2},\quad \lambda_2=-\frac{1}{2},$$
which implies $\mathbf{S}=0$ and correspondingly $IT=0$. This proves the following theorem.

\begin{theorem}\label{thm:stabilityacoustic}
Consider the semi-discrete discontinuous cut finite element method \eqref{scheme:state:DG} for the acoustic system \eqref{System:acoustic:eq} in conservative form, with penalty parameters
$$\lambda_1=\frac{1}{2},\quad \lambda_2=-\frac{1}{2}.$$
The interface treatment is conservative, and does not contribute to temporal
 growth or decay of the energy \eqref{eq:acousticenergyh}.
\end{theorem}

\subsection{Numerical results}
We use the proposed discontinuous CutFEM \eqref{scheme:state:DG} to solve the scalar advection equation \eqref{eq:linear:state} and the acoustic system \eqref{System:acoustic:eq}.  The background mesh on the domain $\Omega$ is uniform with mesh size $h=|\Omega|/N$, where $N$ is the number of elements.  In the numerical simulations, we use $\rM=0.25$ and $\rA=0.75$.

To discretize in time we use the explicit third order TVD Runge-Kutta method \cite{gottlieb2001strong} when the polynomial degree in space $r\leq2$ (see \eqref{eq:BGspace}), i.e.,
\begin{align}
u^{n,1} &=u^{n}+\Delta t  L\left(u^{n},g_h^n,t^{n}\right), \\
u^{n,2} &=\frac{3}{4} u^{n}+\frac{1}{4} u^{n,1}+\frac{1}{4} \Delta t  L\left(u^{n,1},g_h^{n,1},t^{n}+\Delta t  \right), \\
u^{n+1} &=\frac{1}{3} u^{n}+\frac{2}{3} u^{n,2}+\frac{2}{3} \Delta t  L\left(u^{n,2}, g_h^{n,2}, t^{n}+\frac{1}{2} \Delta t \right).
\end{align}
Here $u_t=L(u,g,t)$ is the semi-discrete problem, $\Delta t =t^{n+1}-t^n$ is the time step and $g_h^n, g_h^{n,1},g_h^{n,2}$ denote the approximations of the boundary condition $g(t)$ at different time stages. The inflow information $g$ is imposed via Taylor expansion of $g(t)$ to avoid order reduction with $g_h^n=g(t^n)$, $g_h^{n,1}=g(t^n)+\Delta t  g'(t^n)$ and $g_h^{n,2}=g(t^n)+\frac{\Delta t }{2} g'(t^n) +\frac{(\Delta t )^2}{4}g''(t^n)$,
 for details see~\cite{zhang2011third}. For the approximation of boundary conditions in systems see \cite{ShuDG3}.
When $r=3$ we instead use the fourth order five stages Runge-Kutta method \cite{gottlieb2001strong}. The time step is taken to be $\Delta t =\frac{Ch}{\max_\Omega\{|F'(u_h)|\}}$ with given Courant number $C=0.3,0.2,0.1$, for $r=1,2,3$, respectively.
Here, $\max_\Omega\{|F'(u_h)|\}$ represents the largest absolute value of  the eigenvalues of the Jacobian $ \frac{\partial F( u_h)}{\partial u_h}$ on the domain $\Omega$.
Next, we will demonstrate the accuracy and conservation of the proposed method by solving test problems.  We will measure the error in the following norms
 \begin{align*}
 &||u-u_h||^p_{\Omega}=\sum_{i=1}^2 ||u_i-u_{h,i}||^p_{\Omega_i}, \quad p=1,2,\\
 & ||u-u_h||_{L^\infty(\Omega_1\cup\Omega_2)}= \max\{\max_{x\in \Omega_1}\{|u_1(x,t)-u_h(x,t)|\},\max_{x\in \Omega_2}\{|u_2(x,t)-u_h(x,t)|\}\}.
 \end{align*}
 Here $||\cdot||^p_{\Omega_i}=\int_{\Omega_i} |\cdot|^pdx$ denotes the usual $L^p$-norms in domain $\Omega_i$.
The error in the $L^\infty$-norm is measured as the maximum value of $|u-u_h|$ on the quadrature points of each element, the end points of each elements and the interface point. We note that the quadrature points of the integration in the cut element in each domain $\Omega_i$ are taken over the part of the background element, which is in the domain $\Omega_i$.

\subsubsection{Scalar problem: Accuracy }\label{sec:stationary:scalar:accuracy}
We consider problem \eqref{eq:linear:state} with $x_L=-1$, $x_R=1$, a stationary interface at  $x_\Gamma=10^{-4}$, parameters  $a_1=2,a_2=1$, initial condition
\begin{align}\label{eq:scalar:initial}
f(x)=\left\{\begin{array}{ll}
{\sin(2\pi x),} & {x\in [-1,x_\Gamma],} \\
{2\sin(4\pi (x-x_\Gamma/2)),} & {x\in[x_\Gamma,1],}
\end{array}\right.
\end{align}
and inflow boundary condition $u(x_L,t)=g(t)=\sin(2\pi(-1-2t))$. Note that $u_2(x_\Gamma,0)=2u_1(x_\Gamma,0)$ so the initial condition satisfies the interface condition \eqref{eq:interfacecond}. The exact solution to this problem is
\begin{align}
u(x,t)=
\left\{\begin{array}{ll}
{\sin(2\pi(x-2t)),} & {x\in [-1,x_\Gamma],} \\
{2\sin(4\pi(x-t-x_\Gamma/2)),} & {x\in[x_\Gamma,1].}
\end{array}\right.
\end{align}

We solve  up to time $t=1$ with different mesh sizes and polynomial spaces $r=1, 2, 3$ (i.e., $P^1,P^2,P^3$).
We choose the penalty parameters to be $\lambda_1=0.1$ and $\lambda_2=\lambda_1-1$.
Errors in the $L^2$- and $L^\infty$-norm and the corresponding convergence orders using the proposed  method are shown in the left part of Table \ref{table:interface:accuracy}. For comparison, we also show results using the standard DG method in the right part of Table \ref{table:interface:accuracy}.  For the standard DG method we use the same numerical fluxes but generate the mesh so that the interface $x_\Gamma$ is located on an element edge. We do this by using uniform meshes with $N=N_1+N_2$ elements and mesh size $h_1=|\Omega_1|/N_1,h_2=|\Omega_2|/N_2$ for $\Omega_1$ and $\Omega_2$, respectively. Here $N_1,N_2$ are chosen such that $h_1,h_2$ are close to the mesh size $h=|\Omega|/N$  we use in the CutFEM.
 From the numerical results in Table \ref{table:interface:accuracy}, we observe that the proposed method and the standard DG method have optimal order of accuracy and that the magnitude of the errors in both $L^2$- and $L^\infty$-norm are similar for the two methods.  We have also tested other choices for the parameters $a_1$ and $a_2$ and observed similar results as shown here.
\begin{table}[!htbp]
\caption{\label{table:interface:accuracy} {
Errors and orders of accuracy  at $t=1$ for the problem in Sect.  \ref{sec:stationary:scalar:accuracy} solved by the
 proposed method on a uniform background mesh and by a standard DG method on a quasi-uniform mesh fitted to the interface. Polynomial degrees $1,2,3$ and $N$ elements are used.
} }
\begin{small}
\begin{tabular}{c|cccc|cccc}
\noalign{\smallskip}\hline
 N&$L^2 $ error & order &$L^{\infty} $ error & order &$L^2 $ error & order &$L^{\infty} $ error & order \\\hline
&\multicolumn{4}{c|}{Discontinuous CutFEM method}&\multicolumn{4}{c}{Standard DG method}\\\hline
20	&	2,64E-01	&	-	&	5,60E-01	&	-	&	2,64E-01	&	-	&	5,61E-01	&	-	\\
40	&	4,92E-02	&	2,42	&	1,13E-01	&	2,31	&	4,92E-02	&	2,42	&	1,13E-01	&	2,32	\\
80	&	9,74E-03	&	2,34	&	3,07E-02	&	1,88	&	9,75E-03	&	2,34	&	3,07E-02	&	1,87	\\
160	&	2,22E-03	&	2,13	&	7,96E-03	&	1,95	&	2,22E-03	&	2,13	&	7,96E-03	&	1,95	\\
320	&	5,40E-04	&	2,04	&	2,02E-03	&	1,98	&	5,40E-04	&	2,04	&	2,02E-03	&	1,98	\\
\noalign{\smallskip}\hline\noalign{\smallskip}
20	&	1,21E-02	&	-	&	6,07E-02	&	-	&	1,21E-02	&	-	&	6,07E-02	&	-	\\
40	&	1,35E-03	&	3,16	&	7,78E-03	&	2,97	&	1,35E-03	&	3,16	&	7,78E-03	&	2,96	\\
80	&	1,66E-04	&	3,03	&	1,01E-03	&	2,94	&	1,66E-04	&	3,03	&	1,01E-03	&	2,94	\\
160	&	2,06E-05	&	3,01	&	1,29E-04	&	2,98	&	2,06E-05	&	3,01	&	1,29E-04	&	2,98	\\
320	&	2,58E-06	&	3,00	&	1,61E-05	&	2,99	&	2,58E-06	&	3,00	&	1,61E-05	&	2,99	\\
\noalign{\smallskip}\hline\noalign{\smallskip}
20	&	7,14E-04	&	-	&	5,65E-03	&	-	&	7,14E-04	&	-	&	5,65E-03	&	-	\\
40	&	4,41E-05	&	4,02	&	3,68E-04	&	3,94	&	4,41E-05	&	4,02	&	3,68E-04	&	3,94	\\
80	&	2,75E-06	&	4,00	&	2,29E-05	&	4,00	&	2,75E-06	&	4,00	&	2,29E-05	&	4,00	\\
160	&	1,72E-07	&	4,00	&	1,45E-06	&	3,99	&	1,72E-07	&	4,00	&	1,44E-06	&	3,99	\\
320	&	1,07E-08	&	4,00	&	9,05E-08	&	4,00	&	1,07E-08	&	4,00	&	9,05E-08	&	4,00	\\
\noalign{\smallskip}\hline
\end{tabular}
\end{small}
\end{table}
\subsubsection{Scalar problem: Conservation}\label{sec:stationary:scalar:conservation}
With this example we test how well quantities are conserved,  and how errors and the condition number of the mass matrix depend on the cut size. This example is used in \cite{la2016well} and is similar to Example 1.1, but with less smooth initial data. The initial condition and inflow boundary conditions are $f(x)=0$ and \(u(x_L,t)=g(t)=\sin (4 \pi(-1+3 t))\), respectively. The inflow condition is weakly imposed by the upwind flux information at the inflow boundary.  The exact solution is
\begin{align}
u(x,t)=
\left\{\begin{array}{ll}
{g(t-(x-x_L)/2),} & { t\geq\frac{x-x_L}{2}, \ x\in [-1,x_\Gamma],} \\
{2g(t-x+(x_\Gamma+x_L)/2),} & { t\geq x-\frac{x_\Gamma+x_L}{2},\ x\in [x_\Gamma,1],}\\
{0,} & {else.} \\
\end{array}\right.
\end{align}

We use the proposed  method with $r=2$, i.e. quadratic polynomials and  a uniform background mesh.
The interface is at $x_\Gamma=10^{-4}$. In  Fig. \ref{exe:scalar:solution}, we show the numerical solution   on the background mesh consisting of $400$ elements and the exact solution  at time $t=0.5$ and
 $t=1$.    Our results compare well with those in \cite{la2016well}.

In Fig. \ref{exe:scalar:conservation}, we show the conservation error
\begin{align}
e(t)=&\sum_{n=0}^{N_t-1}\frac{\Delta t }{6}\Big(a_1g_h^{n}-a_2u_h^n(x_R)+4\left(a_1g_h^{n,2}-a_2u_h^{n,2}(x_R)\right)\notag\\
&\quad \quad +a_1g_h^{n,1}-a_2u_h^{n,1}(x_R)\Big)-\ino (u_h^{N_t}(x)-u_h^0(x))dx.
\label{eq:em2}
\end{align}
Here, $N_t$ is the number of time steps from time $0$ to $t$.   $u_h^n, u_h^{n,1}, u_h^{n,2}$ are the approximations of the solution at  time $t^n,t^{n}+\Delta t ,t^n+\Delta t /2$, respectively. The inflow information introduced in each time step is equal  to  $\frac{\Delta t }{6}\left(g_h^{n}+4g_h^{n,2}+g_h^{n,1}\right)$, which is an approximation of the integral $\int_{t^n}^{t^{n+1}}g(t)dt$. The error $e(t)$ takes into account the approximation of the inflow boundary condition.  If instead the exact integral at the  inflow boundary is used there would be a
contribution to the conserved quantity of the  order $\Delta t^4$ corresponding to the integration error.
The conservation error $e$ (Fig. \ref{exe:scalar:conservation}, left panel) with $\lambda_1,\lambda_2$ satisfying condition \eqref{eq:conservcondlamb}, is of the order of machine epsilon and remains on the same level also for other refinements and for other polynomial degrees than $r=2$.  This shows that the interface treatment of the proposed method is conservative. We also show the conservation error when the parameters $\lambda_1=0.25$, $\lambda_2=-0.25$, which  do not satisfy the conservation condition \eqref{eq:conservcondlamb}, see the right panel of Fig. \ref{exe:scalar:conservation}. We observe that the scheme in this case is not conservative and has a large error,  which decrease with mesh refinement but still is large compared to machine error.
\begin{figure}[!bhtp]
\begin{center}
\includegraphics[width=2.3in]{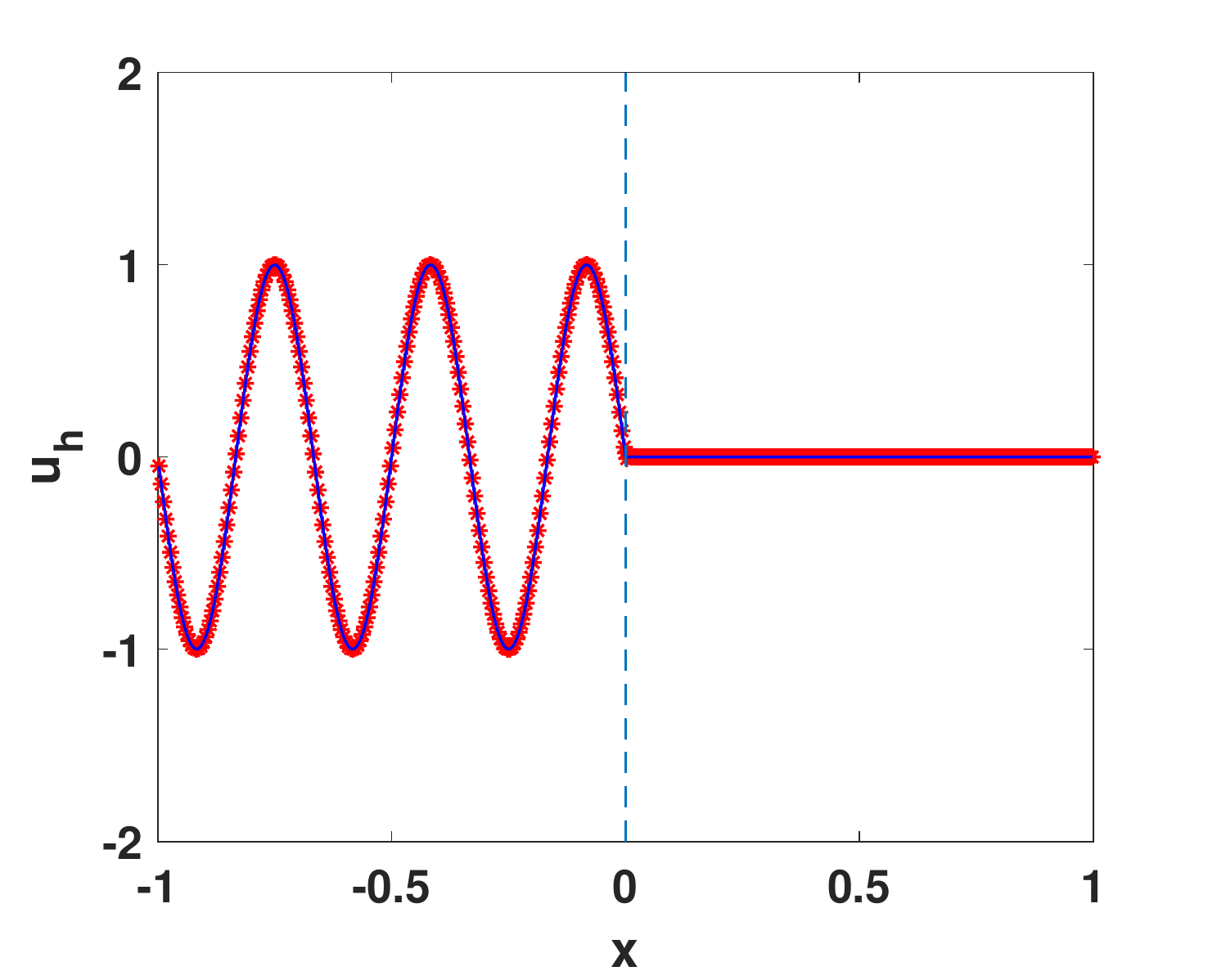}
\includegraphics[width=2.3in]{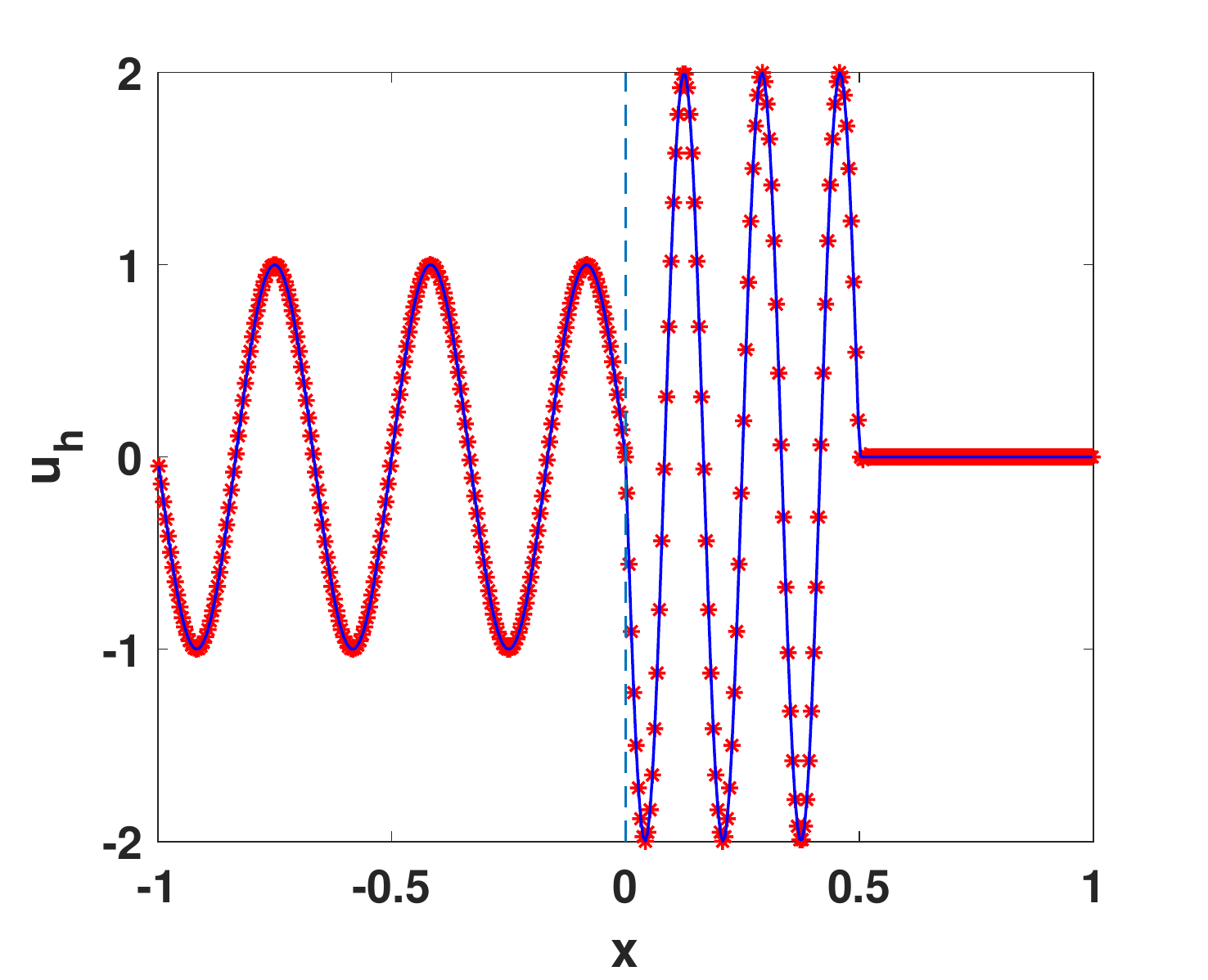}
\caption{
Solutions to the problem in \ref{sec:stationary:scalar:conservation} with zero initial data using discontinuous  quadratic polynomials in space on a uniform mesh with 400 elements.
Left: $t=0.5$. Right: $t=1$. Star: numerical solution. Solid line: exact solution. The dashed line indicates the interface's position.}\label{exe:scalar:solution}
\end{center}
\end{figure}
\begin{figure}[!bhtp]
\begin{center}
\includegraphics[width=2.3in]{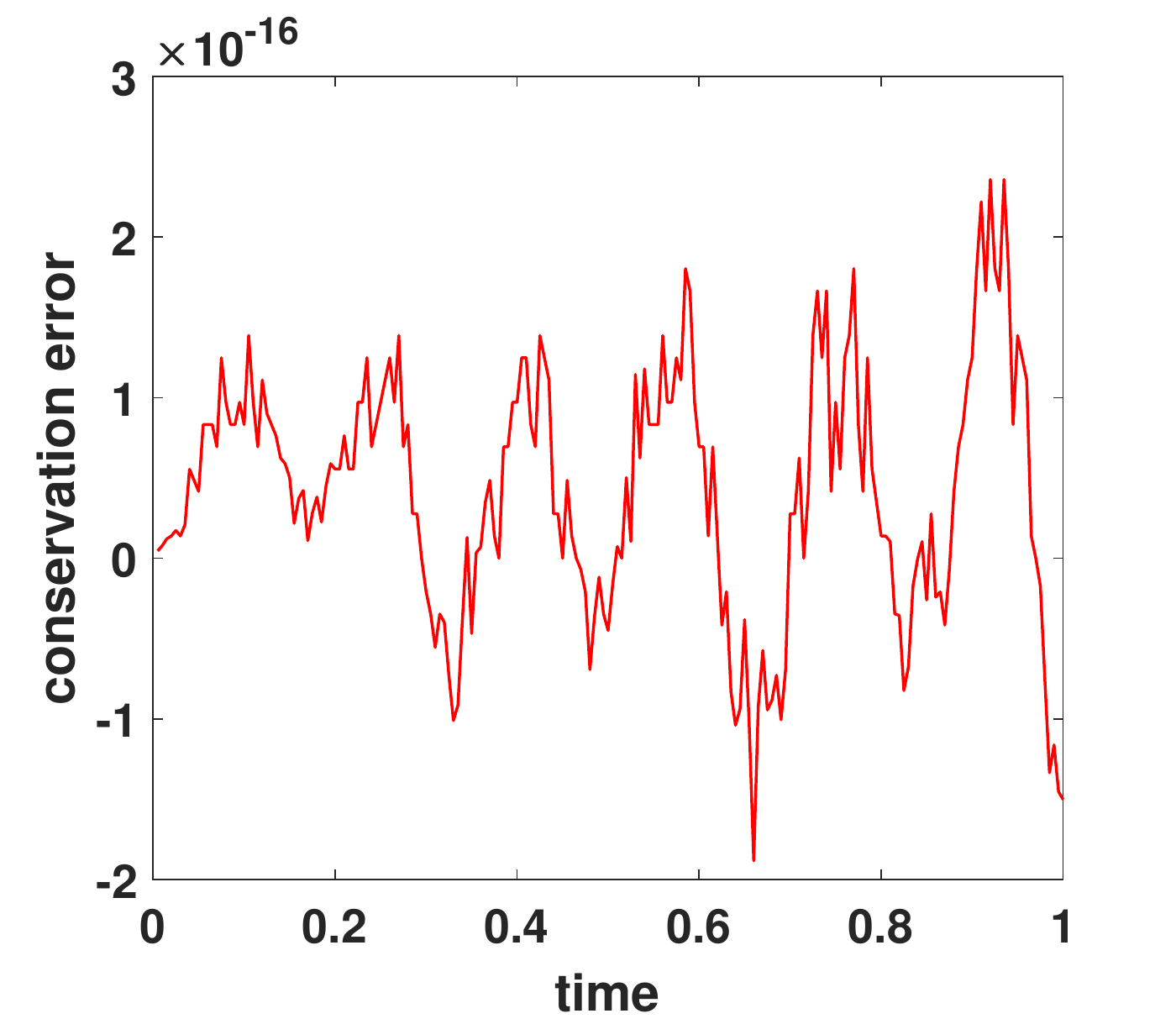}
\includegraphics[width=2.3in]{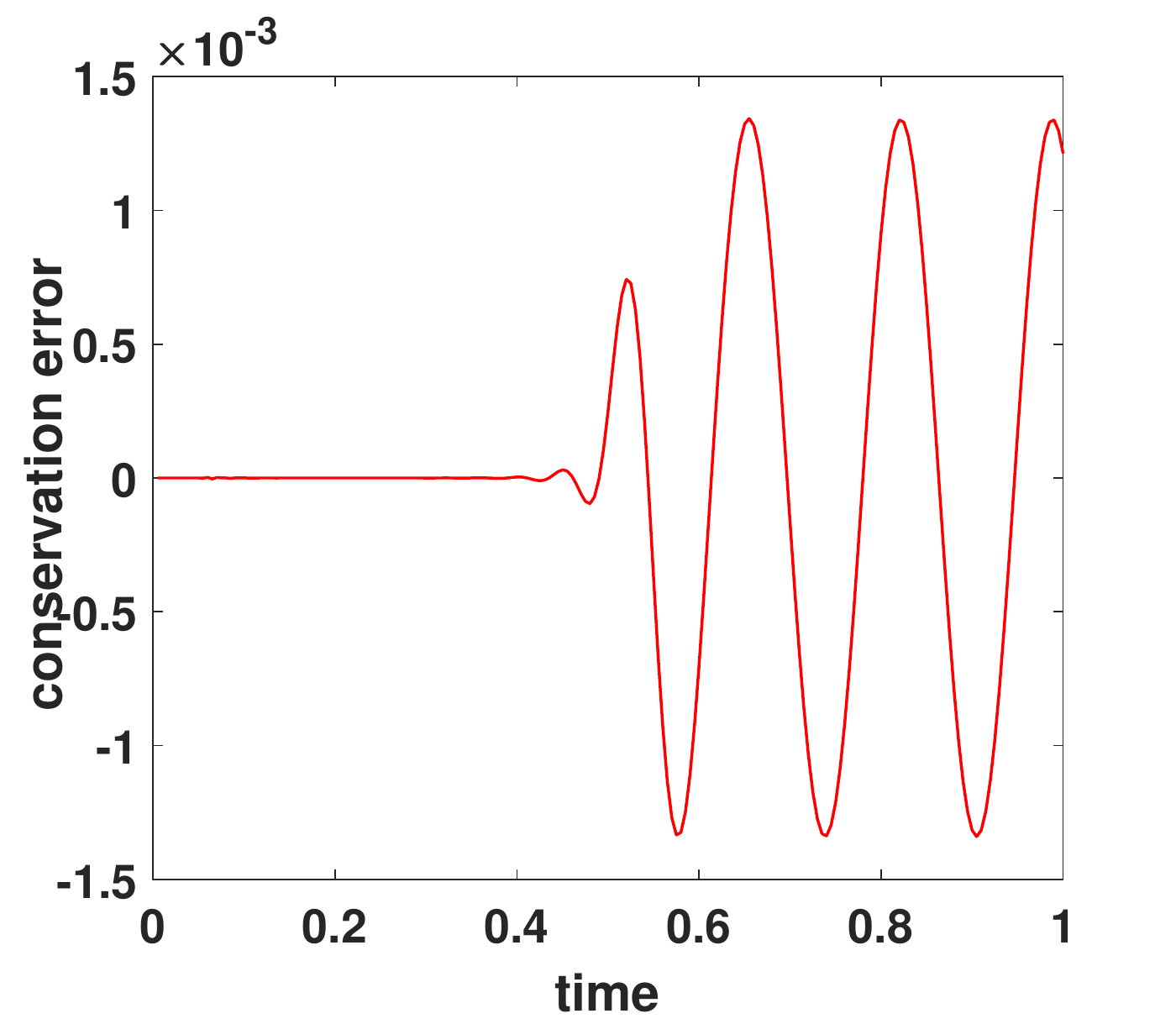}
\caption{
Conservation errors for the example in \ref{sec:stationary:scalar:conservation}. A uniform mesh with 40 elements and piecewise quadratic polynomials are used.
Left: $\lambda_1=0.1,\lambda_2=1-\lambda_1$. Right: $\lambda_1=0.25,\lambda_2=-0.25$.
}\label{exe:scalar:conservation}
\end{center}
\end{figure}

Next we investigate how the error depends on the position of the interface relative the background mesh.  We use  linear polynomials in space, $h=1/200$  and  $x_\Gamma=\alpha h$, with $\alpha$ varying between 0 and 1. Note that there is a mesh node at $x=0 $ so that $\alpha$ is the relative cut size.  In Fig. \ref{exe:interface:erroralpha}, we plot the $L^1-,L^2$-and $L^\infty$-errors at t=1 as a function of $\alpha$.
We have scaled the $L^\infty$-error, by dividing the error with $10$, to show all errors in one figure. We observe that the $L^1-$ and $L^2-$errors are independent of how the interface cuts the background mesh and the $L^\infty$-error does not change much either.
 \begin{figure}[!bhtp]
 \centering
 {\includegraphics[width=2.8in]{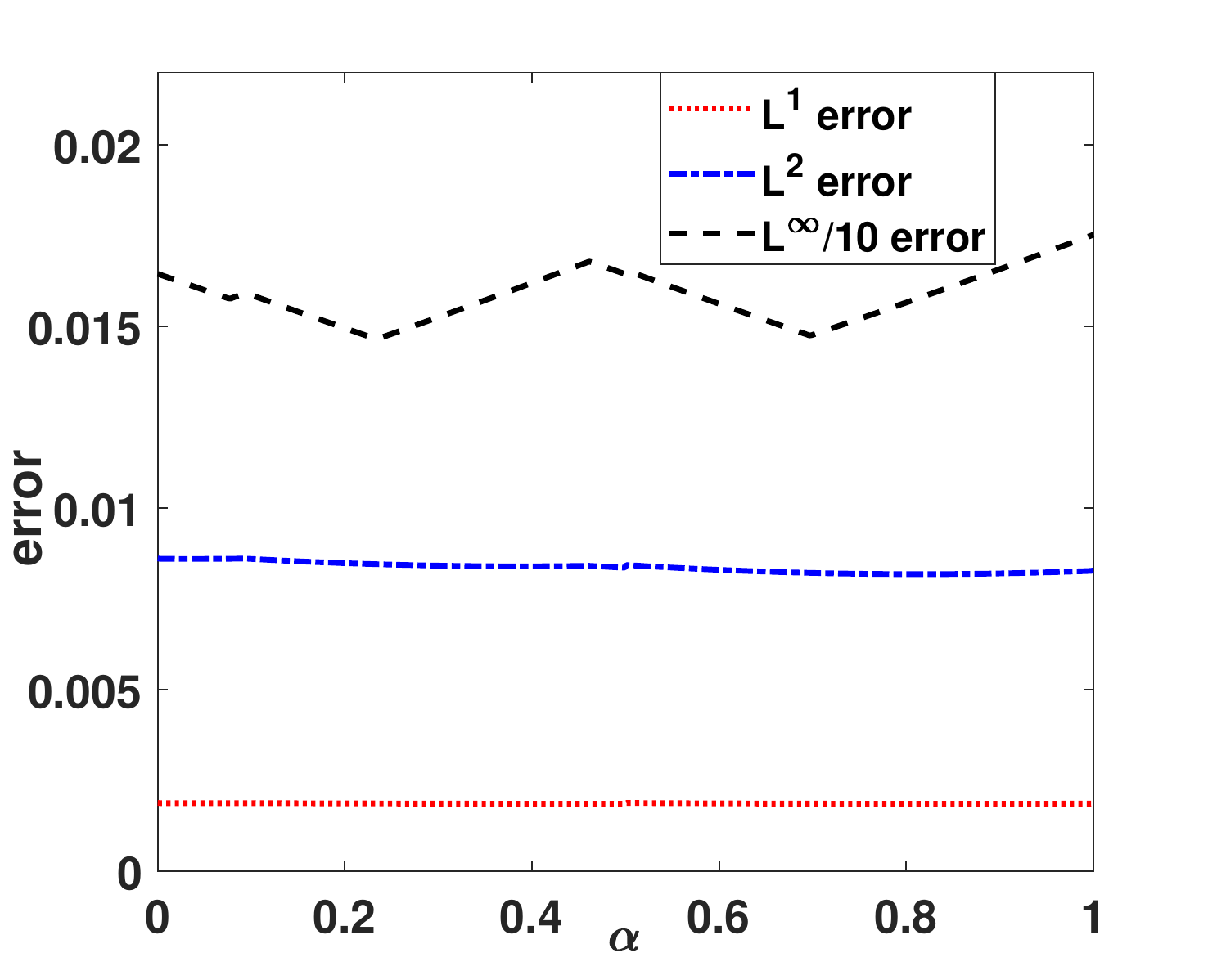}}
\caption{Errors in the numerical solution of the problem in \ref{sec:stationary:scalar:conservation} at $t$=1 as a function of the relative cut size.
A uniform mesh with 400 elements and piecewise linear polynomials are used.
The  interface is at  $x_\Gamma=\alpha h$, and we use 400 equally distributed $\alpha$'s in $[0,1]$.
} \label{exe:interface:erroralpha}
\end{figure}

We have also checked the conditioning of the mass matrix in the same setting as above, for several element types.  The condition  numbers as a function of the relative cut size for piecewise  linear, quadratic and cubic polynomials are shown in  Fig. \ref{exe:interface:condalpha}. We see that the stabilization  controls the condition number so that it stays bounded and on the same level, independently of how the interface cuts the background mesh.
  \begin{figure}[!bhtp]
  \centering
  \subfigure[$P^1$]{
 \includegraphics[width=1.5in]{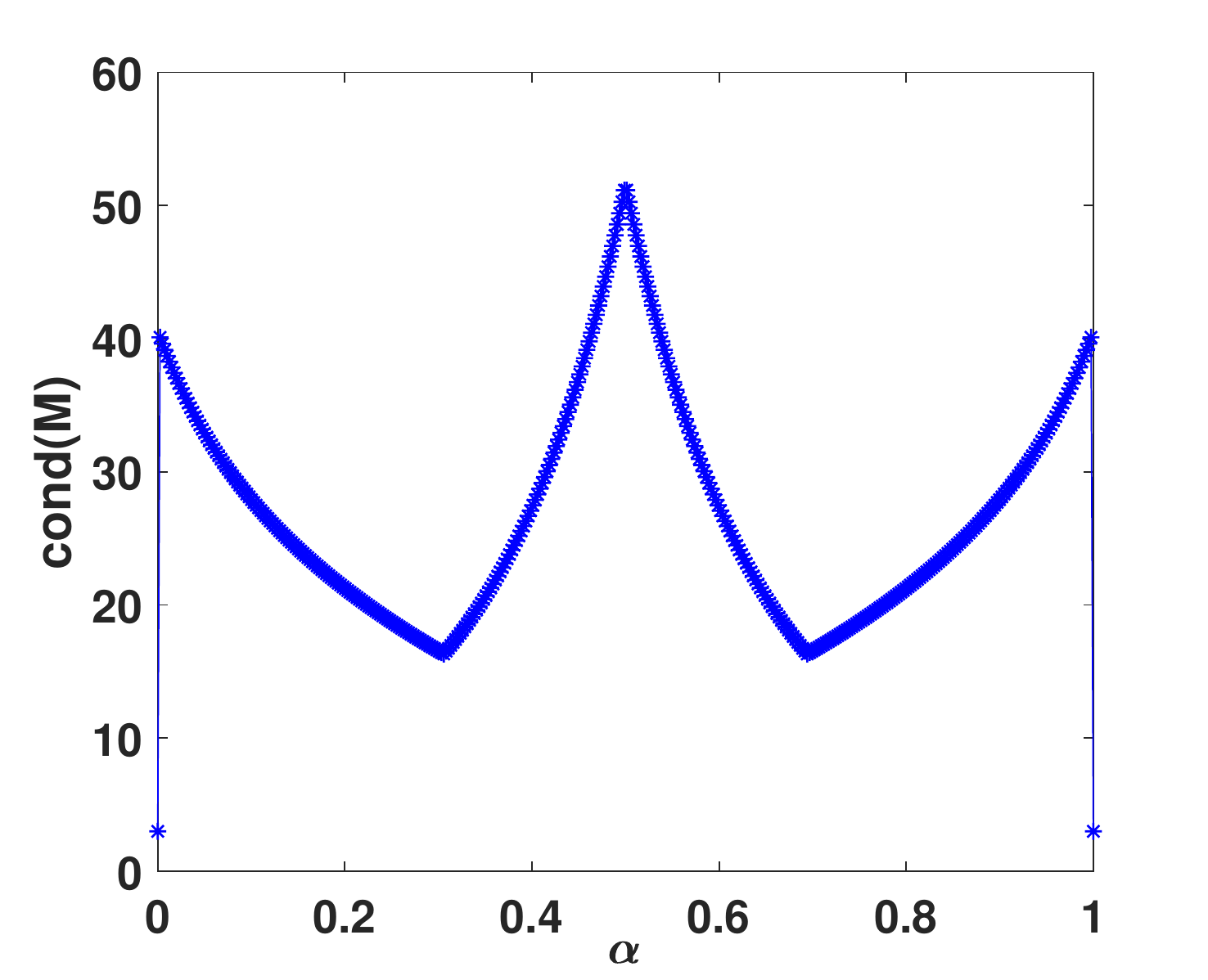}}
  \subfigure[ $P^2$]{
 \includegraphics[width=1.5in]{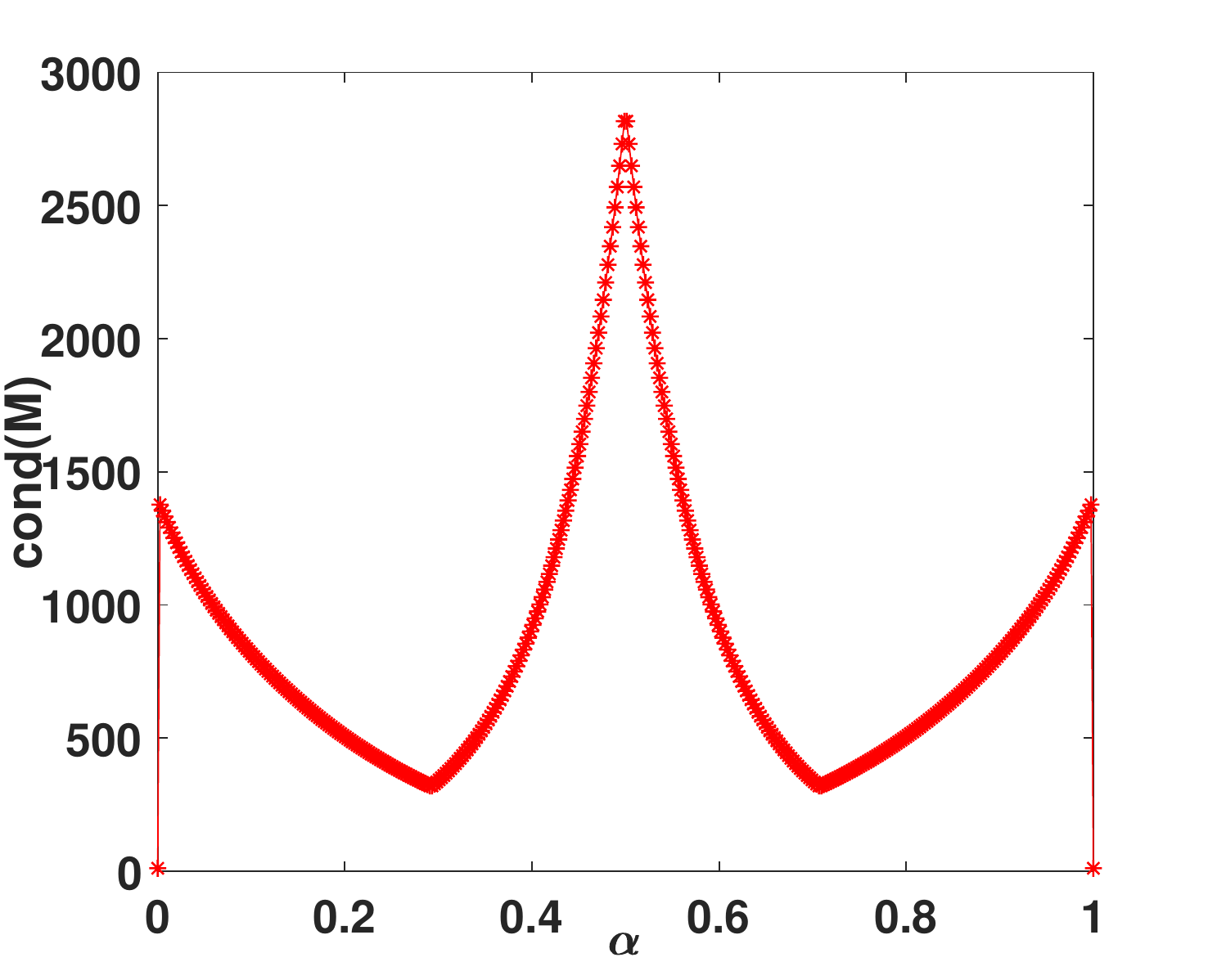}}
  \subfigure[ $P^3$]{
 \includegraphics[width=1.5in]{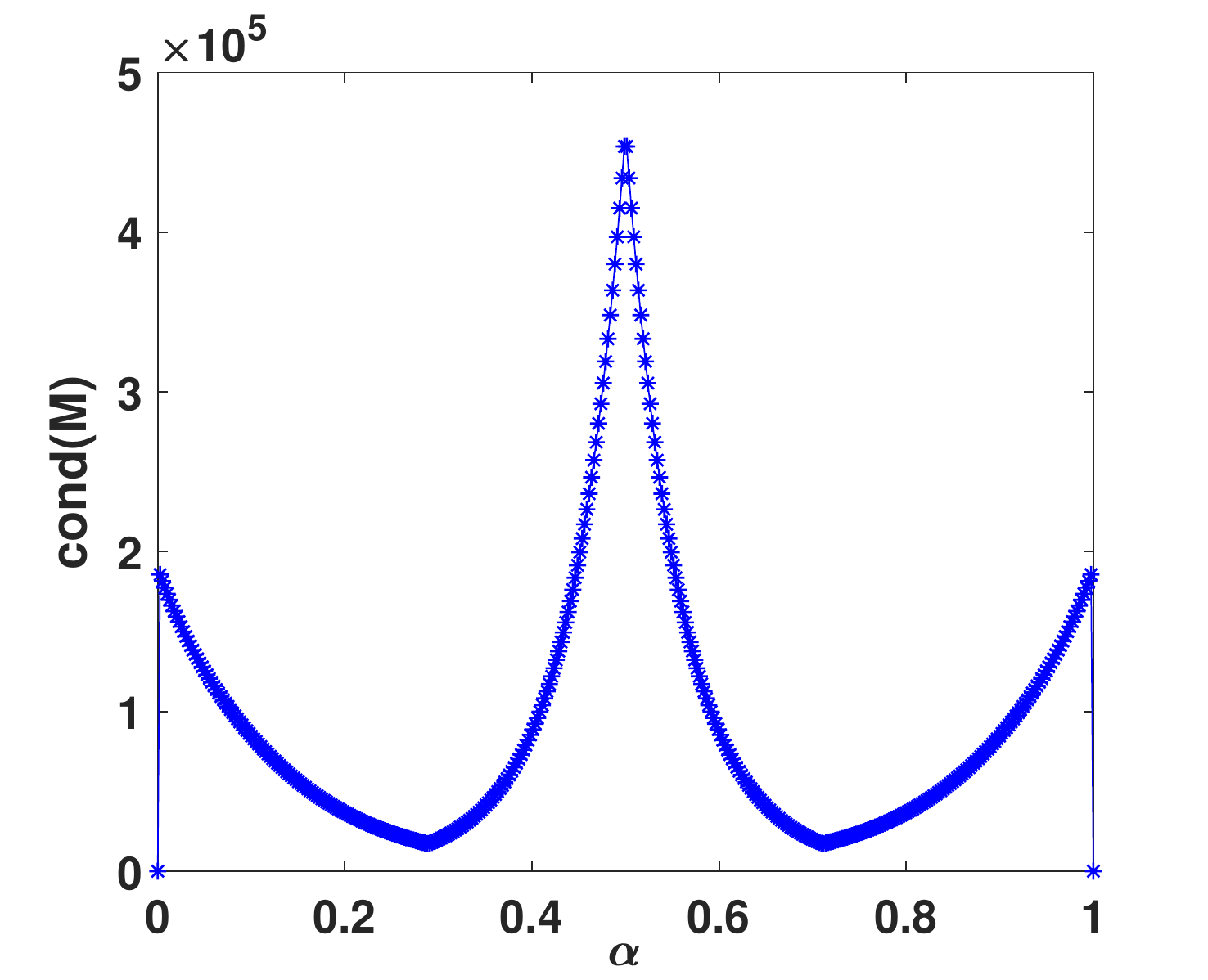}}
\caption{The condition number of the mass matrix as a function of the relative cut size for the problem in \ref{sec:stationary:scalar:conservation} at $t=1$. A uniform mesh with 400 elements and $P^1,P^2,P^3$ polynomials are used. The interface is at  $x_\Gamma=\alpha h$, and we have considered 400 equally distributed $\alpha$'s in $[0,1]$.
}
  \label{exe:interface:condalpha}
\end{figure}

\subsubsection{The acoustic system}\label{sec:stationary:acoustic}
We now use the proposed CutFEM method \eqref{scheme:state:DG} with penalty parameters $\lambda_1=0.5$ and $\lambda_2=-0.5$ to solve the acoustic system in conservation form \eqref{System:acoustic:eq}.  We consider the same example as in~\cite{piraux2001new}. The domain is $\Omega=[0,300]$, a long fluid medium with an interface at $x_\Gamma=96.3$, and with physical parameters
$$
(\rho(x), c(x))=\left\{\begin{array}{ll}
\rho_{1}=1000 \mathrm{kg} / \mathrm{m}^{3}, & c_{1}=1500 \mathrm{m} / \mathrm{s}, \,  \text { if } x \leq x_\Gamma, \\
\rho_{2}=1200 \mathrm{kg} / \mathrm{m}^{3}, & c_{2}=2800 \mathrm{m} / \mathrm{s},  \text { if } x \geq x_\Gamma.
\end{array}\right.
$$
The initial condition is
$$
U(x,0)=\mathbf{f}_{0}(x)=-f(\xi)\left(\begin{array}{c}
\frac{1}{c_{0}} \\
\rho_{0}
\end{array}\right).
$$
Here, $f_{0}(\xi)$ is a spatially bounded sinusoidal function
$$
f_{0}(\xi)=\left\{\begin{array}{ll}
\sin \left(\omega_{c} \xi\right)-\frac{21}{32} \sin \left(2 \omega_{c} \xi\right)+\frac{63}{768} \sin \left(4 \omega_{c} \xi\right)-\frac{1}{512} \sin \left(8 \omega_{c} \xi\right), & \text { if } 0<\xi<\frac{1}{f_{c}}, \\
0 \text { else, with } \xi=t_{0}-\frac{x}{c},
\end{array}\right.
$$
where the central frequency $f_{c}=50 \mathrm{Hz}$, $\omega_c=2\pi f_c$, and $t_{0}=51 \mathrm{ms}$.
  The initial values of the conservative variables $m$ and $q$ are shown to the left in Fig. \ref{exe:system:solution}.  When the wave reaches the interface, the acoustic wave is transmitted and reflected.

We simulate this problem up to time $t=39 \mathrm{ms}$ using different mesh sizes and polynomial spaces. We note that the waves do not reach the boundaries up to time $t=39 \mathrm{ms}$ thus the zero boundary condition is used in our implementation. In the computation the conservative variables are used, but in
 Table \ref{table:system:accuracy}, we give the $L^2$-and $L^\infty$-errors and the corresponding order of accuracy  for the primitive variables.
We observe that the proposed method has optimal order of accuracy also for the acoustic problem.
In Fig. \ref{exe:system:solution}, we show the initial values and the numerical solution at time $t=39ms$.
We see that our scheme can simulate this problem very well and capture the reflected wave and the transmitted wave. In Fig. \ref{exe:system:conservation}, we plot the conservation errors of $m_h$ and $q_h$ with respect to time $t$.
The conservation errors have small oscillations after the wave arrives at the interface. But all errors are of the order of machine epsilon  which demonstrates that the proposed method is conservative also for the acoustic problem. 

\begin{table}[!bhtp]
\caption{\label{table:system:accuracy} {Errors  and orders of accuracy  at $t=39$ms for the acoustic problem in \ref{sec:stationary:acoustic}.
Uniform background meshes with $N$ elements and piecewise polynomials of orders 1,2,3 (top, middle, bottom) are used.}
 }
\begin{small}
\begin{tabular}{c|cccc|cccc}
\noalign{\smallskip}\hline
N&$L^2 $ error & order &$L^{\infty} $ error & order &$L^2 $ error & order &$L^{\infty} $ error & order \\
\noalign{\smallskip}\hline
&\multicolumn{4}{c|}{$p_h$}&\multicolumn{4}{c}{$u_h$}\\
\noalign{\smallskip}\hline\noalign{\smallskip}
200	&	2,57E+01	&	1,99	&	1,26E+02	&	1,90	&	8,68E-06	&	1,98	&	3,76E-05	&	1,90	\\
400	&	4,33E+00	&	2,57	&	2,20E+01	&	2,52	&	1,49E-06	&	2,54	&	6,55E-06	&	2,52	\\
800	&	6,00E-01	&	2,85	&	3,40E+00	&	2,69	&	2,13E-07	&	2,81	&	1,06E-06	&	2,62	\\
1600	&	8,71E-02	&	2,78	&	5,08E-01	&	2,74	&	3,38E-08	&	2,66	&	1,96E-07	&	2,44	\\
3200	&	1,59E-02	&	2,46	&	9,03E-02	&	2,49	&	6,83E-09	&	2,31	&	6,02E-08	&	1,70	\\
6400	&	3,60E-03	&	2,14	&	2,48E-02	&	1,87	&	1,63E-09	&	2,07	&	1,65E-08	&	1,87	\\
\noalign{\smallskip}\hline\noalign{\smallskip}
200	&	8,29E-01	&	3,74	&	4,11E+00	&	3,52	&	2,91E-07	&	3,70	&	1,35E-06	&	3,38	\\
400	&	5,73E-02	&	3,86	&	2,60E-01	&	3,99	&	2,21E-08	&	3,72	&	1,61E-07	&	3,06	\\
800	&	5,00E-03	&	3,52	&	2,94E-02	&	3,14	&	2,16E-09	&	3,35	&	1,96E-08	&	3,04	\\
1600	&	5,80E-04	&	3,11	&	3,67E-03	&	3,00	&	2,60E-10	&	3,06	&	2,45E-09	&	3,00	\\
3200	&	7,12E-05	&	3,02	&	4,60E-04	&	3,00	&	3,22E-11	&	3,01	&	3,07E-10	&	3,00	\\
6400	&	8,87E-06	&	3,01	&	5,75E-05	&	3,00	&	4,02E-12	&	3,00	&	3,84E-11	&	3,00	\\
\noalign{\smallskip}\hline\noalign{\smallskip}
200	&	2,61E-02	&	5,11	&	1,58E-01	&	4,68	&	1,14E-08	&	4,82	&	1,05E-07	&	3,68	\\
400	&	8,49E-04	&	4,94	&	9,82E-03	&	4,01	&	5,15E-10	&	4,47	&	6,55E-09	&	4,01	\\
800	&	5,03E-05	&	4,08	&	6,07E-04	&	4,02	&	3,16E-11	&	4,03	&	4,05E-10	&	4,02	\\
1600	&	3,14E-06	&	4,00	&	3,80E-05	&	4,00	&	1,97E-12	&	4,00	&	2,53E-11	&	4,00	\\
3200	&	1,97E-07	&	4,00	&	2,37E-06	&	4,00	&	1,23E-13	&	4,00	&	1,58E-12	&	4,00	\\
\noalign{\smallskip}\hline
\end{tabular}
\end{small}
\end{table}

\begin{figure}[!bhtp]
\begin{center}
\subfigure[Initial value: $m$]{
\includegraphics[width=2.2in]
{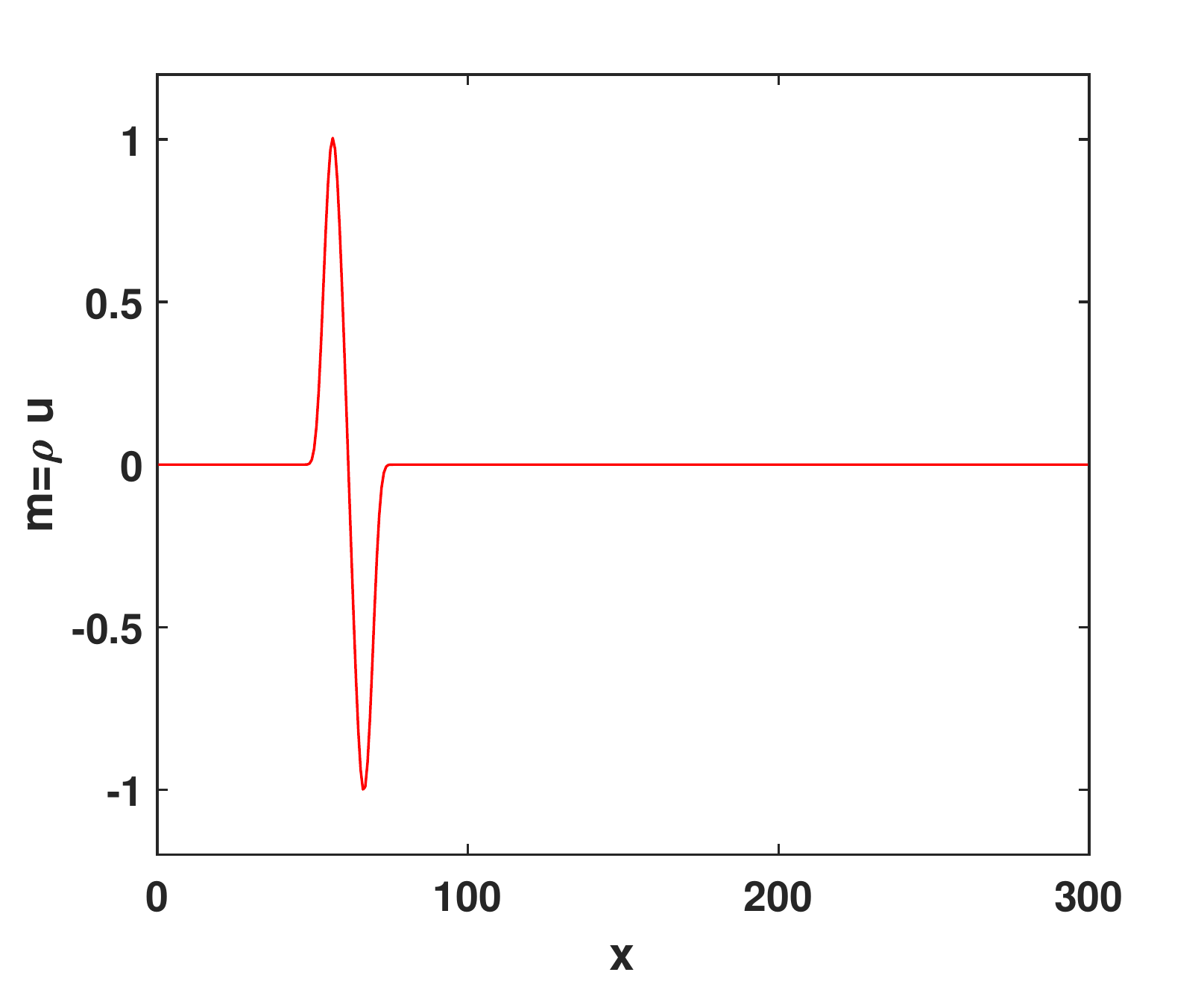}}
\subfigure[Solution: $m$, $m_h$ at $t=39\mathrm{ms}$]{
\includegraphics[width=2.15in]
{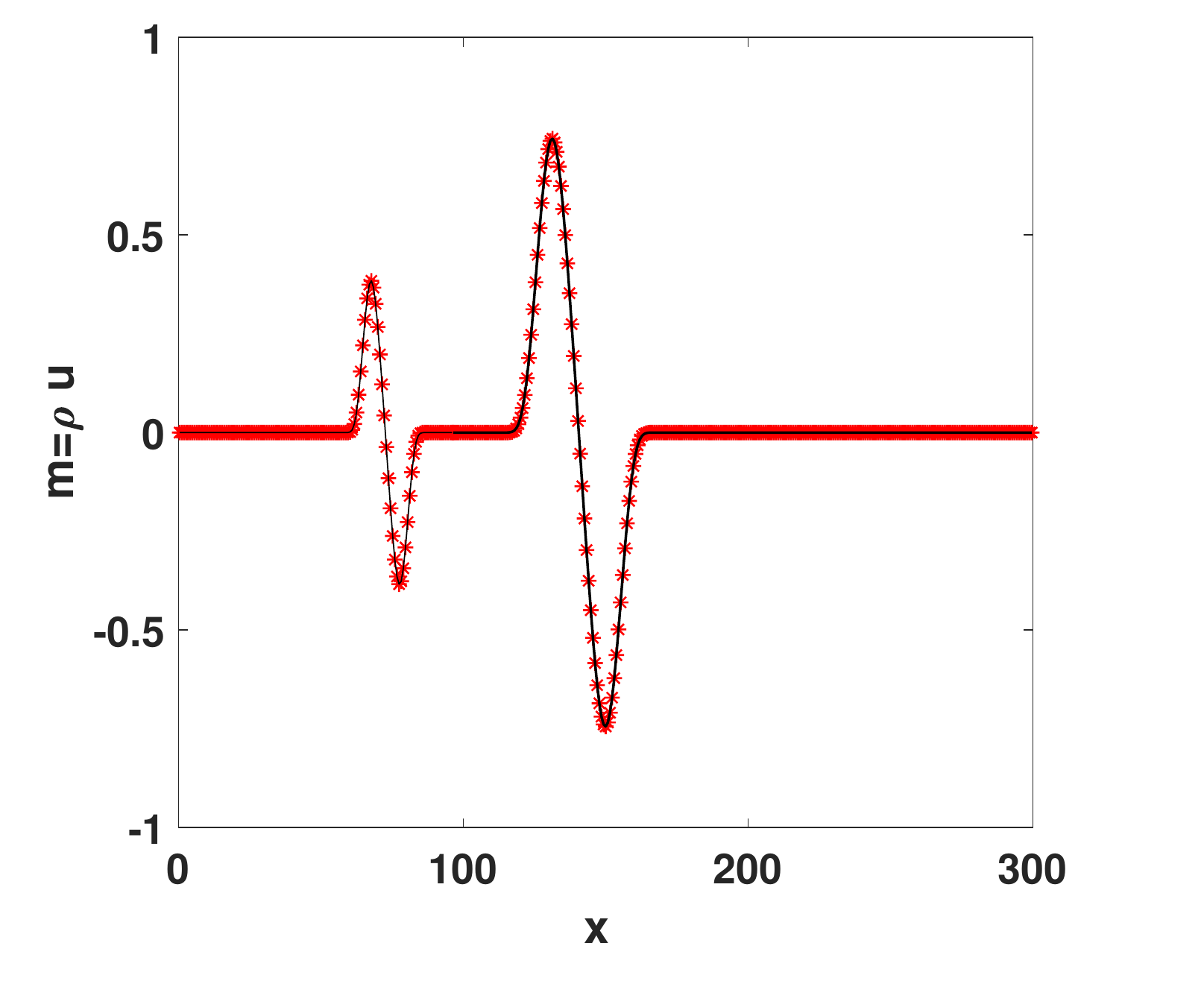}}
\subfigure[Initial value: $q$]{
\includegraphics[width=2.2in]
{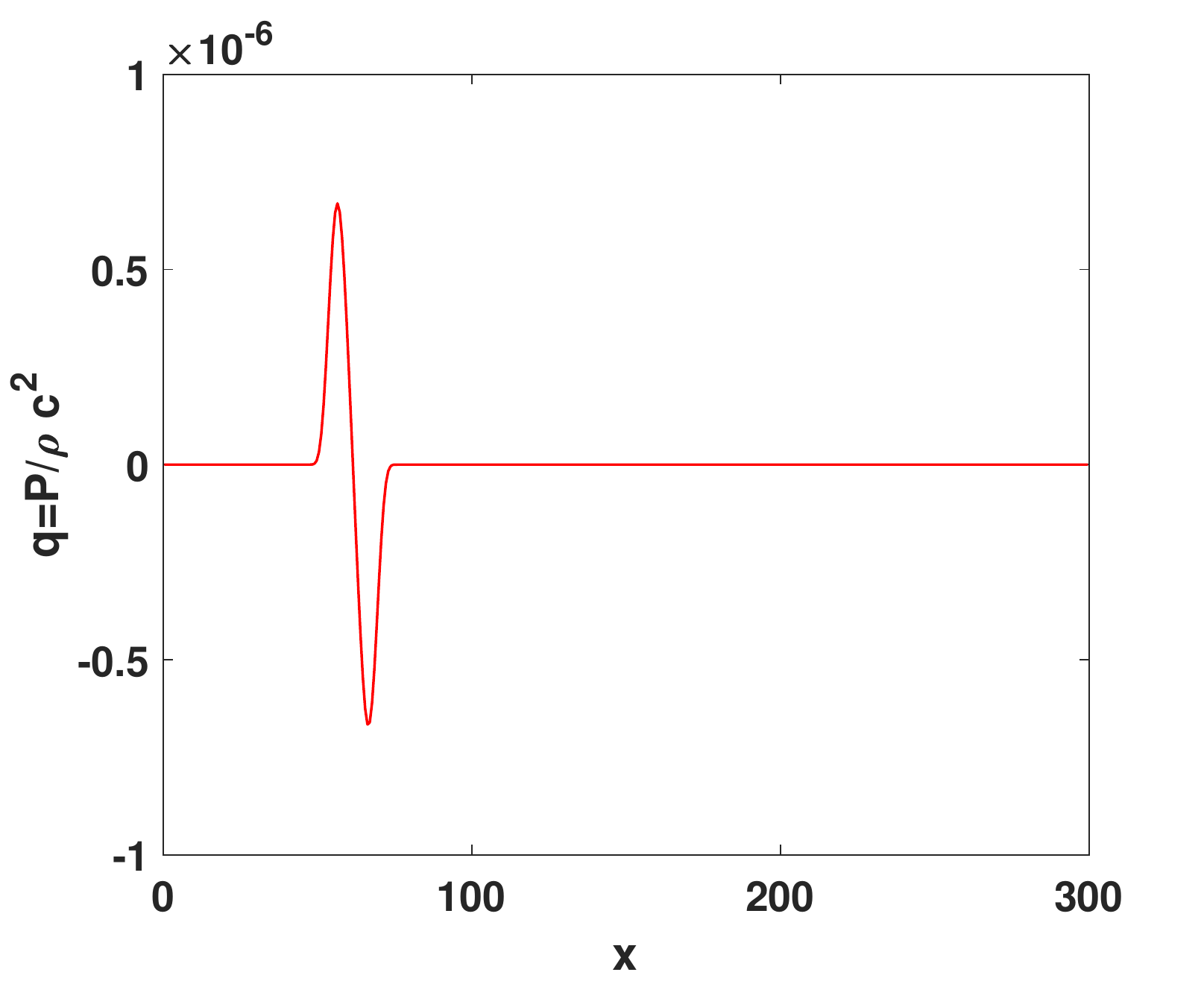}}
\subfigure[Solution: $q$, $q_h$ at $t=39 \mathrm{ms}$]{
\includegraphics[width=2.2in]
{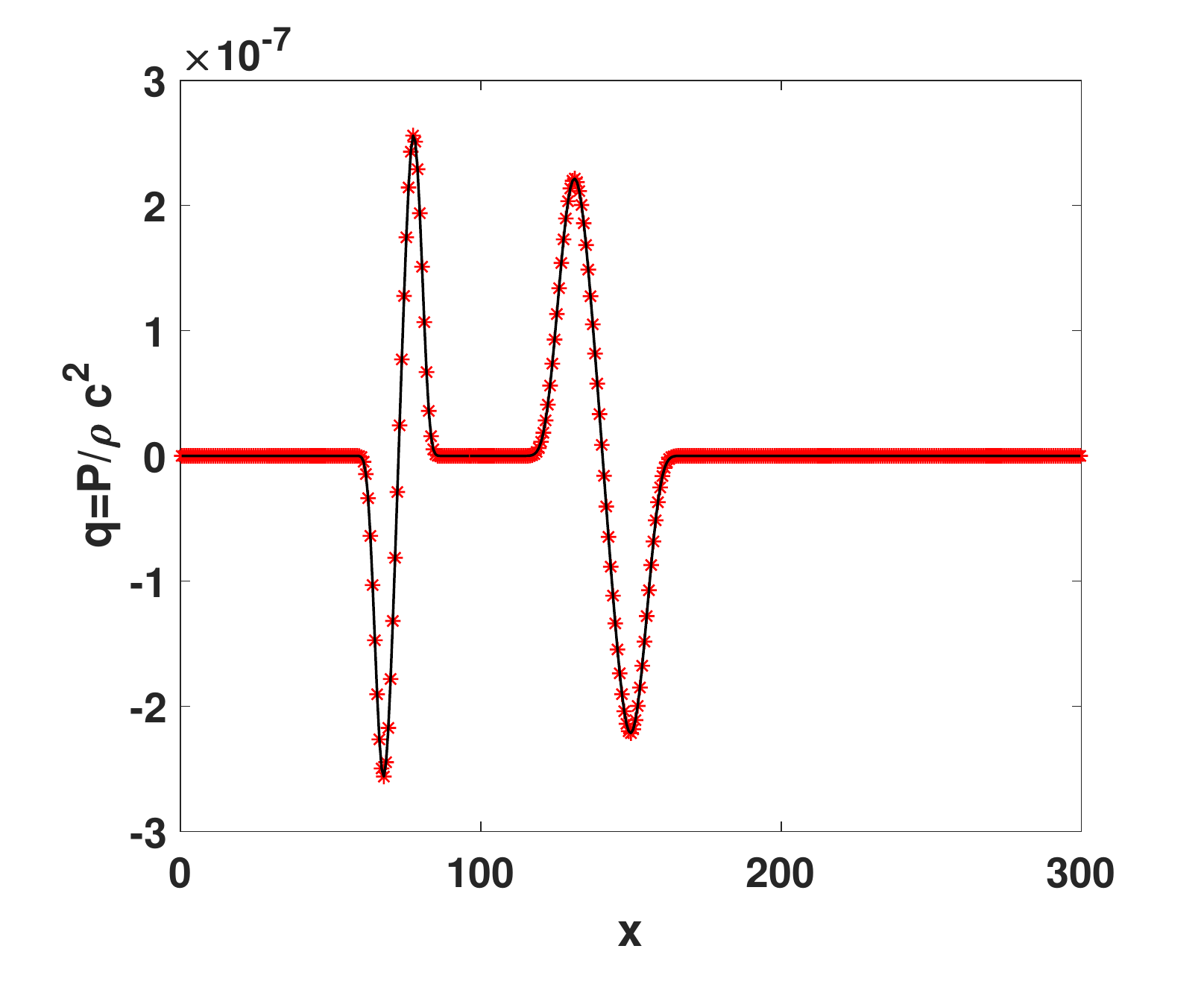}}
\caption{Initial data and solution at $t=39$ms  for the acoustic problem in \ref{sec:stationary:acoustic}.  A uniform background mesh with 400 elements and piecewise quadratic polynomials are used.
Solid line: exact solution. Symbols: numerical solution.
}
\label{exe:system:solution}
\end{center}
\end{figure}

\begin{figure}[!bhtp]
\begin{center}
\includegraphics[width=2.2in]
{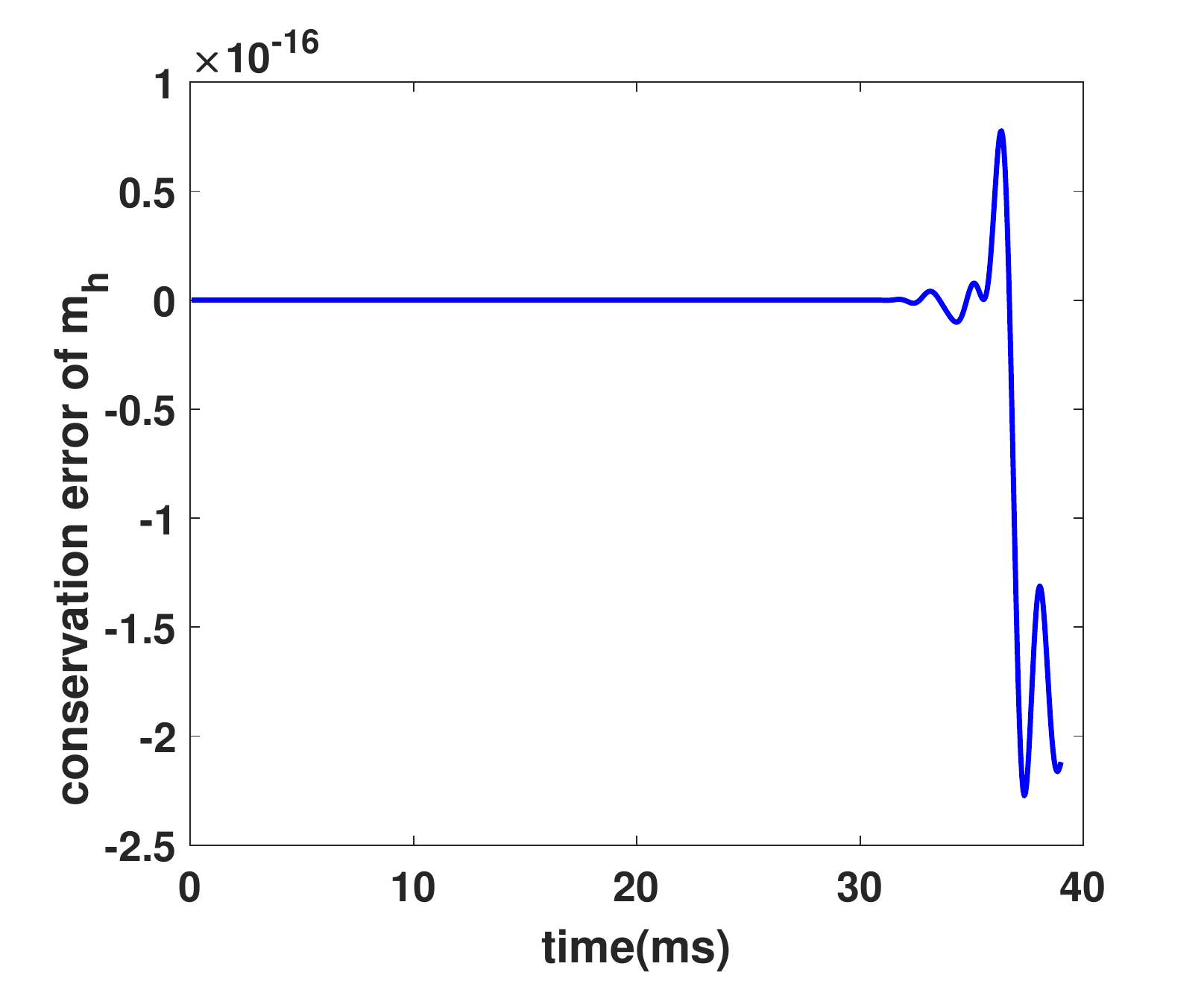}
\includegraphics[width=2.2in]
{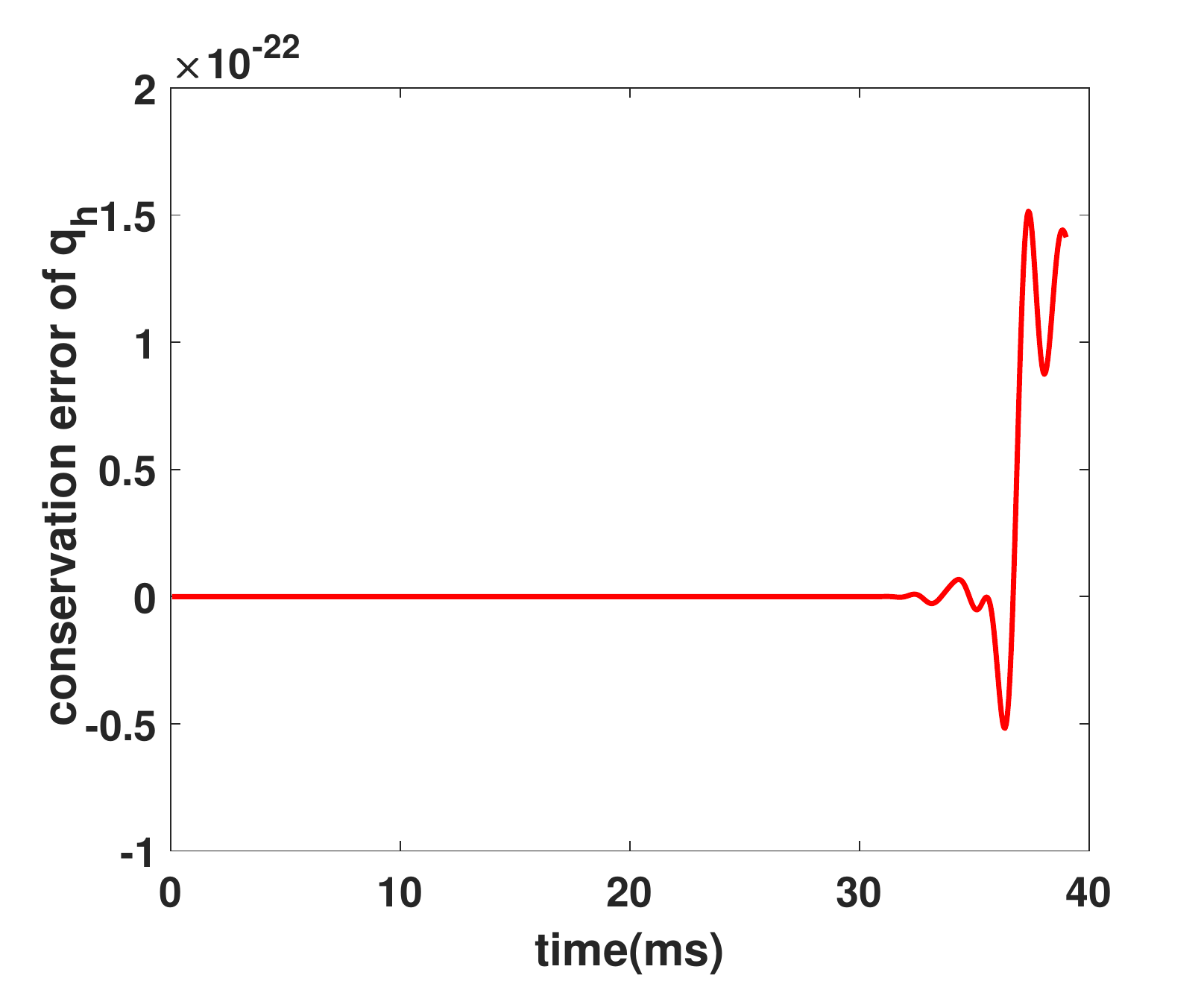}
\caption{Conservation errors for the acoustic example in \ref{sec:stationary:acoustic}. A uniform background mesh with 400 elements and piecewise
quadratic polynomials are used.
}
\label{exe:system:conservation}
\end{center}
\end{figure}

\section{Moving interface}
We now consider the scalar hyperbolic problem \eqref{eq:model}-\eqref{eq:interfacecond} for $t\in [0,T]$, with flux
\begin{equation}\label{eq:move:flux}
F(u)=au=\left\{\begin{array}{ll}
F_1(u_1)\equiv a_1 u_1, & x \in\Omega_1(t), \\
F_2(u_2)\equiv a_2 u_2, & x\in \Omega_2(t),
\end{array}\right.
\end{equation}
 and a moving interface with $x_\Gamma'(t)\neq0$. We assume  $a_1-x_\Gamma'(t)$, $a_2-x_\Gamma'(t)$ are non-zero, and have the same sign at any fixed time $t$. In the following we define a space-time CutFEM with discontinuous elements in both space and time following~\cite{hansbo2016cut,zahedi2017space,frachon2019cut}. We emphasise that
 we do not explicitly construct a space-time domain in $R^{d+1}$
 as is done in for example \cite{SBV11}.
 Here $d$ is the space dimension. The method we propose here is based on approximating the space-time integrals in the weak form by using first a quadrature rule in time. The implementation of the space-time unfitted finite element method we propose is straightforward and simple starting from an implementation of CutFEM for a stationary interface.

\subsection {Mesh and spaces}\label{eq:SPmeshandspace}
As before, let $\mathcal{T}_h$ be a quasi-uniform partition of the domain $\Omega$ generated independently of the position of the interface and let $\me_h$ denote the set containing the edges in this mesh. On this time independent mesh, that we refer to as the background mesh, we define the polynomial space $\widetilde{\mathcal{V}_h^{r_s}}$ as in \eqref{eq:BGspace}. For time $t\in [0, T]$ define $\mt_{h,i}(t)$ as in \eqref{eq: meshi} and $\me_{h,i}(t)$ as in \eqref{eq: edgesi}. These sets are now time dependent since the interface is moving and $\Omega_i$ changes with time. We also define the set
\begin{align}
 \mt_{h,\Gamma}(t)=\left\{I_j \in \mt_h: I_j \cap \Gamma(t) \neq \emptyset\right\}.
\end{align}
We discrete the interval $[0,T]$ with $0=t_0<t_1<t_2<\cdots<t_N=T$. During the time interval $I^n=[t^{n-1},t^n]$, the active meshes $\mt_{h,1}^n$ and  $\mt_{h,2}^n$, contain those elements in the background mesh that create the following subdomains $\mn_{h,1}^n$ and $\mn_{h,2}^n$, respectively
\begin{align}
 \mn_{h,1}^n&=\bigcup_{t\in I^n}\bigcup_{I_j\in \mt_{h,1}(t)}I_j, \
 \mn_{h,2}^n=\bigcup_{t\in I^n}\bigcup_{I_j\in \mt_{h,2}(t)}I_j.
\end{align}
We let $\me_{h,i}^n$  denote the set of interior edges in the active mesh $\mt_{h,i}^n$, for $i=1,2$.
We also define the set of elements in the background mesh that are cut by the interface during the interval $I^n$
\begin{align}
\mt_{h,\Gamma}^n=\left\{I_j \in  \mt_{h,\Gamma}(t):  t\in I^n \right\}.
\end{align}
Let $\Fhi^n$ contain those edges in the mesh $\mt_{h,\Gamma}^n$ that also belong to $\me_{h,i}^n$. Note that the set $\Fhi^n$ does not change in the time interval $I^n$.  For an illustration see Fig.  \ref{fig:move:domain}.
\begin{figure}[tbhp]
\begin{center}
\includegraphics[width=2.8in]{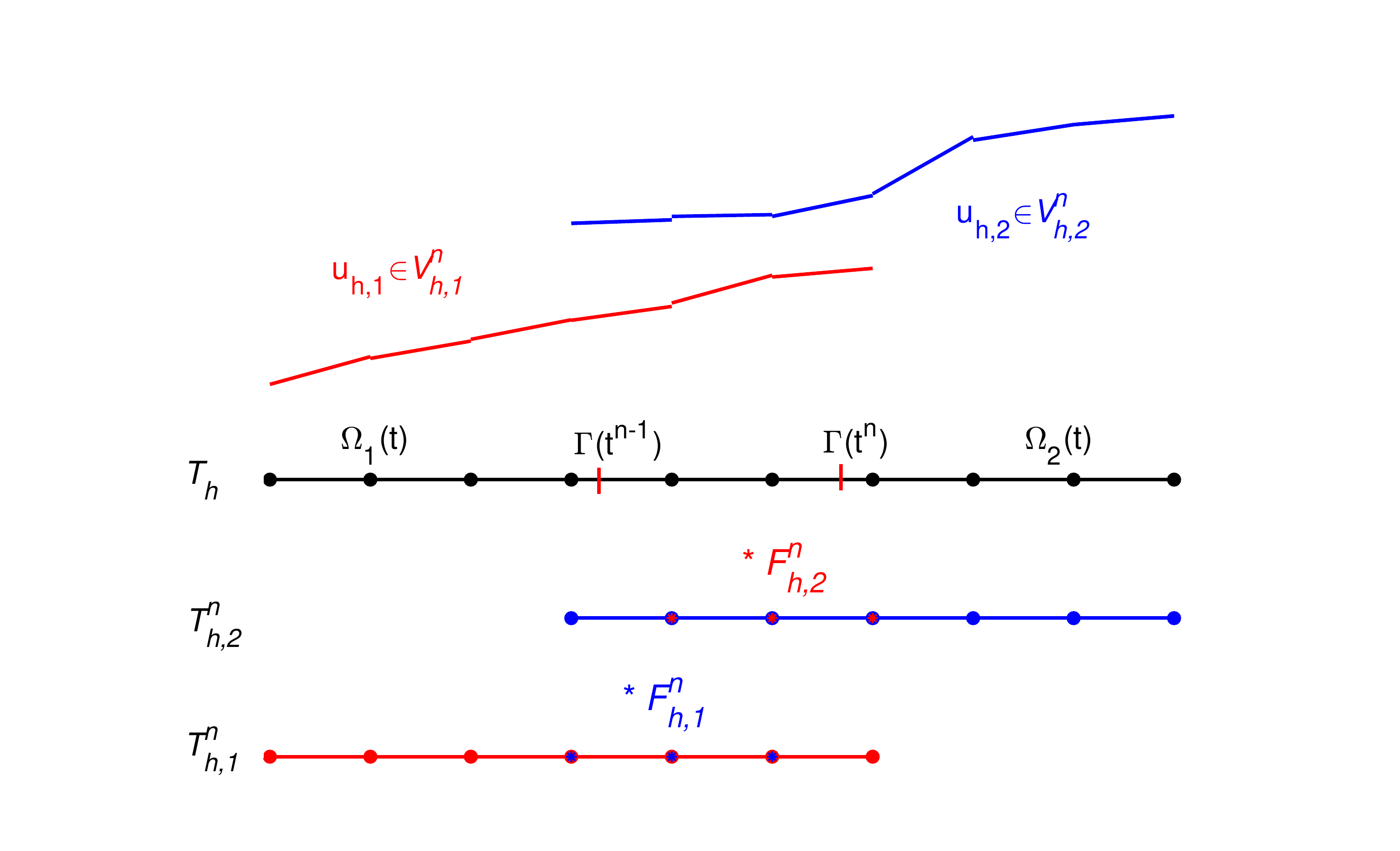}
\caption{Illustration of the active meshes and a function $u_h\in V_h^{n,1}$ at some time $t\in I^n$.
}
\label{fig:move:domain}
\end{center}
\end{figure}

In the proposed space-time method, we use piecewise polynomial spaces both in time and space. On the space-time slab $I^n\times \mn_{h,i}^n$, $i=1,2$ we define the space
$$V_{h,i}^{n,r}=\operatorname{P^{r_t}}(I^n) \otimes \widetilde{\mathcal{V}_h^{r_s}} |_{\mt_{h,i}^{n}}.$$
Here $r=(r_s, r_t)$, where $r_s$ and $r_t$ are the degree of the polynomials used in space and  time, respectively. Define the function space $\Vn$ as
\begin{equation}
\Vn=\left\{v_h=(v_{h,1},v_{h,2}): v_{h,i} \in V_{h,i}^{n,r}, i=1,2 \right\}.
\end{equation}
For example, for piecewise linear elements in time and space, $r=(1,1)$, a function $v_h\in \Vn$ can be expressed as $v_h=(v_{h,1},v_{h,2})$ with
\begin{align}\label{eq:def:u_h:p1}
v_{h, i}=v_{i00}+v_{i01} \frac{x-x_{k}}{h / 2}+v_{i10} \frac{t-t^{n-1}}{\Delta t^n}+v_{i 11} \frac{x-x_{k}}{h / 2} \frac{t-t^{n-1}}{\Delta t^n}.
\end{align}
Here, $v_{ikj}$ are the coefficients of the basis functions: $1, \frac{x-x_{k}}{h/ 2}, \frac{t-t^{n-1}}{\Delta t^n}, \frac{t-t^{n-1}}{\Delta t^n}\frac{x-x_{k}}{h/ 2}$ and $\Delta t^n=t^{n}-t^{n-1}$ denotes the time step of interval $I^n$.
Functions in $\Vn$ are discontinuous both in space and time. We define the jump and average of a function $v$ at $x$ as in~\eqref{eq: limitvx} and~\eqref{eq: jumpandmean} and the jump at a time $t^{n}$ as
\begin{align*}
&[v]^n=v^{n,+}-v^{n,-}, \text{ with } v^{n,-}=\lim\limits_{\epsilon\to 0^+}v(x,t^{n}-\epsilon), \ v^{n,+}=\lim\limits_{\epsilon\to 0^+}v(x,t^{n}+\epsilon). 
\end{align*}
When a function is single valued at  $t^n$ we will use the notation $v^n(\cdot)=v(\cdot,t^n)$.

\subsection{Weak formulation}
We now present a weak formulation where space and time are treated similarly.
For each time interval $I^n$,  given $u_h^{n-1,-}$,  find $u_h \in \Vn$ such that for $\forall v_h \in \Vn$
\begin{align}
&(u_h^{n,-},v_h^n)_{\Omega_1(t^{n})\cup\Omega_2(t^{n})} -(u_h^{n-1,-},v_h^{n-1})_{\Omega_1(t^{n-1})\cup\Omega_2(t^{n-1})}
\notag\\
&-\int_{I^n} (u_h,(v_h)_t)_{\Omega_1(t)\cup\Omega_2(t)} \ dt+ \int_{I^n} a_h(u_h,v_h) \ dt  + \rA \int_{I^n}  J_0(u_h,v_h)  \ dt= 0.
\label{scheme:move:cutDG0}
\end{align}
Here
 $u_h^{n-1,-}$ is the solution from the previous space-time slab (with $u_h^{0,-}$ given by the initial condition) and
\begin{align}
a_h(u_h,v_h) &=-(F(u_h),(v_h)_x)_{\Omega_1(t)\cup\Omega_2(t)}  - \sum_{i=1}^{2}\sum_{e\in\me_{h,i}(t)} \widehat{F}_e(u_h)[v_h]_e
\notag\\
&\quad
-\left([(F(u_h)-x_\Gamma'u_h)v_h]_{\Gamma(t)}+[F(u_h)-x_\Gamma'u_h]_\Gamma[\lambda v_h]_{\Gamma(t)} \right).
\label{spacetimeAh}
\end{align}
The flux $\widehat{F}_e(u_h)$ is defined as in~\eqref{eq:fluxe:state} with $\lambda_e=|a_i|$ for $e\in\me_{h,i}$, and the penalty parameter $\lambda$ is piecewise constant, see~\eqref{eq:lam},
  and will be chosen such that the scheme is stable and conservative (see Theorem \ref{thm:stabilitymovep}).  The stabilization term  $J_0(u_h,v_h)$ is defined as in~\eqref{stable0}, but with the set $\Fhi^n$ instead of $\Fhi(t)$. Since the set $\Fhi^n$ does not change in the time interval $I^n$,  and $u_h$ and $v_h$ are polynomials in time, the integral  $\int_{I^n} J_0(u_h,v_h)  \ dt$, can be computed analytically.  The stabilization term is introduced to control the condition number of the resulting system matrix independently of how the geometry cuts through the background mesh, and defines an extension of $u_h$ to the entire active mesh which is needed when the space-time integrals in the weak form are approximated by quadratures rules, see Section 4.4.

To see that the weak formulation above is consistent, we multiply equation~\eqref{eq:model} by a test function $v\in V_h^{n,r}$, integrate in both space and time and impose the interface condition \eqref{eq:interfacecond} weakly. Integration by parts in space and using \eqref{eq:fluxone}, together with integration by parts in time, yields
\begin{align}\label{eq:integrationtime}
& \int_{I^n} \int_{\Omega_1(t)\cup\Omega_2(t)} (uv)_t  \ dx dt = \int_{I^n}  \frac{d}{dt} \int_{\Omega_1(t)\cup\Omega_2(t)} u v \ dx dt + \int_{I^n} [x_\Gamma'(t)u v]_\Gamma  \ dt \nonumber \\
& =
(u^n,v^n)_{\Omega_1(t^{n})\cup\Omega_2(t^{n})} -(u^{n-1},v^{n-1})_{\Omega_1(t^{n-1})\cup\Omega_2(t^{n-1})} + \int_{I^n} [x_\Gamma'(t)u v]_\Gamma  \ dt.
\end{align}
By adding $([u]^{n-1},v^{n-1})_{\Omega_1(t^{n-1})\cup\Omega_2(t^{n-1})}$,   using $[u]^{n}=0$ at all time $t=t^{n}$ and using the identity $u_tv=(uv)_t-uv_t$  we get the proposed weak formulation.

Choosing the test function $v_h=1$ in \eqref{scheme:move:cutDG0} we have
\begin{align}
\int_{\Omega(t^{n})}& u_h^{n,-} \ dx-\int_{\Omega(t^{n-1})} u_h^{n-1,-} \ dx\notag\\
&=\int_{I^n}\left(\widehat{F}(u_h(x_L,t))-\widehat{F}(u_h(x_R,t))\right) \ dt \notag\\
&+\int_{I^n}([F(u_h)-x_\Gamma'u_h]_\Gamma+(\lambda_2-\lambda_1)[F(u_h)-x_\Gamma'u_h]_\Gamma) \ dt.
\end{align}
With $\lambda_2-\lambda_1+1=0$, as in \eqref{eq:conservcondlamb},  the proposed space-time CutFEM is conservative.

\begin{remark}
Note that we can also consider the space-time formulation without integration by parts in time. Thus,  given $u_h^{n-1,-}$, find $u_h \in \Vn $ such that
\begin{align}
&\int_{I^n} ((u_h)_t, v_h )_{\Omega_1(t)\cup\Omega_2(t)} dt+([u_h]^{n-1},v^{n-1}_h)_{\Omega_1(t^{n-1})\cup\Omega_2(t^{n-1})}
\notag\\
-&\int_{I^n} (F(u_h),(v_h)_x)_{\Omega_1(t)\cup\Omega_2(t)} \ dt - \int_{I^n}  \sum_{i=1}^{2}\sum_{e\in\me_{h,i}} \widehat{F}_e(u_h)[v_h]_e\ dt
\notag\\
-&\int_{I^n}\left([F(u_h)v_h]_\Gamma+[F(u_h)-x_\Gamma'u_h]_\Gamma[\lambda v_h]_\Gamma\right) \ dt
+\rA\int_{I^n} J_0(u_h,v_h)dt=0,
\label{scheme:move:cutDG1}
\end{align}
for all  $v_h \in \Vn$. This weak formulation is also consistent. With $v_h=1$  we get
\begin{align}
&\int_{I^n}\int_{\Omega_1(t)\cup\Omega_2(t)}(u_h)_t \ dxdt- \int_{I^n}(\widehat{F}(u(x_L,t))-\widehat{F}(u(x_R,t)) ) \ dt \notag\\ \quad
&-\int_{I^n}\left([F(u_h)]_\Gamma+[F(u_h)-x_\Gamma'u_h]_\Gamma(\lambda_2-\lambda_1)\right) \ dt +\int_{\Omega(t^{n-1})}[u_h]^{n-1} \ dx=0.
\end{align}
If \eqref{eq:integrationtime} (with $v=1$) holds for the discrete function $u_h$, that is if
\begin{align}
&\int_{I^n}\int_{\Omega_1(t)\cup\Omega_2(t)}(u_h)_t \ dxdt - \int_{I^n} [x_\Gamma'u_h]_\Gamma \ dt + \int_{\Omega(t^{n-1})}[u_h]^{n-1} \ dx
\notag\\ \quad
&= \int_{\Omega(t^{n})} u_h^{n,-} \ dx -\int_{\Omega(t^{n-1})} u_h^{n-1,-} \ dx,\label{eq:integrationtime2}
\end{align}
it follows that the scheme \eqref{scheme:move:cutDG1} is conservative when \eqref{eq:conservcondlamb} holds, i.e. $\lambda_2-\lambda_1=-1$. However, in the fully discrete scheme we use a quadrature rule to approximate the time integral. Since the domain is time-dependent, \eqref{eq:integrationtime2} can not hold exactly.
Therefore a method based on the weak formulation \eqref{scheme:move:cutDG1} may not be exactly conservative with a conservation error depending on the accuracy of the quadrature rule. This is demonstrated in the numerical examples.
\end{remark}

\subsection{Stability analysis of the semi-discrete scheme}
We now consider the energy stability of the proposed semi-discrete space-time CutFEM  \eqref{scheme:move:cutDG0}.
Similar to the scalar advection problem \eqref{eq:linear:state} with a stationary interface, we consider the weighted energy with $\eta>0$, i.e.,
\begin{align}\label{eq:move:def:energy}
E_\eta(t,u_h)=\frac{1}{2}\int_{\Omega_1(t)} (u_{h,1}(x,t))^2 \ dx+\frac{\eta}{2}\int_{\Omega_2(t)} (u_{h,2}(x,t))^2 \ dx.
\end{align}
For convenience, we will use the notation $\Omega_1^n$, $\Omega_2^n$ for $\Omega_1(t^n)$, $\Omega_2(t^n)$, respectively.  Choosing $v_h=(u_{h,1},\eta u_{h,2})$ in \eqref{scheme:move:cutDG0} we have
\begin{align}
&\int_{\Omega_1^{n}} (u_{h,1}^{n,-})^2 \ dx-\int_{\Omega_1^{n-1}}u_{h,1}^{n-1,-}u_{h,1}^{n-1,+} \ dx-\int_{I_n}\int_{\Omega_1(t)}u_{h,1}(u_{h,1})_t  \ dx dt
\notag\\
&+\eta \int_{\Omega_2^{n}} (u_{h,2}^{n,-})^2 \ dx -\eta\int_{\Omega_2^{n-1}}u_{h,2}^{n-1,-}u_{h,2}^{n-1,+}dx -\eta\int_{I_n}\int_{\Omega_2(t)}u_{h,2}(u_{h,2})_t  \ dx dt
 \notag\\
&=-\int_{I_n} a_h(u_h,v_h) \ dt  -\rA \int_{I_n}  J_0(u_h,v_h)  \ dt.
\label{eq:move:stable3}
\end{align}
Integrating the last two terms in the left hand side of equation \eqref{eq:move:stable3} we get
\begin{align}
Lhs&=-\frac{1}{2}\left(\int_{\Omega_1^{n}} (u_{h,1}^{n,-})^2 dx-\int_{\Omega_1^{n-1}}(u_{h,1}^{n-1,+})^2dx-\int_{I_n}x_\Gamma'(u_{h,1}( x_\Gamma,t)) ^2\ dt\right)\notag\\
&-\frac{\eta}{2}\left(\int_{\Omega_2^{n}} (u_{h,2}^{n,-})^2 dx-\int_{\Omega_2^{n-1}}(u_{h,2}^{n-1,+})^2dx+\int_{I_n}x_\Gamma'(u_{h,2}( x_\Gamma,t)) ^2 \ dt\right)
\notag\\
&+\int_{\Omega_1^{n}} (u_{h,1}^{n,-})^2 dx-\int_{\Omega_1^{n-1}}u_{h,1}^{n-1,+}u_{h,1}^{n-1,-}dx \notag\\
&+\eta\left(\int_{\Omega_2^{n}} (u_{h,2}^{n,-})^2 dx-\int_{\Omega_2^{n-1}}u_{h,2}^{n-1,+}u_{h,2}^{n-1,-}dx\right)
\notag\\
=&E_\eta(t^{n},u_h^{n,-})-E_\eta(t^{n-1},u_h^{n-1,-}) + \frac{1}{2} \left( \| [u_{h,1}]^{n-1} \|^2_{\Omega_1^{n-1}} + \eta \| [u_{h,2}]^{n-1} \|^2_{\Omega_2^{n-1}} \right)
\notag\\
&-\frac{1}{2} \int_{I_n} \left(\eta x_\Gamma'(u_{h,2}( x_\Gamma,t)) ^2-x_\Gamma'(u_{h,1}( x_\Gamma,t)) ^2\right) \ dt.
\label{eq:energy:lhs}
\end{align}
As in the analysis in Section 3.3, using the definition of $a_h(\cdot,\cdot)$, and integrating the term $(au_{h,i},(u_{h,i})_x)_{\Omega_i(t)}$, taking into account that $u_{h,i}$ is discontinuous across the edges, using that $\frac{a}{2}[u_{h}^2]_e=a\{ u_h \}_e[u_{h}]_e$, and the definition of $\widehat{F}_e$ \eqref{eq:fluxe:state}, we have
\begin{align}
&-\int_{I_n}  a_h(u_h,v_h)  \ dt -\rA \int_{I_n}  J_0(u_h,v_h)  \ dt  \nonumber \\
&+\frac{1}{2} \int_{I_n} x_\Gamma'\left(\eta (u_{h,2}( x_\Gamma,t)) ^2-(u_{h,1}( x_\Gamma,t)) ^2\right) \ dt = -ET-IT,
\label{eq:move:RHS}
\end{align}
where $v_h=(u_{h,1},\eta u_{h,2})$ and
\begin{align}\label{eq:ET}
ET&=\int_{I_n}\left(\frac{a_1}{2}\sum_{e\in \me_{h,1}}([u_{h,1}]_e)^2+\eta\frac{a_2}{2}\sum_{e\in \me_{h,2}}([u_{h,2}]_e)^2 \right)\ dt,
\notag\\
&+\int_{I_n}\rA \left(\sum_{e\in \mathcal{F}_{h,1}^n} J_0(u_{h,1},u_{h,1})+\sum_{e\in \mathcal{F}^n_{h,2}} \eta J_0(u_{h,2},u_{h,2}) \right) \ dt,
\end{align}
\begin{align}\label{eq:IT}
IT=
&\int_{I_n} \left(\frac{1}{2}(a_1-x_\Gamma')u_{h,1} +\lambda_1[(a-x_\Gamma')u_h]_\Gamma \right)u_{h,1} \ dt
\notag\\
-&\eta 
\int_{I_n} \left(\frac{1}{2}(a_2-x_\Gamma')u_{h,2}+\lambda_2[(a-x_\Gamma')u_{h}]_\Gamma \right)u_{h,2} \ dt. 
\end{align}
Collecting the results in \eqref{eq:move:stable3}-\eqref{eq:IT} and letting $\widetilde{a_1}=a_1-x_\Gamma'$ and $\widetilde{a_2}=a_2-x_\Gamma'$
yields
\begin{align}
&E_\eta(t^n,u_h^{n,-})-E_\eta(t^{n-1},u_h^{n-1,-})+  \frac{1}{2} \left( \| [u_{h,1}]^{n-1} \|^2_{\Omega_1^{n-1}} + \eta \| [u_{h,2}]^{n-1} \|^2_{\Omega_2^{n-1}} \right) \notag\\
&=-ET-\int_{I_n}\mathbf{u}_h^{\text T}S\mathbf{u}_h\ dt,
\end{align}
with
\begin{equation} \label{eq:matrixAm}
S=\begin{pmatrix}
(\frac{1}{2}-\lambda_1)\widetilde{a_1}    & \frac{1}{2}(\widetilde{a_2}\lambda_1+\widetilde{a_1} \eta \lambda_2)  \\
\frac{1}{2}(\widetilde{a_2}\lambda_1+\widetilde{a_1} \eta \lambda_2)   &  -(\frac{1}{2}+\lambda_2) \eta \widetilde{a_2}
\end{pmatrix},
\quad\mathbf{u}_h=\left(\begin{array}{c}
u_{h,1}(x_\Gamma,t) \\
u_{h,2}(x_\Gamma,t)
\end{array}\right).
\end{equation}
In \eqref{eq:ET} $ET\geq 0$ and will not contribute to energy growth. Hence, if the matrix $S$ is positive semi-definite, we obtain energy stability. Note that the matrix $S$ in \eqref{eq:matrixAm} is of the same form as in the case of a stationary interface,  see \eqref{eq:matrixAs},  but with $\widetilde{a_i}$ instead of $a_i$.  
Thus,  by Lemma \ref{lemma:matrixA} we have the following theorem.
\begin{theorem}\label{thm:stabilitymovep}
Consider the problem \eqref{eq:model}-\eqref{eq:interfacecond} with the flux function \eqref{eq:move:flux},  and a moving interface at $x_\Gamma(t)$ such that
$a_1-x_\Gamma'(t)$ and $a_2-x_\Gamma'(t)$ have equal, non-zero and constant sign for all time $t$. With penalty parameters $\lambda_1$ and $\lambda_2$ satisfying \eqref{eq:conservcondlamb} and
\begin{align}\label{eq:state:parametercondition}
\left\{\begin{array}{ll}
{\lambda_1\leq\frac{1}{2},\lambda_2\leq-\frac{1}{2},}&{\text{if } a_1-x_\Gamma'>0,a_2-x_\Gamma'>0,}\\
{\lambda_1\geq\frac{1}{2},\lambda_2\geq-\frac{1}{2},}&{\text{if } a_1-x_\Gamma'<0,a_2-x_\Gamma'<0,}\\
\end{array}\right.
\end{align}
 the space-time CutFEM~\eqref{scheme:move:cutDG0}  is conservative and there exists a positive $\eta$ such that the energy defined in \eqref{eq:move:def:energy} does not grow with time.
\end{theorem}

\subsection{Quadrature in time}\label{sec:quadrature}
As in \cite{hansbo2016cut,zahedi2017space,frachon2019cut}, we approximate the space-time integrals in the variational formulation using quadrature rules, first in time and then in space.
Note that using a quadrature rule in time we have
\begin{align}
\int_{I^n}\int_{\Omega_i(t)} f(x,t) dxdt\approx \sum_{q=1}^{n_{q}} \omega_{q}^{n}\int_{\Omega_i(t_q^n)} f(x,t_q^n)dx.
\end{align}
Here, $\omega_{q}^{n}$ are the quadrature weights, $t_{q}^{n}$, $q=1, \ldots, n_{q}$ are quadrature points in the interval $I^n$, and $n_{q}$ is the number of quadrature points.

In the numerical examples,  both the trapezoidal rule and Simpson's rule are used.  In the time interval $I^n=[t^{n-1},t^{n}]$,  the trapezoidal rule is given by  two quadrature points, $t_{1}^{n}=t^{n-1}$ and $t_{2}^{n}=t^{n},$ and weights $\omega_{1}^{n}=\omega_{2}^{n}=\frac{\Delta t^{n}}{2}$. In Simpson's quadrature rule, the three quadrature points are $t_{1}^{n}=t^{n-1},$ $t_{2}^{n}=\frac{t^{n-1}+t^{n}}{2}$, and $t_{3}^{n}=t^{n}$, and the weights are $\omega_{1}^{n}=\omega_{3}^{n}=\frac{\Delta t^{n}}{6}$ and $\omega_{2}^{n}=\frac{4 \Delta t^{n}}{6}$.

\subsection{Numerical examples with moving interfaces}
We use the proposed space-time cut finite element method \eqref{scheme:move:cutDG0} to solve problem \eqref{eq:model}-\eqref{eq:interfacecond} with flux \eqref{eq:move:flux} and an moving interface $x_\Gamma(t)$.

\subsubsection{Scalar problem with a moving interface: Accuracy} \label{sec:ex1moving}
We use $a_1=2$, $a_2=1$, $x_\Gamma(0)=10^{-4}$, $x'_\Gamma=0.111$, and the initial value
\begin{align}
u(x,0)=f(x)=\left\{\begin{array}{ll}
{\sin(2\pi x),} & {x\in[-1,x_\Gamma(0)]}, \\
{\beta\sin(2\pi \beta x+2\pi x_\Gamma(0) (1-\beta))}, & {x\in[x_\Gamma(0),1].}
\end{array}\right.
\end{align}
Here $\beta=\frac{a_1-x_\Gamma'}{a_2-x_\Gamma'}$. This initial condition satisfies the interface condition \eqref{eq:interfacecond}. The inflow boundary condition $g(t)=u(x_L,t)=\sin(2\pi(x_L-2t))$ is used. The outflow boundary condition is used on the right boundary.
The exact solution is
\begin{align}\label{test:move:exact}
u(x,t)=\left\{\begin{array}{ll}
{\sin(2\pi (x-2t)),} & {x\in[-1,x_\Gamma(t)]}, \\
{\beta\sin(2\pi \beta(x-t)+2\pi x_\Gamma(0)(1-\beta))}, & {x\in[x_\Gamma(t),1].}
\end{array}\right.
\end{align}
We use the space-time CutFEM \eqref{scheme:move:cutDG0} with discontinuous  piecewise linear polynomials in time and discontinuous  piecewise linear and  quadratic polynomials in space. For the time integration we use Simpson's rule. The time step is $\Delta t=h/12$ when linear elements are used in space, that is $r=(1,1)$, and $\Delta t=0.005h$ when quadratic polynomials are used in space, that is $r=(2,1)$. In the latter case the time step is small enough so that the error is not dominated by the error in the time discretization.  We use $\rA=0.75$, $\lambda_1=0$, and $\lambda_2=\lambda_1-1$. We solve the problem up to time $t=0.1$. In Table \ref{table:move:accuracy}, we show the $L^2$-and $L^\infty$-errors for different mesh sizes $h=2/N$  and we observe that the space-time CutFEM \eqref{scheme:move:cutDG0} has the optimal order of accuracy for this moving interface problem. Note that we use a uniform background mesh with mesh size $h$ and the interface cuts the mesh arbitrarily as it evolves in time.  
\begin{table}[!bhtp]
\caption{\label{table:move:accuracy} {
Errors and orders of accuracy at $t=0.1$ for the  problem  in Sect.  \ref{sec:ex1moving}  with a moving interface. The approximation uses space-time polynomials of orders (1,1) and  (2,1), respectively, and a uniform  background mesh with $N$ elements in space.
} }
\begin{small}
\begin{tabular}{c|cccc|cccc}
\hline\noalign{\smallskip}
N &$L^2 $ error & order &$L^{\infty} $ error & order &$L^2 $ error & order &$L^{\infty} $ error & order \\
\noalign{\smallskip}\hline\noalign{\smallskip}
&\multicolumn{4}{c|}{$P^1$ in space}&\multicolumn{4}{c}{$P^2$ in space}\\\noalign{\smallskip}\hline\noalign{\smallskip}
20&	1,46E-01	&	-	&	4,50E-01	&	-	&	1,25E-02	&	-	&	7,70E-02	&	-	\\
40&	3,91E-02	&	1,90	&	1,40E-01	&	1,69	&	1,63E-03	&	2,94	&	1,03E-02	&	2,90	\\
80&	1,00E-02	&	1,96	&	3,76E-02	&	1,89	&	2,07E-04	&	2,98	&	1,31E-03	&	2,98	\\
160&	2,56E-03	&	1,97	&	9,65E-03	&	1,96	&	2,62E-05	&	2,98	&	1,94E-04	&	2,76	\\
320&	6,41E-04	&	1,99	&	2,71E-03	&	1,83	&	3,27E-06	&	3,00	&	2,38E-05	&	3,03	\\
\hline\noalign{\smallskip}
\end{tabular}
\end{small}
\end{table}

\subsubsection{Scalar problem with a moving interface: Conservation}\label{sec:ex2moving}
We consider the same example as in \cite{la2016well} but with a moving interface
\begin{align}
x_\Gamma(t)=x_\Gamma(0)+0.4\sin(t)(x_\Gamma(0)-x_L)(x_R-x_\Gamma(0)),
\end{align}
where $x_\Gamma(0)=-0.499$.
Let $a_1=2$, $a_2=1$, $f(x)=0$ on the domain $\Omega=[-1,1]$.
The inflow boundary condition $g(t)=\sin (4 \pi(-1+3 t))$ is used on the left boundary. We solve the problem up to time $t=1$ when the outflow information is still zero. The space-time CutFEM is used to solve the problem with parameters $\lambda_1=0$ and $\lambda_2=\lambda_1-1$, a uniform background mesh with $400$ elements, and linear elements both in time and space, $r=(1,1)$. Simpson's rule is used for the time integration with Courant number $C=1/6$.
We measure the conservation error as in \eqref{eq:em2}, but replacing $u_h^n$ by $u_h^{n,-}$ and noting that $\Omega=\Omega_1(t)\cup\Omega_2(t)$.

In the left panel of Fig. \ref{fig:moving:spacetime2:p0p0:p1p1},  the numerical solution $u_h$ is shown and we can observe that the proposed space-time CutFEM can simulate the problem with a moving interface well.  We note that the solution has a weak discontinuity and it is not surprising that small oscillations appear.
In the right side of Fig. \ref{fig:moving:spacetime2:p0p0:p1p1},
we show the conservation error $e(t)$ for the numerical solution $u_h$. We see that the proposed space-time method is conservative.  In Fig. \ref{fig:moving:spacetime1:p1p1:conservationerror},  we also show the numerical solution $u_h$ and the conservation error $e(t)$ using the variational formulation \eqref{scheme:move:cutDG1}. We observe that this scheme also simulates this problem well, but the conservation error is significantly larger when the weak formulation \eqref{scheme:move:cutDG1} is used.

\begin{figure}[tbhp]
\begin{center}
\includegraphics[width=2.3in]{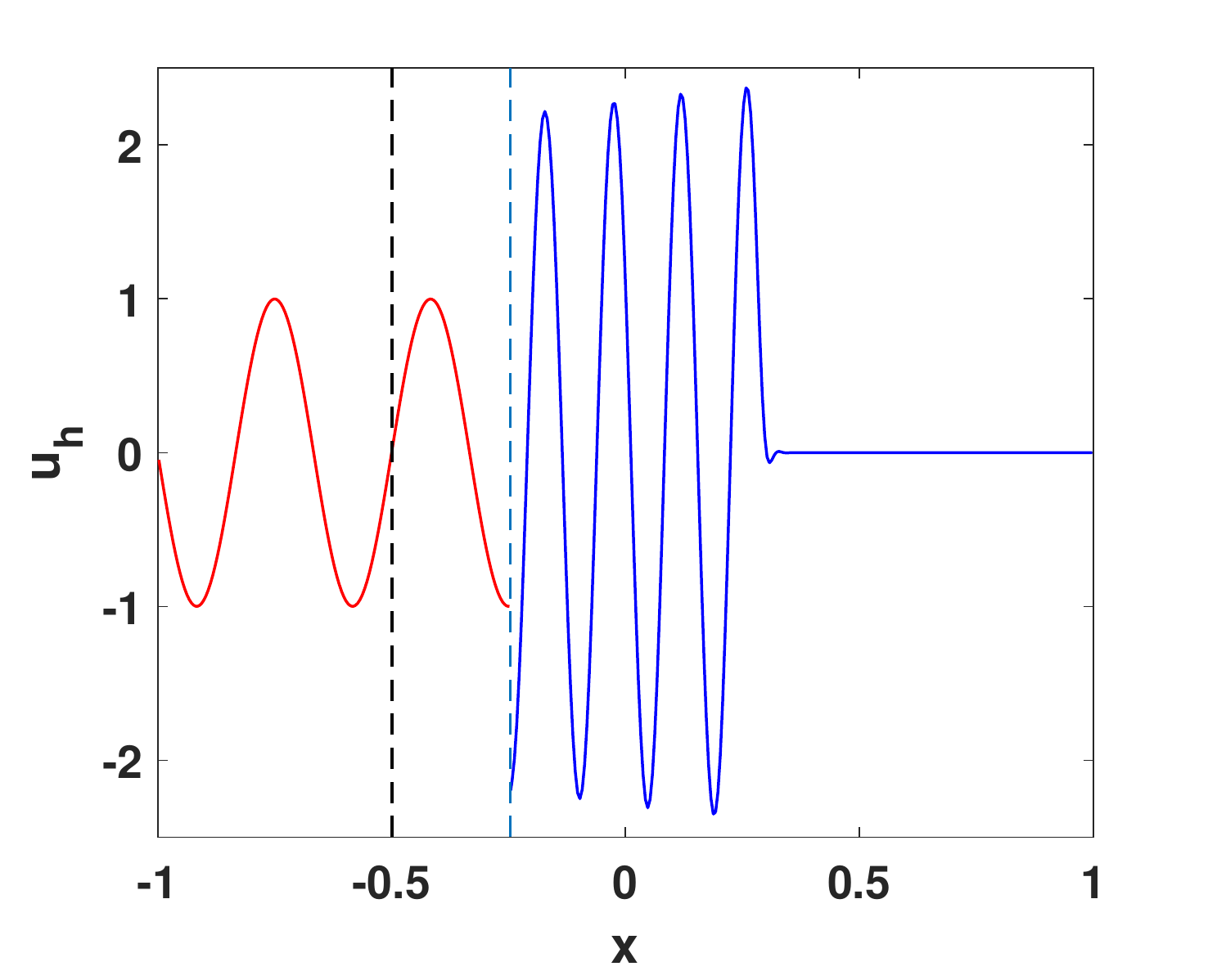}
\includegraphics[width=2.3in]{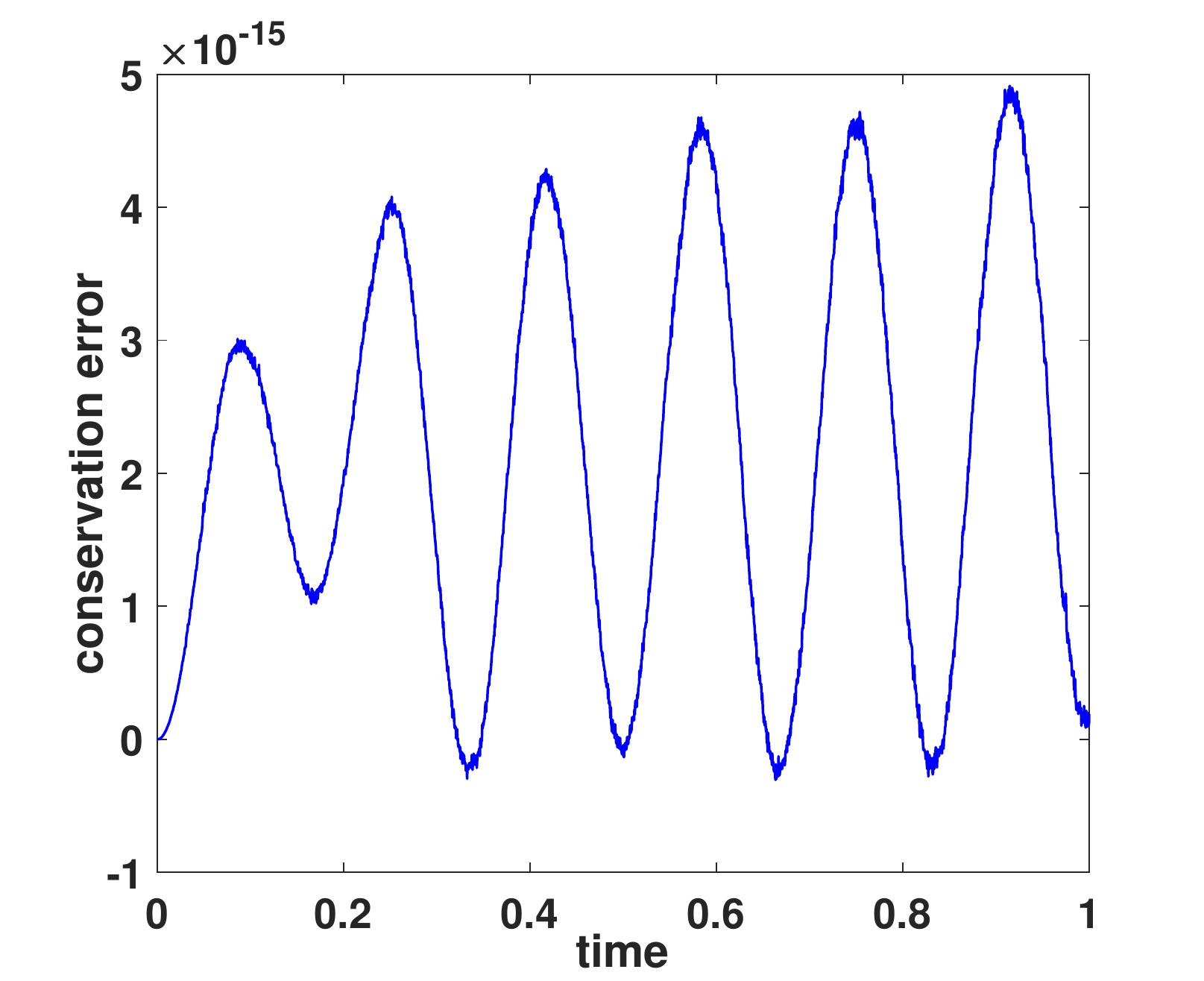}
\caption{Results for the problem  in Sect. \ref{sec:ex2moving}, solved by the proposed space-time method \eqref{scheme:move:cutDG0} with r=(1,1) on a uniform background mesh with 400 elements in space. Left: $u_h$ at $t=1$, with initial and present interface positions  indicated by dashed lines. Right: Conservation error $e(t)$ (see Section 3.4.2).
}
\label{fig:moving:spacetime2:p0p0:p1p1}
\end{center}
\end{figure}

\begin{figure}[tbhp]
\begin{center}
\includegraphics[width=2.3in]{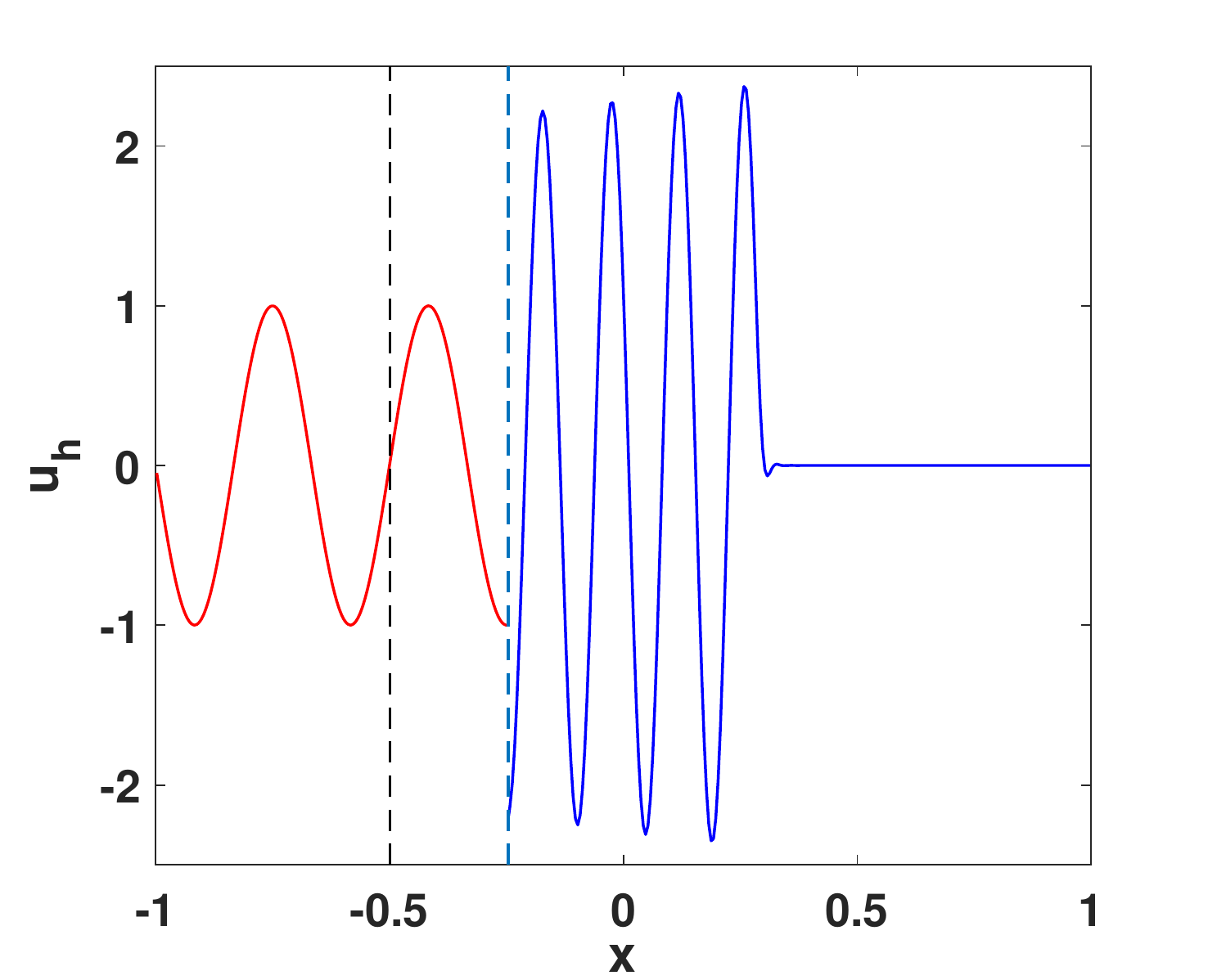}
\includegraphics[width=2.3in]{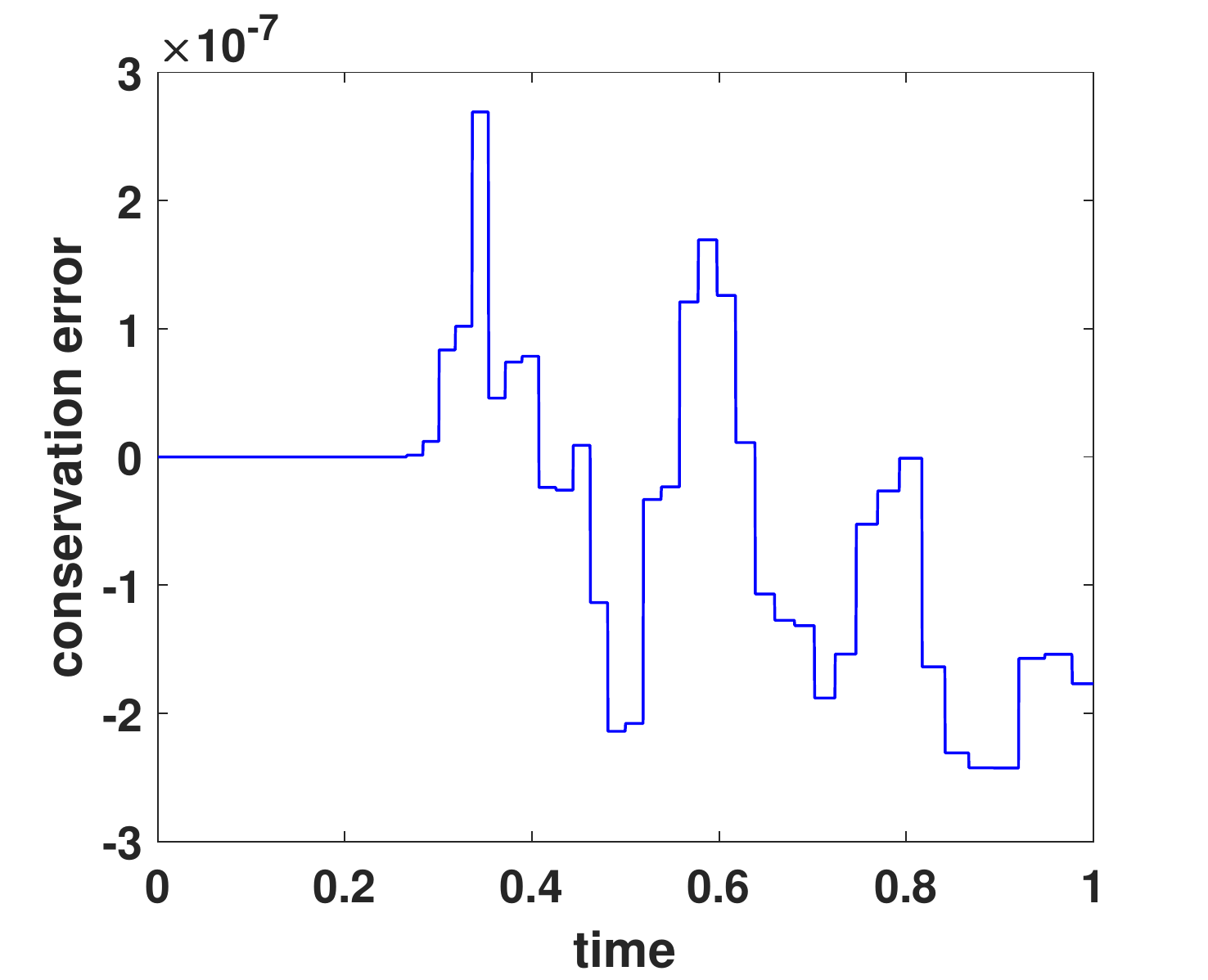}
\caption{Results for the  problem in Sect. \ref{sec:ex2moving}, solved by the space-time  method \eqref{scheme:move:cutDG1} with r=(1,1)  on a uniform background mesh with 400 elements in space. Left: $u_h$ at $t=1$, with initial and present interface positions indicated by dashed lines. Right: Conservation errors, (see Section \ref{sec:stationary:scalar:conservation}).
}
\label{fig:moving:spacetime1:p1p1:conservationerror}
\end{center}
\end{figure}

\subsection{A locally implicit method }
We now combine the proposed space-time CutFEM  with an explicit CutFEM. The space-time method is active in a neighbourhood of the interface and the explicit CutFEM method is applied away from the interface. For simplicity, we only consider the case $a_1-x_\Gamma'>0$, $a_2-x_\Gamma'>0$, and piecewise linear polynomials both in space and time.

Recall the sets $\Fhi^n$, $i=1,2$ from Section \ref{eq:SPmeshandspace}. We now let $\Omega_{l}$ be the subdomain containing the set of elements that have an edge in $\cup_{i} \Fhi^n$ and denote by $\Omega_{i, E}$ the remaining part of $\Omega_{i}$, i.e, the elements in $\Omega_i$ that are not in $\Omega_{l}$. Note that $\Omega_{i,E} \subset \Omega_i(t)$ for all  $t\in I^n$ and that no elements in $\Omega_{i,E}$ are cut by the interface during the time interval $I^n$.
In Fig. \ref{LocalImplicit} we illustrate how the space-time domain $I^n\times \Omega$ can be partitioned into the three parts, $I^n \times \Omega_l$, $I^n \times \Omega_{1,E}$, and  $I^n \times \Omega_{2,E}$.

\begin{figure}[tbhp]
\begin{center}
\includegraphics[width=2.8in]{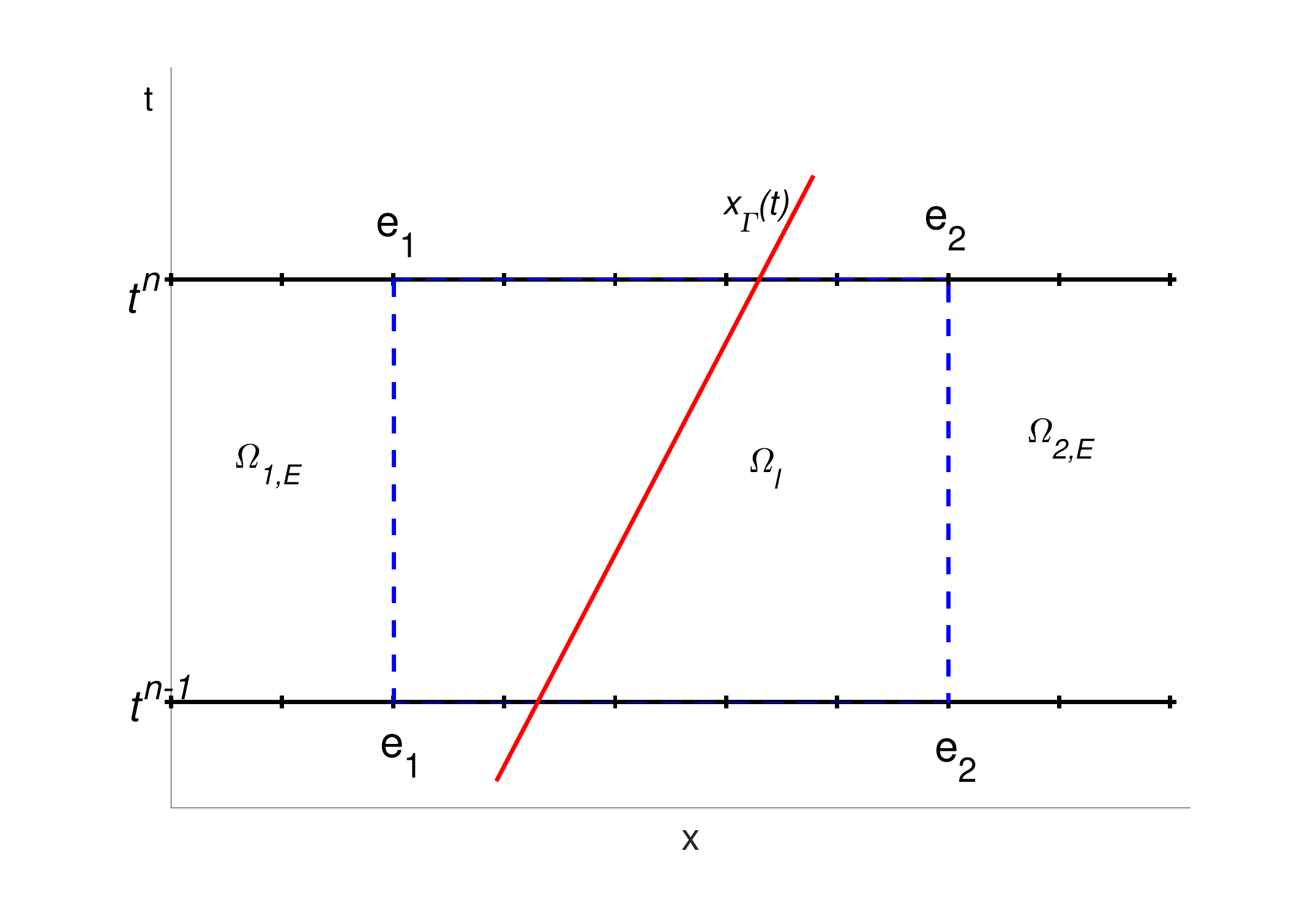}
\caption{The domains in the locally implicit scheme. }
\label{LocalImplicit}
\end{center}
\end{figure}

In regions away from the interface we want to apply a standard explicit DG method. Recall the mesh $\mt_{h}$ and the piecewise polynomial space $\widetilde{\mathcal{V}_h^1}$ defined in Section~\ref{sec:notapro}. 
We define the following  meshes and spaces restricted to $\Omega_{i,E}$,
\begin{equation}\label{eq: meshiR}
  \mt_{h,i}^E=\left\{I_j \in \mt_{h}: I_j \cap \Omega_{i,E} \neq \emptyset\right\}, 
\end{equation}
and
\begin{equation}
\VE=\widetilde{\mathcal{V}_h^1} |_{\mt_{h,i}^E}, \quad i=1,2.
\end{equation}
We now formulate a standard DG method with a two stage second order Runge-Kutta method: given $\hat{u}_{h,i}^{n-1}\in \VE$ find $\hat{u}_{h,i}^{(1)} \in \VE$ and $\hat{u}_{h,i}^{n} \in \VE$ such that
\begin{align}
&\left( \hat{u}_{h,i}^{(1)}-\hat{u}_{h,i}^{n-1},\hat{v}_h \right)_{\Omega_{i,E}}+\Delta t^n A_h(\hat{u}_{h,i}^{n-1},\hat{v}_{h})=0, \quad \forall \hat{v}_{h} \in \VE,  \label{eq:stage1}\\
&\left( \hat{u}_{h}^{n}-\frac{1}{2}(\hat{u}_{h}^{n-1}+\hat{u}_{h}^{(1)}), \hat{v}_{h} \right)_{\Omega_{i,E}}+\frac{\Delta t^n}{2} A_h(\hat{u}_{h}^{(1)},\hat{v}_{h}) =0, \quad \forall \hat{v}_{h} \in \VE,  \label{eq:stage2}
\end{align}
with
\begin{align}\label{scheme:Ah:time}
A_h(\hat{u}_{h,i},\hat{v}_{h})&=
- (F(\hat{u}_{h,i}),(\hat{v}_{h})_x)_{\Omega_{i,E}} -\sum_{i=1}^{2}\sum_{e\in\me_{h,i} \cap \Omega_{i,E}} \widehat{F}_e(\hat{u}_{h,i})[\hat{v}_{h}]_e,
\end{align}
and $\widehat{F}_e(\hat{u}_{h,i})$ as in~\eqref{eq:fluxe:state} at interior edges with $\lambda_e=|a_i|$ for $e\in\me_{h,i}$, $i=1,2$.

In the space-time slab $I^n \times \Omega_l$ we use the proposed space-time method. The active meshes and the spaces are defined exactly as in Section~\ref{eq:SPmeshandspace}, but with $\mt_{h,i}(t)$, $i=1,2$ restricted to $\Omega_l$. Thus, given $u_h^{n-1,-}$, the solution from the previous space-time slab, find $u_h \in \Vn $ such that
\begin{align}
&\sum_{i=1}^2\left((u_h^{n,-},v_h^{n})_{\Omega_l\cap \Omega_i(t^{n})} -(u_h^{n-1,-},v_h^{n-1})_{\Omega_l\cap \Omega_i(t^{n-1})}
-\int_{I^n} (u_h,(v_h)_t)_{\Omega_l\cap \Omega_i(t)} \ dt\right)
\notag\\
&+ \int_{I^n} a_h(u_h,v_h) \ dt  + \rA \int_{I^n}  J_0(u_h,v_h)  \ dt= 0, \quad \forall v_h \in \Vn,
\end{align}
with
\begin{align}
a_h(u_h,v_h) &=-\sum _{i=1}^2\left( (F(u_h),(v_h)_x)_{\Omega_l \cap \Omega_i(t)}  + \sum_{e\in\me_{h,i}(t)\cap\Omega_l} \widehat{F}_e(u_h) [v_h]_e \right)
\notag\\
&\quad
-\left([(F(u_h)-x_\Gamma'u_h)v_h]_{\Gamma}+[F(u_h)-x_\Gamma'u_h]_\Gamma[\lambda v_h]_{\Gamma} \right),
\end{align}
and $\lambda$ as in Theorem~\ref{thm:stabilitymovep}. We choose $r=(1,1)$.

Taking the test functions to be one in both schemes, i.e., $\hat{v}_h=1$, and $v_h=1$ and assuming for simplicity that the contributions at the physical boundary $x=x_L$ and $x=x_R$ vanish we have
\begin{align} \label{eq:vhat1}
&\sum_{i=1}^{2} \left( (\hat{u}_{h,i}^{n}, 1)_{\Omega_{i,E}}- (\hat{u}_{h,i}^{n-1}, 1)_{\Omega_{i,E}}
-\frac{\Delta t^n}{2} \left(\widehat{F}_{e_i}(\hat{u}_{h,i}^{n-1}) +\widehat{F}_{e_i}(\hat{u}_{h,i}^{(1)}) \right)  \right)=0,\\
&\sum_{i=1}^{2} \left(  (u_h^{n},1)_{\Omega_l\cap \Omega_i(t^{n})} -(u_h^{n-1,-},1)_{\Omega_l\cap \Omega_i(t^{n-1})}
-\int_{I^n} \widehat{F}_{e_i} (u_h)  \ dt \right)= 0,
\end{align}
where \eqref{eq:vhat1} is obtained by multiplying \eqref{eq:stage1} with 1/2 and adding to equation~\eqref{eq:stage2}.   Thus, for the locally implicit scheme to be conservative, we need the numerical fluxes at the edges $e_i$,  between $\Omega_{i,E}$ and $\Omega_l$ (see Fig. \ref{LocalImplicit}), to satisfy
\begin{align}\label{eq:conservfluxatFi}
\int_{I^n} \widehat{F}_{e_i}(u_h) dt&=\frac{\Delta t^n}{2}\left(\widehat{F}_{e_i}(\hat{u}_{h,i}^{n-1})+\widehat{F}_{e_i}(\hat{u}_{h,i}^{(1)})\right), \quad i=1,2.
\end{align}

Taking into account boundary conditions, the flux's direction, as well as condition~\eqref{eq:conservfluxatFi}, we choose the fluxes at the domain's boundaries and edges $e_1,e_2$ as
\begin{align*}
\widehat{F}_L(\hat{u}_{h,1}^{n-1})&= F(g(t^{n-1})),
&& \widehat{F}_L(\hat{u}_{h,1}^{(1)})=F(g(t^{n-1}+\Delta t^n)), \\
\widehat{F}_R(\hat{u}_{h,2}^{n-1})&= F(\hat{u}_{h,i}^{n-1}(x_R,t^{n-1})),
&& \widehat{F}_R(\hat{u}_{h,2}^{(1)})=F(\hat{u}_{h,i}^{(1)}(x_R,t^{n-1})),\\
\widehat{F}_{e_1}(\hat{u}_{h,1}^{n-1})&={F}(\hat{u}_{h,1}^n(x_{e_1},t^{n-1})),
&& \widehat{F}_{e_1}(\hat{u}_{h,1}^{(1)})={F}(\hat{u}_{h,1}^{(1)}(x_{e_1},t^{n})), \\
\widehat{F}_{e_2}(\hat{u}_{h,2}^{n-1})&= F_{e_2}(u_h^{n-1,+}),
&&  \widehat{F}_{e_2}(\hat{u}_{h,2}^{(1)})=F_{e_2}(u_h^{n,-}),\\
\int_{I^n} \widehat{F}_{e_1}(u_h) dt&= \frac{\Delta t^n}{2}\left(\widehat{F}_{e_1}(\hat{u}_{h,1}^{n-1})
+\widehat{F}_{e_1}(\hat{u}_{h,1}^{(1)})\right),\\
\int_{I^n} \widehat{F}_{e_2}(u_h) dt &= \frac{\Delta t^n}{2}\left({F}_{e_2}(u_h^{n-1,+})+{F}_{e_2}(u_h^{n,-}) \right).
\end{align*}
The method can straightforwardly be used in a computation by time-stepping first in $\Omega_{1,E}$, then in $\Omega_l$, and finally in $\Omega_{2,E}$. The first and last steps are explicit, while the middle step is implicit.

\subsubsection{Numerical examples}
We now test the accuracy and conservation of the locally implicit method.  Consider first the same example as in Section~\ref{sec:ex1moving} with a time step size $\Delta t^n=\Delta t=h/12$. Table~\ref{table:localimplicit:accuracy} shows the expected second order accuracy in the $L^2$-norm. In $L^\infty$-norm  convergence is slower.
We have also solved this problem on finer meshes, $N=1280,2560,5120$, and we observe that the convergence is slower than the optimal second order convergence. 
 When a smaller time step is used the degradation of convergence rate occurs at a finer grid, indicating that the problem is related to the discretization of time. Since both the fully implicit and the standard method work well, we conclude that the problem comes from the coupling.
\begin{table}[!bhtp]
\caption{\label{table:localimplicit:accuracy} {
 Errors and orders of accuracy at $t=0.1$ for the  problem  in \ref{sec:ex1moving} with a moving interface when using the locally implicit method. The uniform  background mesh has $N$ elements in space.
} }
\begin{tabular}{ccccccc}
\hline\noalign{\smallskip}
N &{$L^{1} $} error & order &${L^2} $ error & order&${L^\infty} $ error & order\\
\noalign{\smallskip}\hline\noalign{\smallskip}
20	&	1,29E-01	&	-	&	1,46E-01	&	-	&	4,47E-01	&	-	\\
40	&	3,43E-02	&	1,91	&	3,91E-02	&	1,90	&	1,39E-01	&	1,68	\\
80	&	8,63E-03	&	1,99	&	1,00E-02	&	1,96	&	3,75E-02	&	1,89	\\
160	&	2,17E-03	&	1,99	&	2,56E-03	&	1,97	&	9,65E-03	&	1,96	\\
320	&	5,40E-04	&	2,01	&	6,43E-04	&	1,99	&	3,30E-03	&	1,55	\\
\hline\noalign{\smallskip}
\end{tabular}
\end{table}

Next we solve the example in Section~\ref{sec:ex2moving} and simulate this problem up to time $t=1$ with $N=400$ uniform elements in  the background mesh, and time step $\Delta t^n=h/12$.  We show the numerical solution in the left of Fig.~\ref{fig:move:localimplicit:p1p1}. We see that the locally implicit scheme can simulate this problem well and captures the discontinuity at the interface.
We measure the conservation error by $e(t)$
using the inflow information based on the numerical integration used in the scheme. We show the conservation error in the right of Fig.  \ref{fig:move:localimplicit:p1p1}. The conservation error is of the order of machine epsilon.
\begin{figure}[tbhp]
\begin{center}
\includegraphics[width=2.3in]{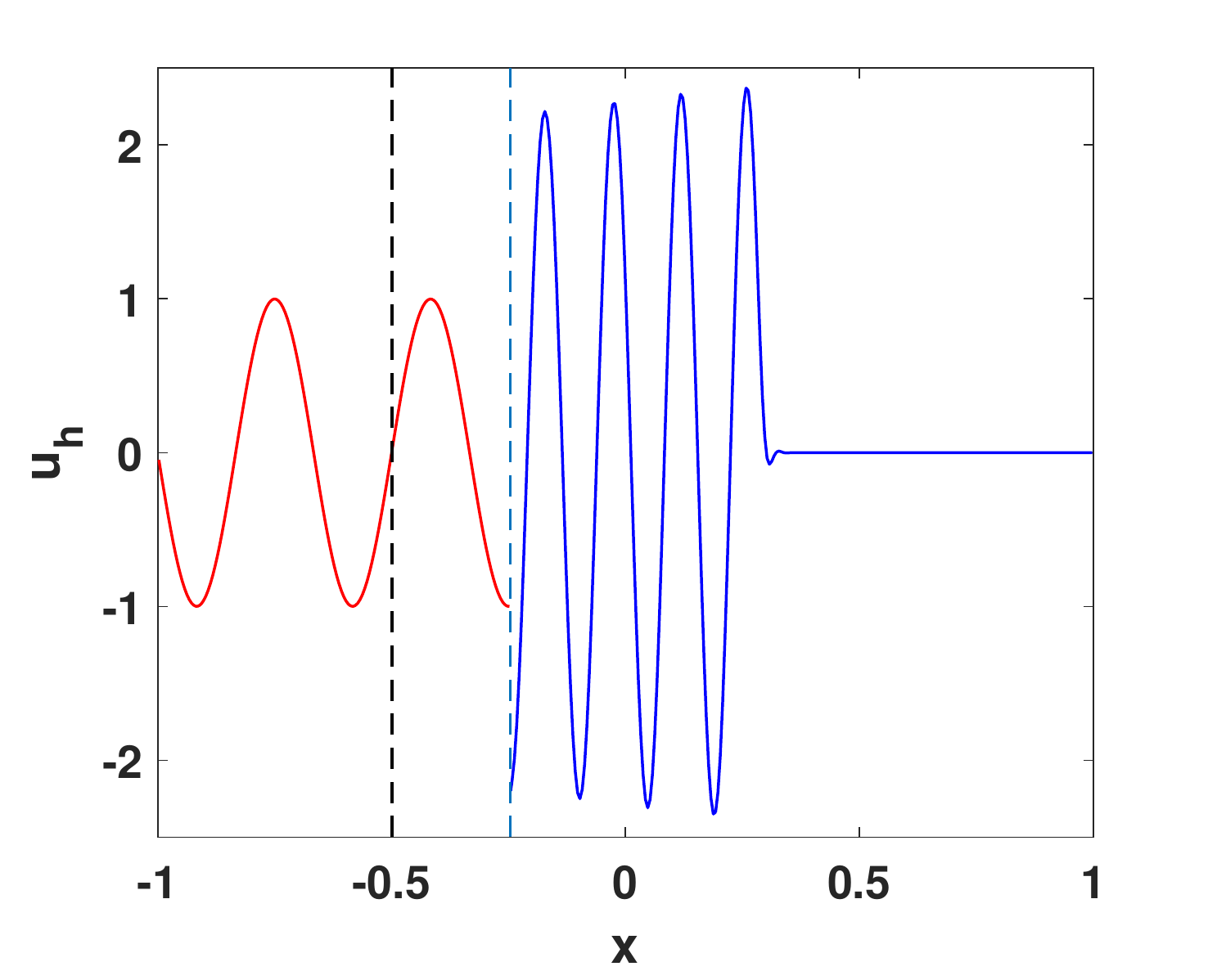}
\includegraphics[width=2.3in]{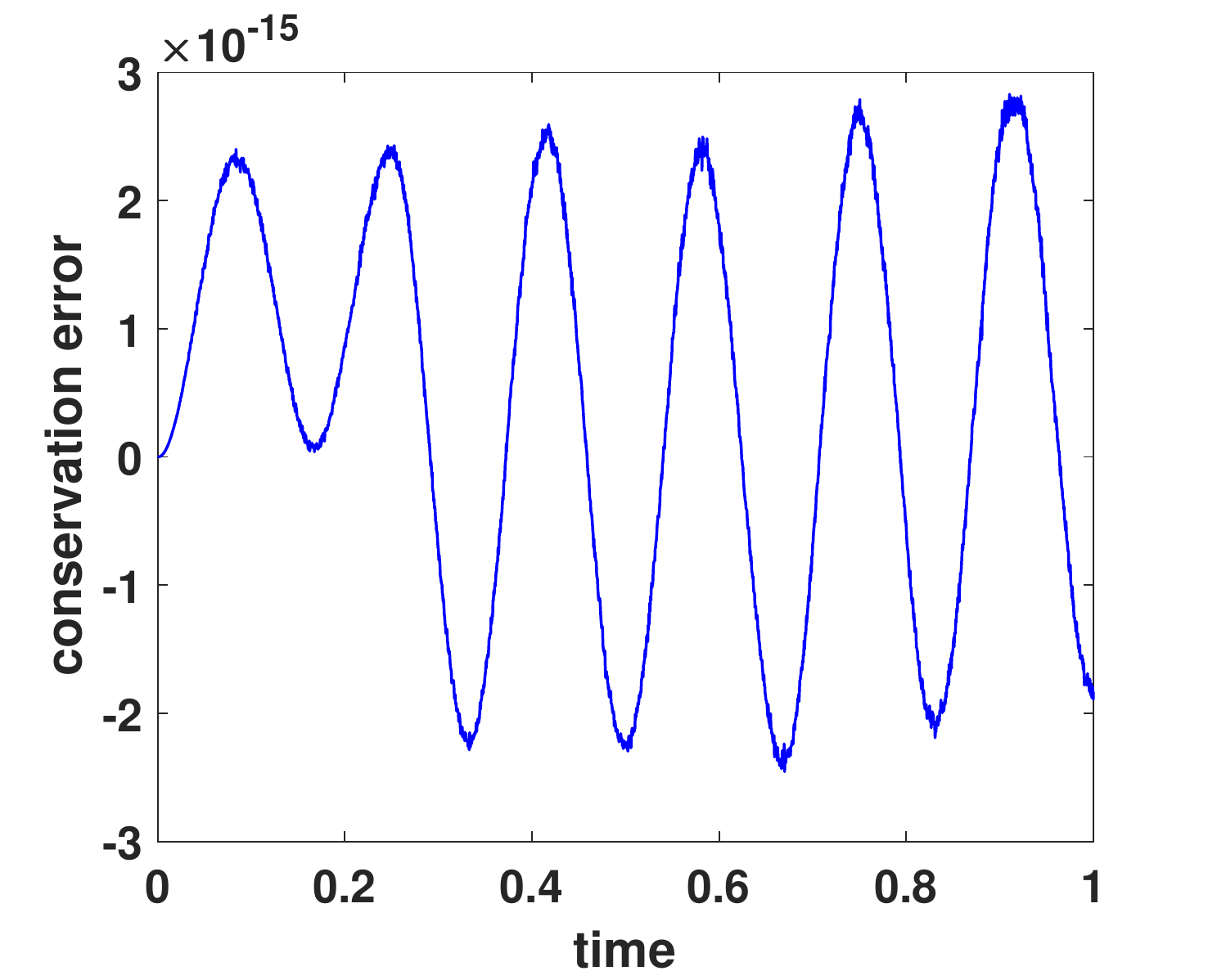}
\caption{Results for the  problem with a moving interface in Sect. \ref{sec:ex2moving}, discretized by the locally implicit method on a uniform background mesh with 400 elements in space. Left: solution $u_h$ at $t=1$. Right: conservation error $e(t)$.
}
\label{fig:move:localimplicit:p1p1}
\end{center}
\end{figure}

\section{Extension to two space dimensions}
\label{sec:2d}
Let $\Omega$ be a bounded convex domain in $\R^2$, with polygonal boundary $\partial \Omega$ and let $\Gamma$ be a smooth internal boundary that separates the domain $\Omega$ into two subdomains $\Omega_1$ and $\Omega_2$ such that $\bar{\Omega} = \bar{\Omega}_1 \cup \bar{\Omega}_2$.  Consider the hyperbolic conservation law
\begin{alignat}{2}
&u_{t}(x,y,t)+\div \F(u(x,y,t))=0, && \quad {(x,y) \in \Omega_1\cup \Omega_2,  \quad t \in (0,T]}  \label{eq:model2D} \\
&{u(x,y,0)=f(x,y),} &&  \quad   {(x,y) \in \Omega_1\cup \Omega_2,} \label{eq:initialcond2D} \\
&[\F(u(x,y,t))\cdot\n]_\Gamma=0,  &&  \quad   {(x,y) \in \Gamma},
\label{eq:interfacecond2D}
\end{alignat}
together with suitable boundary conditions.
Here $\n$ is the unit normal vector of $\Gamma$, $\F(u)=\mathbf{a} u$,  $\mathbf{a}=\mathbf{a}_1$ in $\Omega_1$ and $\mathbf{a}=\mathbf{a}_2$ in $\Omega_2$. Only a stationary interface is considered.

\subsection{The finite element method}
Let $\mt_h$ be a quasi-uniform simplicial mesh of the domain $\Omega$ generated independently of the position of the interface $\Gamma$ and let
$\widetilde{\V}$ be the finite element space on $\mt_h$ consisting of piecewise  polynomials of degree at most r. We define the active meshes $\mt_{h,i}$, 
the set of edges in each active mesh $\me_{h,i}$, and the set $\Fhi$ where the stabilization is applied, $i=1,2$,  as in Section \ref{sec:notapro} but now $I_j$ denotes a triangle in $\mt_h$. The active finite element spaces are
\begin{equation}
\Vi=\widetilde{\V}|_{\mt_{h,i}}, \, i=1,2.
\end{equation}
and we let $\V=\mathcal{V}_{h,1}^r \times \mathcal{V}_{h,2}^r$. 
Given the initial condition we find $u_h(0)=u_h(\cdot,0) \in \V$ such that
\begin{equation}
(u_h(x,y,0),v_h)_{\Omega_1\cup\Omega_2}+\gamma_MJ_1(u_h(x,y,0),v_h)=(f(x,y),v_h)_{\Omega_1\cup\Omega_2},  \quad \forall v_h \in \V.
\end{equation}
We propose the following weak formulation: Find $u_h(\cdot, t) \in \V$ such that for almost all $t \in (0,T]$
\begin{align}
\left( (u_h)_t,v_h \right)_{\Omega_1\cup\Omega_2}+\gamma_MJ_1((u_h)_t,v_h)+a_h(u_h,v_h)+\gamma_AJ_0(u_h,v_h)=0,
\end{align}
for $\forall v_h \in \V$.
Here
\begin{align}\label{scheme:ah2D}
a_h(u_h,v_h)=&- (\F(u_h), \nabla v_h)_{\Omega_1\cup\Omega_2} -\sum_{i=1}^{2}\sum_{e\in\me_{h,i}} (\{ \F(u_h) \cdot \n \}_e, [v_h]_e)_e-(\frac{\lambda_e}{2}[u_h]_e , [v_h]_e)_e \nonumber \\
&-\int_\Gamma \left( [\F(u_h) \cdot \n v_h]_\Gamma+[\F(u_h)]_\Gamma[\lambda v_h]_\Gamma \right) \ ds,
\end{align}
and
\begin{equation}\label{stab2D}
J_s(u_h, v_h)=\sum_{i=1}^2\sum_{e \in \Fhi} \sum_{k=0}^{r} \omega_kh^{2 k+s}\left( \left[\partial^ku_{h,i}\right]_e,\left[\partial^{k} v_{h,i}\right]_{e} \right)_e.	
\end{equation}

\subsection{Numerical example}
Let $\Omega=[x_{\min}, x_{\max}] \times [y_{\min}, y_{\max}]$ and the interface $\Gamma$ be the line $x+y=c_0$ where $c_0 $ is a constant. The subdomain $\Omega_1=\{ (x,y) \in \Omega: x+y\leq c_0\}$, $\Omega_2=\{ (x,y) \in \Omega: x+y\geq c_0\}$. We choose $x_{\min}=y_{\min}=-1$, $x_{\max}=y_{\max}=1$, and $c_0 > x_{\max}+y_{\min}$.
The unit normal to $\Gamma$ is  $\n=\frac{1}{\sqrt{2}}(1,1)$ and the time step is chosen as $\Delta t =  \frac{0.5h}{ (2r+1)\max_\Omega(|\mathbf{a}|)}$ .
The boundary conditions are
\begin{align}
&u(x_{\min},y,t)=g(x_{\min},y,t),\\
&u(x,y_{\min},t)=g(x,y_{\min},t),\\
&\text{outflow BC on the remaining boundaries.}
\end{align}

\subsubsection{Convergence study} \label{convergence study 2D}
Let $\mathbf{a}_1=(3,1)$ and $\mathbf{a}_2=(2,1)$ and set $c_0=0.5$. A solution to equation \eqref{eq:model2D} is
\begin{align}\label{eq: example2D1}
&u_1(x,y,t)=\sin(\pi(x+y-4t)), \quad (x,y)\in \Omega_1,\\
\label{eq: example2D2}
&u_2(x,y,t)=\frac{4}{3}\sin(4/3\pi(x+y-3t-c_0/4)), \quad (x,y)\in \Omega_2,
\end{align}
and with $g = u_1$ it satisfies the boundary conditions.
Note that on the interface, $x+y=c_0$, we have $u_1=\sin(\pi(c_0-4t))$ and $u_2=\frac{4}{3}\sin(4/3\pi(3/4c_0-3t))$. Thus, the solution given by Eq.  \eqref{eq: example2D1}-\eqref{eq: example2D2} satisfies the interface condition \eqref{eq:interfacecond2D}.

We solve the problem on a uniform mesh until $t=1$ with a time step $\Delta t =0.5 h / ((2r+1) \sqrt{10})$ for $r=1,2$, (i.e. P1, P2 elements) and $h = 2/N_x$, with $N_x=20,40,80,160,300$. In Fig. \ref{fig:convergence2D} we show that the $L^2$-error versus mesh size $h$. The convergence order of the method follows the optimal order $r+1$.

\begin{figure}[h] 	
	\centering	
	\includegraphics[width=2.8in]{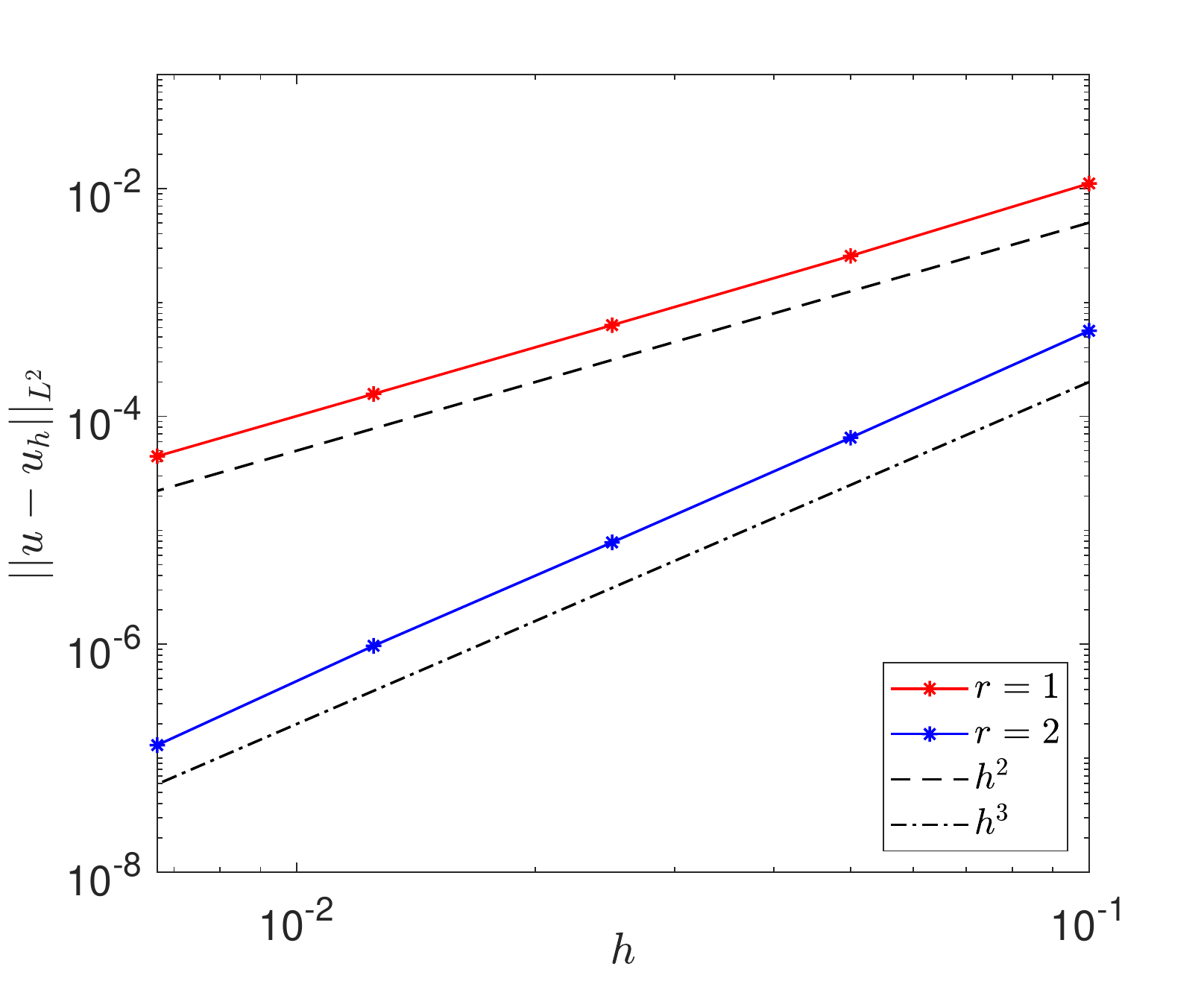}
	\caption{The $L^2$-error at t = 1 versus mesh size $h$ for the problem in Sect.  5.2.1 in two space dimensions. Polynomials of degree  r=1, 2 are used and convergence order r+1 is obtained.
 \label{fig:convergence2D}}
\end{figure}

\subsubsection{Conservation study}
Let $\mathbf{a}_1=(3,1)$ and $\mathbf{a}_2=(1,2)$ and set $c_0=0.25$ and denote $\mathcal{C}$ the circle with center $(-0.3,-0.3)$ and radius $0.3$. We consider the initial condition
\begin{align}
f(x,y)= & \left\{ \begin{array}{ll}
         1 & \mbox{if $(x,y) \in \mathcal{C}$,}\\
          0 & \mbox{else},
        \end{array} \right.
  \nonumber
\end{align}
and boundary data $g=0$.
We solve this problem using the proposed scheme with $r=1$ on a uniform mesh where $N_x = N_y = 200$ and a time step defined as above.
We use two different sets of  penalty parameters $\lambda_1$ and $\lambda_2$ where the first set satisfies the conservation condition \eqref{eq:conservcondlamb} while the second set does not.
In Fig. \ref{fig:conservation2Dvizu} we show the numerical solution at different time instances with $\lambda_i$, $i=1,2$ satisfying the conservation condition. The mass concentrated in the circle $\mathcal{C}$ is away from the interface initially, but evolves and passes through the interface.
In Fig. \ref{fig:conservation2D}  we show that when penalty parameters satisfy \eqref{eq:conservcondlamb}, the method is conservative. In contrast, if the conservation condition is not satisfied, the conservation error can be large and increases significantly when the part with a mass reaches and passes through the interface. Note that the exact solution is not smooth, but with condition \eqref{eq:conservcondlamb} the method is still conservative.

\begin{figure}[h]
    \subfigure[$t=0$]
    	{
    	\scalebox{0.4}    	
	\centering	
	\includegraphics[width=1.5in]{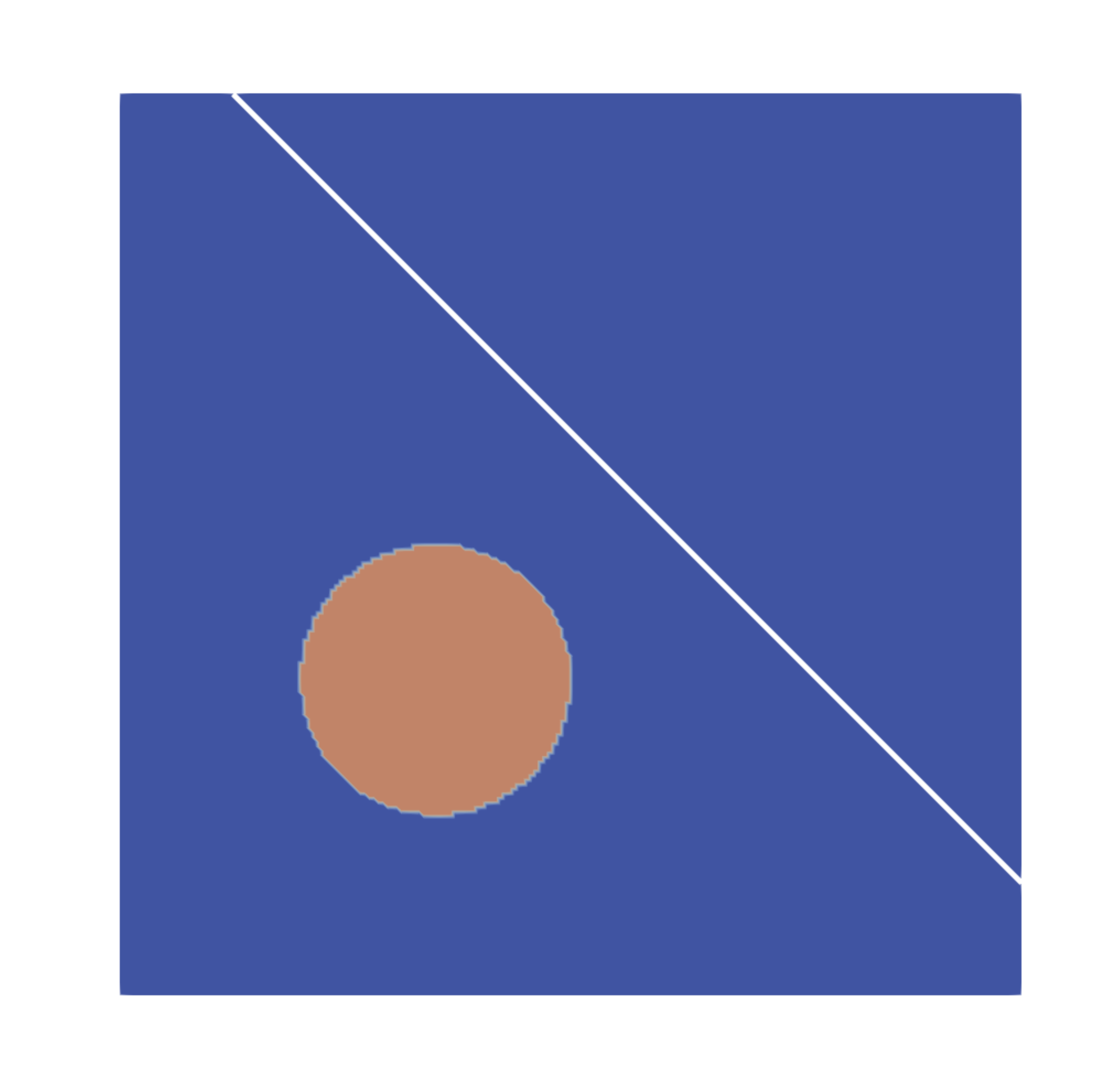}
        }
          \subfigure[$t=0.2$]
    	{
    	\scalebox{0.4}    	
	\centering	
	\includegraphics[width=1.5in]{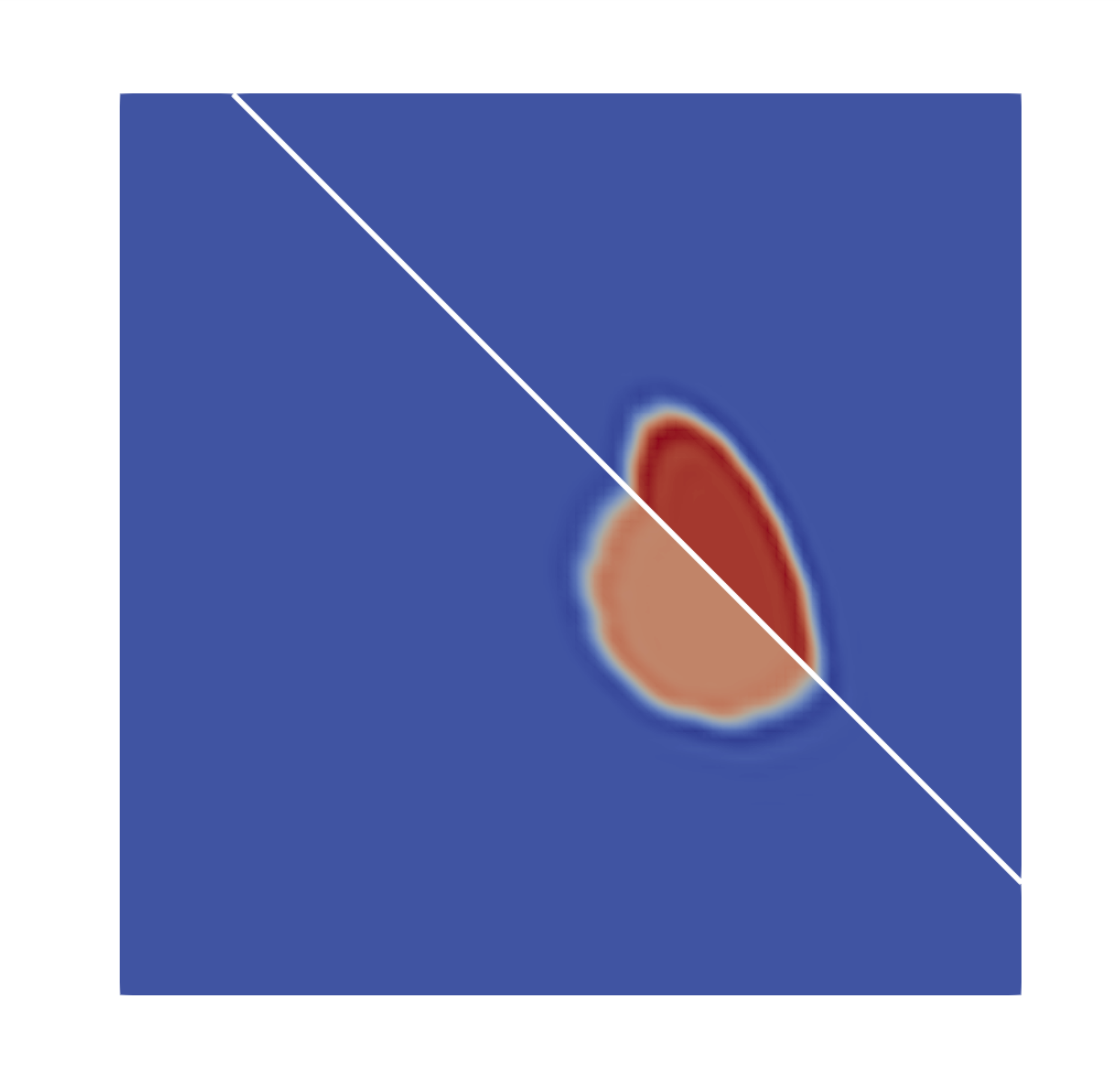}
        }
        \subfigure[$t=0.4$]
    	{
    	\scalebox{0.4}    	
	\centering	
	\includegraphics[width=1.5in]{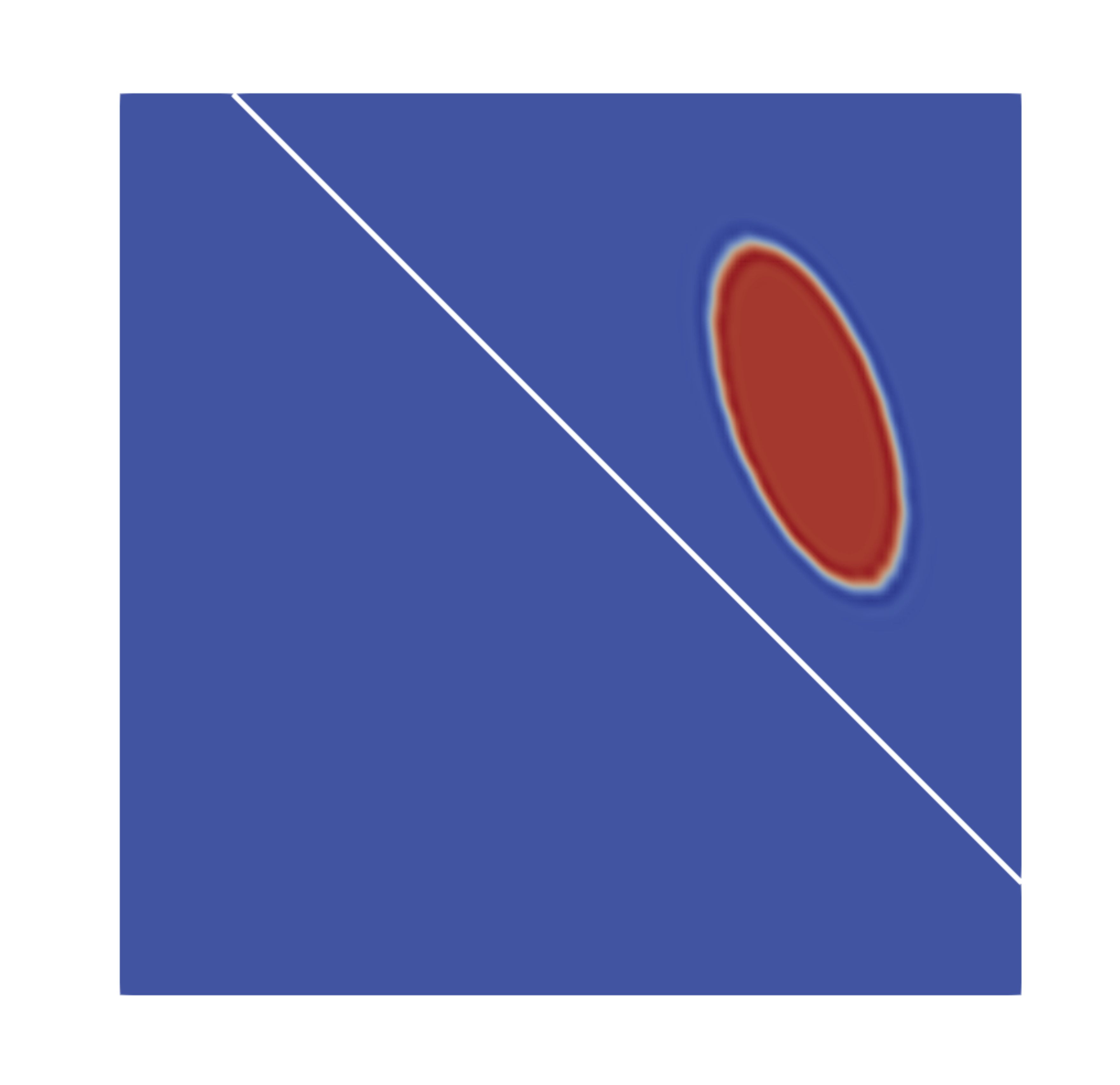}
        }      	
        \caption{Solution obtained at different time instances for $r=1$. Here we have chosen $\lambda_1 = 0$ and $\lambda_2 = -1$. The white line is the interface $\Gamma$. Blue corresponds to $u=0$, orange to $u=1$ and red to $u=4/3$.}
        \label{fig:conservation2Dvizu}
\end{figure}

\begin{figure}[h] 	
	\centering	
	\includegraphics[width=2.8in]{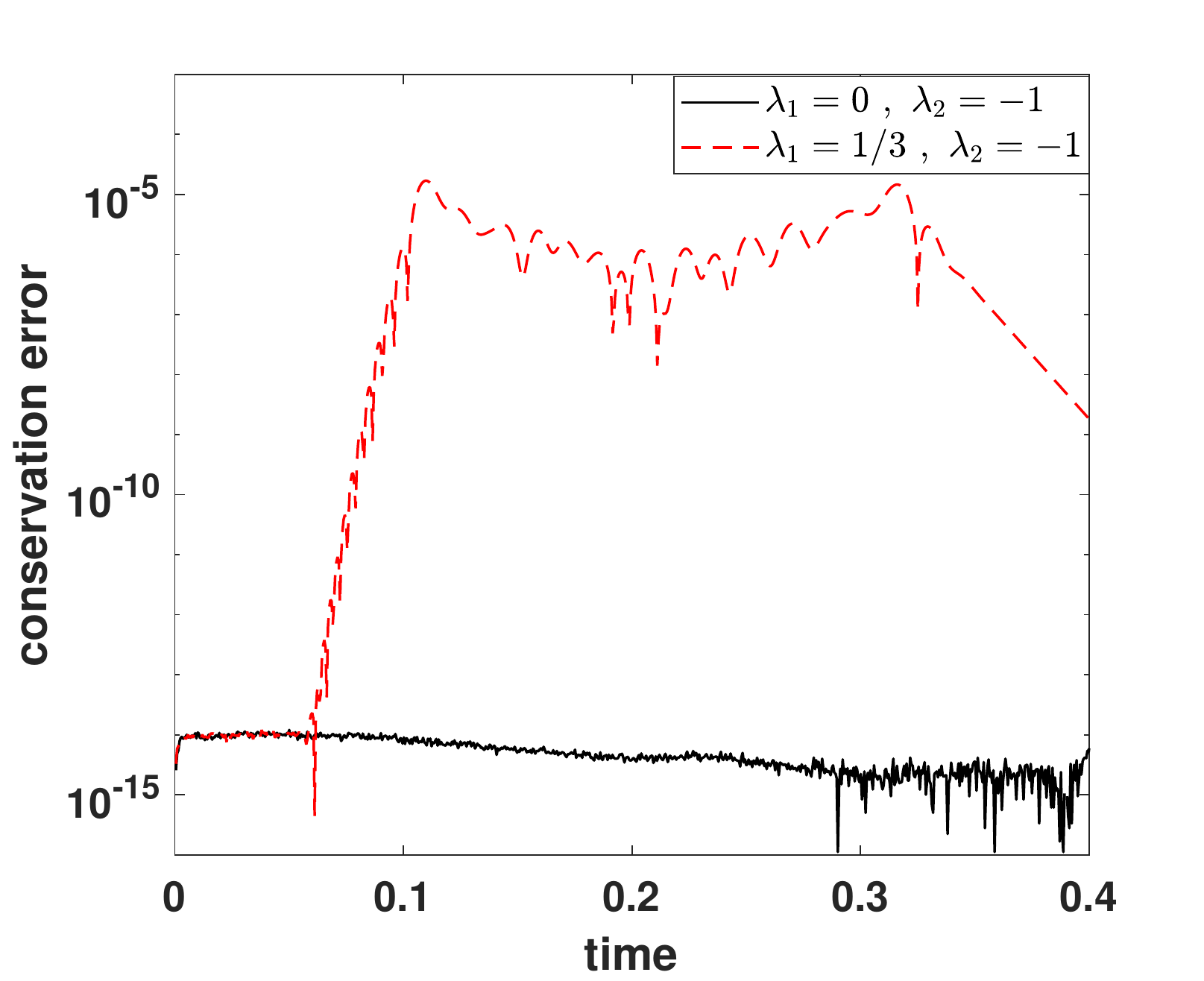}
	\caption{Comparison of the conservation error for different choices of $\lambda_i$.\label{fig:conservation2D}}
\end{figure}

\section{Conclusion}
\label{sec:conclusion}

We have presented two high order CutFEM based on the DG framework, applicable to conservation laws with discontinuous coefficients in the flux across stationary and moving interfaces, respectively. Our methods use standard DG-elements, but do not require the elements to be aligned with interfaces where coefficients are discontinuous. Ghost penalty stabilization is included in the weak forms to allow for similar time-step restrictions as in the standard DG approach.
We have established discrete conservation, accuracy and stability for the methods. The proposed methods are described and analyzed in one-dimensional settings, but we also present computations in two dimensions for a stationary interface case, which demonstrates that the methodology can directly be extended to higher dimension.

The first method is based on a method of lines approach, and is an extension of the method in ~\cite{fu2021high} to handle stationary material interfaces.  The mass matrix appearing in the semi-discrete system is  block-diagonal with most blocks as in the standard DG approach, but with one larger diagonal block, which couples the degrees of freedom associated with elements in the vicinity of the interface. The non-diagonal entries are caused partly by the contributions from integrals in the cut region,   
 and partly by the ghost penalty stabilization, which couples elements cut by the interface with their neighbours.  In one space dimension this coupling is not a big issue. However, one can produce a block-diagonal matrix with less coupling, in particular in higher dimension, by applying stabilization restrictively as proposed in~\cite{larson2021conservative}.
Our focus in this work has been on the interface treatment
and we developed CutFEM that are globally conservative. By changing the stabilization to the macro element stabilization in~\cite{larson2021conservative} our method would also preserve the local conservation property of the discontinuous Galerkin formulation on macro elements.

The second method is for moving interfaces. It is a space-time CutFEM based on discontinuous elements in both space and time. The method is stable and conservative, but implicit. In particular we have shown, numerically and analytically, that using a weak form based on integration by parts in time is essential for discrete conservation. Since we are using standard DG techniques as building blocks we believe that the extension of this method to multiple space dimensions is also straightforward. The implicit character of our space-time CutFEM is however a drawback.
We demonstrate in a scalar case how a more efficient method can be achieved by using the space-time elements only locally in the vicinity of interfaces. This idea moves the difficulty from the non-aligned moving interface to a stationary aligned interface, where the space-time elements need to be coupled to standard method of lines DG methods while maintaining stability, accuracy and conservation.

In computations we have observed that the temporal accuracy is sometimes degraded in the coupled case, and more work is required to understand and avoid this degradation. We also believe that a more difficult extension is to generalize the coupling between space-time elements and the standard method-of lines DG methods to systems and to multi-dimensions, while maintaining stability, accuracy and conservation, and allowing for explicit time-stepping in large parts of the domain.

\appendix
\begin{appendix}
\section{Positivity of the ${S}$-matrix}\label{appendix:proof}
In this appendix,  we show that under the conditions in Theorem \ref{thm:stabilityscalarp} the matrix $S$ in \eqref{eq:matrixAs} is positive semi-definite.
\begin{lemma}\label{lemma:matrixA}
Consider
\begin{align}
S=\left(
\begin{array}{cc}
   (\frac{1}{2}-\lambda_1)a_1 & \frac{a_2\lambda_1+ a_1\eta \lambda_2}{2}   \\
\frac{a_2\lambda_1+a_1\eta\lambda_2}{2}  & -(\lambda_2+\frac{1}{2})\eta a_2
   \end{array}\right),
   \label{def:energy:S}
   \end{align}
 where $\lambda_1=\lambda_2+1$ and $a_1,a_2$ are either strictly positive or strictly negative. There exists positive $\eta$ such that the matrix $S$ is positive semi-definite under the condition
\begin{align}\label{eq:state:parametercondition}
\left\{\begin{array}{ll}
{\lambda_1\leq\frac{1}{2},\lambda_2\leq-\frac{1}{2},}&{\text{if } a_1>0,a_2>0,}\\
{\lambda_1\geq\frac{1}{2},\lambda_2\geq-\frac{1}{2},}&{\text{if } a_1<0,a_2<0.}\\
\end{array}\right.
\end{align}
\end{lemma}
\begin{proof}
To reduce the number of parameters we rewrite
 the conservation condition \eqref{eq:conservcondlamb} as
\begin{equation}\label{eq:flux2}
\lambda_2+\frac{1}{2}=\lambda_1-\frac{1}{2},
\end{equation}
and introduce
\begin{equation}\label{eq:flux2}
\sigma\equiv-\lambda_2-\frac{1}{2}=-\lambda_1+\frac{1}{2}.
\end{equation}
In terms of $\sigma$ we have
\begin{equation}
S=\left(
\begin{array}{cc}
   a_1\sigma & -\frac{a_2+\eta a_1}{2}\sigma+\frac{a_2-\eta a_1}{4}   \\
-\frac{a_2+\eta a_1}{2}\sigma+\frac{a_2-\eta a_1}{4}  & a_2\eta\sigma
   \end{array}
\right).
\end{equation}
To investigate if $S$ is positive semi-definite we study the eigenvalues of $S$. 
The eigenvalues, $\theta$, satisfy
\begin{align}
|S-\theta I|
&=\theta^2-tr(S)\theta+\det(S)=0,
\end{align}
where $tr(S)=a_2\sigma(\beta+\eta)$ with $\beta=\frac{a_1}{a_2}$, and
\begin{align}
\det(S)
&=\frac{a_2^2}{4}(1-\beta\eta)\left( (1+\beta\eta){\sigma}-(1-\beta\eta)\left({\sigma^2}+\frac{1}{4}\right) \right).
\end{align}
By assumption $\beta,\eta>0$. We have
\begin{equation}
\theta_{1,2}=\frac{a_2}{2}\sigma(\beta+\eta)\pm\sqrt{\frac{a_2^2\sigma^2(\beta+\eta)^2}{4}-\det(S)}.
\end{equation}
Both eigenvalues are nonnegative precisely if equivalently
\begin{equation} \label{eq:conddetA}
a_2\sigma\geq 0 \text{ and } \det(S)\geq 0.
\end{equation}

When $\sigma>0$ we have from \eqref{eq:conddetA} that  $a_2>0$, and
\begin{align}
0\leq 1-\eta \beta \leq \frac{2\sigma}{(\sigma+1/2)^2},
\end{align}
or equivalently $1-\frac{2\sigma}{(\sigma+1/2)^2}\leq\eta \beta\leq 1$. Hence,  given $a_1>0, a_2>0, \beta=a_1/a_2>0$ and parameters $\lambda_1<1/2, \lambda_2<-1/2$ with $\lambda_2-\lambda_1+1=0$,  matrix $S$ is positive semi-definite under the condition
\begin{align}\label{eq:energy:gamma1}
\frac{(\lambda_2+1)^2}{\lambda_2^2}\leq \eta \beta\leq 1.
\end{align}

When $\sigma<0$ we have from \eqref{eq:conddetA} that $a_2<0$ and
\begin{align}
0\geq 1-\eta \beta \geq \frac{2\sigma}{(\sigma+1/2)^2},
\end{align}
or equivalently $1-\frac{2\sigma}{(\sigma+1/2)^2}\geq\eta \beta\geq 1$. This means, given $a_1<0,a_2<0, \beta=a_1/a_2>0$, and parameters $\lambda_1>1/2,\lambda_2>-1/2$ with $\lambda_2-\lambda_1+1=0$,  matrix $S$ is  positive semi-definite when
\begin{align}\label{eq:energy:gamma2}
\frac{(\lambda_2+1)^2}{\lambda_2^2}\geq \eta \beta\geq 1.
\end{align}
When $\sigma =0$, that is $\lambda_1=1/2$ and $\lambda_2=-1/2$, we need to have $\det(S)=0$ in order for \eqref{eq:conddetA} to be satisfied.  We have that $\det(S)=0$ if $\eta \beta=1$ for both $a_i>0$ and $a_i<0$.
Thus, there always exist positive $\eta$ under the condition \eqref{eq:state:parametercondition} such that $S$ is positive semi-definite and the energy  $E_\eta$ is non-increasing.
\end{proof}

Next, we will show a stability condition of $\lambda_i$ in the scheme \eqref{scheme:state:DG} without the conservation condition \eqref{eq:conservcondlamb}.
\begin{lemma}\label{lemma:matrix:stable}
Consider
\begin{align}
S=\left(
\begin{array}{cc}
   (\frac{1}{2}-\lambda_1)a_1 & \frac{a_2\lambda_1+ a_1\eta \lambda_2}{2}   \\
\frac{a_2\lambda_1+a_1\eta\lambda_2}{2}  & -(\lambda_2+\frac{1}{2})\eta a_2
   \end{array}\right),
   \label{def:energy:S}
   \end{align}
 where $a_1,a_2$ are either strictly positive or strictly negative. There exists positive $\eta$ such that the matrix $S$ is positive semi-definite under the condition $\lambda_1-\lambda_2\geq\frac{1}{2}$ and
\begin{align}\label{eq:state:parametercondition}
\left\{\begin{array}{ll}
{\lambda_1\leq\frac{1}{2},\lambda_2\leq-\frac{1}{2},}&{\text{if } a_1>0,a_2>0,}\\
{\lambda_1\geq\frac{1}{2},\lambda_2\geq-\frac{1}{2},}&{\text{if } a_1<0,a_2<0.}\\
\end{array}\right.
\end{align}
\end{lemma}

\begin{proof}
We will investigate when $S$ is a positive semi-definite matrix by studying the eigenvalues of S. Without the condition $\lambda_2-\lambda_1+1=0$, similarly as the above proof, we have $tr(S)=\left(\frac{1}{2}-\lambda_1\right)a_1-\left(\lambda_2+\frac{1}{2}\right)a_2\eta$ and
\begin{align}
\text{det}(S)=\left(\lambda_1-\frac{1}{2}\right)\left(\lambda_2+\frac{1}{2}\right)a_1a_2\eta-\frac{(a_2\lambda_1+a_1\eta\lambda_2)^2}{4}.
\end{align}
The eigenvalues are nonnegative if we have $tr(S)\geq0,\det(S)\geq 0$.
By assumption $\beta=\frac{a_1}{a_2}>0$, if $a_2>0$, $tr(S)\geq0,\det(S)\geq 0$ are equal to
\begin{align}\label{ieq:trs:pos}
\left(\frac{1}{2}-\lambda_1\right)\beta\geq\left(\lambda_2+\frac{1}{2}\right)\eta,
\end{align}
and
\begin{align}\label{ieq:lemma2:delta}
2\left(\lambda_1-\lambda_2-\frac{1}{2}\right)\beta\eta\geq (\lambda_1-\beta\eta\lambda_2)^2 \Longleftrightarrow\lambda_2^2\beta^2\eta^2-2\beta\eta\left(\lambda_1-\lambda_2-\frac{1}{2}+\lambda_1\lambda_2\right)+\lambda_1^2\leq0.
\end{align}
With $\lambda_1\leq\frac{1}{2},\lambda_2\leq-\frac{1}{2}$, the inequality \eqref{ieq:trs:pos} holds for any $\eta>0$.
Therefore, to ensure that there exists $\eta>0$ such that \eqref{ieq:lemma2:delta} is satisfied, we need
\begin{align}
\Delta:=\left(\lambda_1-\lambda_2-\frac{1}{2}\right)\left(\lambda_1-\lambda_2-\frac{1}{2}+2\lambda_1\lambda_2\right)
=\left(\lambda_1-\lambda_2-\frac{1}{2}\right)\left(2\lambda_1-1\right)\left(\lambda_2+\frac{1}{2}\right)\geq0,
\end{align}
With $\lambda_1\leq\frac{1}{2},\lambda_2\leq\frac{1}{2}$ and $\lambda_1-\lambda_2\geq\frac{1}{2}$, it follows that $\Delta\geq 0$. Then there exits an $\eta>0$ with
\begin{align}
\max\left\{0,\frac{\left(\lambda_1-\lambda_2-\frac{1}{2}+\lambda_1\lambda_2\right)-\sqrt{\Delta}}{\lambda_2^2}
\right\}\leq\beta\eta
\leq\frac{\left(\lambda_1-\lambda_2-\frac{1}{2}+\lambda_1\lambda_2\right)+\sqrt{\Delta}}{\lambda_2^2}.
\end{align}
This shows that when $a_1,a_2>0$ a sufficient condition for the scheme to be stable are $\lambda_1\leq \frac{1}{2}$, $\lambda_2\leq -\frac{1}{2}$, and $\lambda_1-\lambda_2\geq \frac{1}{2}$.

If $a_1,a_2<0$, we first need  $tr(S)\leq0$, that is
\begin{align}
\left(\frac{1}{2}-\lambda_1\right)\beta\leq\left(\lambda_2+\frac{1}{2}\right)\eta.
\end{align}
Then, we need $\det(S)\geq0$  as in \eqref{ieq:lemma2:delta}.
Thus, we need
\begin{align}
\left(\lambda_1-\lambda_2-\frac{1}{2}\right)
\left(2\lambda_1-1\right)\left(\lambda_2+\frac{1}{2}\right)\geq0.
\end{align}
If $\lambda_1\geq\frac{1}{2},\lambda_2\geq-\frac{1}{2}$ and $\lambda_1-\lambda_2-\frac{1}{2}\geq0$, there always exists $\eta>0$ such that
\begin{align}
\max\left\{0,\frac{\left(\lambda_1-\lambda_2-\frac{1}{2}+\lambda_1\lambda_2\right)-\sqrt{\Delta}}{\lambda_2^2}
\right\}\leq\beta\eta\leq\frac{\left(\lambda_1-\lambda_2-\frac{1}{2}+\lambda_1\lambda_2\right)+\sqrt{\Delta}}{\lambda_2^2},
\end{align}
and the scheme is stable.

\end{proof}

\section{Proof of Theorem 2}\label{proofoferror}
In this appendix, we give the proof of the error estimate in Theorem 2 but first we briefly recall some useful inequalities. For $s \geq0$, let $\|\cdot\|_{s, \omega}$ and $|\cdot |_{s, \omega}$ denote the usual norm and semi-norm of Sobolev space $H^{s}(\omega)$, respectively and $\|\cdot\|_{s,\mt_{h}}^{2}=\sum_{T \in \mathcal{T}_{h}}\|\cdot\|_{s, T}^{2}$. For $s=0$, the norm $\|\cdot\|_{0, \omega}$ and  $|\cdot |_{0, \omega}$ is the standard $L^2$-norm and we often use the notation $\| \cdot \|_\omega$.
We will use the notation $a \lesssim b \Leftrightarrow a \leq C b$ with $C$ being some constant that is independent of the mesh parameter $h$ and how the interface cuts the mesh  $\mt_h$.

Let $T \in \mathcal{T}_{h}$, the following trace inequalities hold
\begin{align}
\|v\|_{\partial T} & \lesssim h^{-1 / 2}\|v\|_{T}+h^{1 / 2}| v |_{1,T},  \ \forall   v \in H^{1}(T),  \label{eq:traceH1}
\\
\|v\|_{\Gamma \cap T} & \lesssim h^{-1 / 2}\|v\|_{T}+h^{1 / 2} | v |_{1,T}, \ \forall  v \in H^{1}(T), \label{eq:traceH1G}
\end{align}
where the first is a standard trace inequality and the second is proven in \cite{hansbo2003finite}. We also have
the following inverse inequality \cite{brenner2008mathematical},
\begin{align}\label{inverse:discrete1}
\left |v_{h}\right | _{j,T} & \lesssim h^{s-j}\left | v_{h}\right |_{s,T}, \ \forall v_{h} \in \V,  \ 0 \leq s \leq j, \quad T \in \mathcal{T}_{h}.
\end{align}
Following \cite{brezzi2004discontinuous} we use the $L^{2}$-projection. For $i=1,2$, let $\pi_{h,i}: L^{2}\left(\mt_{h,i} \right) \rightarrow \Vi$ denote the $L^{2}$-projection onto $\Vi$.  For all $v_i \in H^{r+1}\left(\mt_{h,i} \right)$ we recall the following standard estimate
\begin{align}\label{interpolation}
\left\|v_i-\pi_{h,i} v_i \right\|_{k,T} &\lesssim h^{r+1-k}\|v_i\|_{r+1, T}, \     T \in \mt_{h,i},
\end{align}
where $k=0, 1, \cdots, r+1$.  
We also need extension operators that can extend functions defined in $\Omega_i$ to $\mt_{h,i}$. The extension theorem in \cite{stein2016singular} provides continuous extension operators $E_i: H^{s}(\Omega_i)  \rightarrow H^{s}\left( \R ^d \right)$ such that for all  $v_i \in H^{s}(\Omega_i)$, $E_i v_i |_{\Omega_i} = v_i$ and
\begin{equation}\label{eq:stabextension}
\left\| E_i v_i \right\|_{s, \mathbb{R}^{d}} \lesssim \|v_i\|_{s, \Omega_i}, \quad  i=1,2, \ s\geq0.
\end{equation}
We now define an extension operator $E$ such that for all $v \in \oplus_{i=1}^2 H^s(\Omega_i)$,  $Ev= (E_1 v_1, E_2 v_2)$ and we denote $Ev$ by $v^e$.  Using the $L^{2}$-projection $\pi_{h,i}$ and the extension operators we define the following projection operator
\begin{equation}
\pi_h:\oplus_{i=1}^2 L^2(\mt_{h,i} ) \ni (v_1,v_2) \mapsto (\pi_{h,1} E_1 v_1, \pi_{h,2} E_2 v_2) \in \V.
\end{equation}
We are now ready to prove Theorem 2.
\begin{proof}
Note that $u=u_i$ for $x \in \Omega_i$, is the solution to problem \eqref{eq:model}-\eqref{eq:interfacecond} with $F(u)$ as in \eqref{eq:linear:state} and $x_\Gamma'(t)=0$, and is sufficiently smooth: $u_i \in L^{\infty}\left([0,T] ; H^{r+1}({\Omega_i}) \right)$ and $(u_i)_{t} \in L^{\infty}([0,T] ;$ $ {H^{r+1}({\Omega_i})}  )$.  For $t\in (0,T]$ we have from consistency that $u^e(t)=u^e(\cdot,t)$, the extension of exact solution $u$, satisfies
\begin{align}\label{scheme:state:u}
\left( u^e_t,v_h \right)_{\Omega_1\cup\Omega_2}+\gamma_MJ_1(u^e_t,v_h)+A_h(u^e,v_h)=0, \, \forall v_h\in\V.
\end{align}
Further, $u_h(t)=u_h(\cdot, t)\in\V$, is the solution to \eqref{scheme:state:DG}. Subtracting \eqref{scheme:state:u} from \eqref{scheme:state:DG}, we get the error equation
\begin{align}
\label{scheme:state:erroreq}
\left( (u^e-u_h)_t,v_h \right)_{\Omega_1\cup\Omega_2}+\gamma_M J_1((u^e-u_h)_t,v_h)+A_h(u^e-u_h,v_h)=0, \, \forall v_h\in\V.
\end{align}
We write the error as a sum of two terms $u^e-u_h=\xi+\zeta$,  where
$\xi=(\xi_1,\xi_2)$, with $\xi_i(x,t)=u^e_i(x,t)-\pi_{h,i}E_i u_{i}(x,t)$ and $\zeta=(\zeta_1,\zeta_2)$ with $\zeta_i(x,t)=(\pi_{h,i} E_i u_{i}(x,t)-u_{h,i}(x,t))$ $\in\Vi$, $i=1,2$. Rewriting the error equation \eqref{scheme:state:erroreq} in terms of $\xi$ and $\zeta$ we get
\begin{align}\label{scheme:state:erroreq1}
\left( (\xi+\zeta)_t,v_h \right)_{\Omega_1\cup\Omega_2}+\gamma_MJ_1((\xi+\zeta)_t,v_h)+A_h(\xi+\zeta,v_h)=0, \, \forall v_h\in\V.
\end{align}
Let $v_h=(\zeta_1,\eta \zeta_2)$  in \eqref{scheme:state:erroreq1}. Defining a weighted energy similar to \eqref{eq:define:energy}, i.e.,
$$E_{\eta}^\zeta(t)= \frac{1}{2} \left(||\zeta_1||^2_{\Omega_1}+\gamma_M J_1(\zeta_1,\zeta_1)\right)+\frac{\eta}{2}\left(||\zeta_2||^2_{\Omega_2}+\gamma_MJ_1(\zeta_2,\zeta_2)\right),$$
and following the stability analysis in Section \ref{sec:stability} we get
\begin{align}
\frac{d}{dt}E_{\eta}^\zeta(t)=&
-\mathbf{\zeta_\Gamma}^T S \mathbf{\zeta_\Gamma}-\sum_{e\in\me_{h,1}}\frac{|a_1|}{2}[\zeta_1]_e^2-\sum_{e\in\me_{h,2}}\frac{|a_2|\eta}{2}[\zeta_2]_e^2 -\gamma_A J_0(\zeta_1,\zeta_1)-\eta \gamma_A J_0(\zeta_2,\zeta_2)\notag\\
&-\left(\int_{\Omega_1} (\xi_{1})_t\zeta_1 dx + \gamma_M J_1((\xi_{1})_t,\zeta_1) \right)-{\eta} \left(\int_{\Omega_2}(\xi_2)_t\zeta_2 dx + \gamma_M J_1((\xi_2)_t,\zeta_2) \right) \notag\\
&-\gamma_A J_0(\xi_1,\zeta_1)-\eta \gamma_A J_0(\xi_2,\zeta_2)\notag\\
&+\int_{\Omega_1}a_1\xi_1(\zeta_1)_xdx+\sum_{e\in\me_{h,1}}a_1\xi_{1,e}^{-}[\zeta_1]_e
-(1-\lambda_1)a_1\xi_{1,\Gamma}\zeta_{1,\Gamma}
-\lambda_1a_2\xi_{2,\Gamma}\zeta_{1,\Gamma}\notag\\
&+\eta\left(\int_{\Omega_2}a_2\xi_2(\zeta_2)_xdx+\sum_{e\in\me_{h,2}}a_2\xi_{2,e}^{-}[\zeta_2]_e
+(1+\lambda_2)a_2\xi_{2,\Gamma}\zeta_{2,\Gamma}
-\lambda_2a_1\xi_{1,\Gamma}\zeta_{2,\Gamma}\right).
\label{scheme:state:erroreq2}
\end{align}
Here, the matrix $S$ is as in \eqref{eq:matrixAs}, $\xi_{i,\Gamma}=\xi_i(x_\Gamma,t)$, $\zeta_{i,\Gamma}=\zeta_i(x_\Gamma,t)$, and
$\mathbf{\zeta_\Gamma}= \left(\begin{array}{c} \zeta_1(x_\Gamma,t)
\\ \zeta_2(x_\Gamma,t)\end{array}\right). $
Using the Cauchy-Schwartz inequality and Young's inequality, we have
\begin{align}
&- J_0(\xi_i,\zeta_i) \leq (J_0(\xi_i,\xi_i))^{1/2} (J_0(\zeta_i,\zeta_i))^{1/2} \leq J_0(\zeta_i,\zeta_i)+\frac{1}{4}J_0(\xi_i,\xi_i) \notag\\
&-\left(\int_{\Omega_i} (\xi_{i})_t\zeta_i dx + \gamma_M J_1((\xi_{i})_t,\zeta_i) \right) \leq
\frac{1}{2}\left(||(\xi_i)_t||^2_{\Omega_i}+\gamma_MJ_1((\xi_i)_t,(\xi_i)_t) \right) +E_{\eta}^\zeta(t)
\notag\\
&\sum_{e\in\me_{h,1}}a_1\xi_{1,e}^{-}[\zeta_1]_e+ \eta\sum_{e\in\me_{h,2}}a_2\xi_{2,e}^{-}[\zeta_2]_e\leq
\frac{1}{2}\left(\sum_{e\in\me_{h,1}}|a_1|[\zeta_1]_e^2
+\eta\sum_{e\in\me_{h,2}}|a_2|[\zeta_2]^2_e\right)
\notag\\
&+\frac{1}{2}\left(\sum_{e\in\me_{h,1}}|a_1| \|\xi_1\|_e^2+\eta\sum_{e\in\me_{h,2}}|a_2| \|\xi_2\|^2_e\right).
\end{align}
Using the above three inequalities, we get
\begin{align}
\frac{d}{dt}E_{\eta}^\zeta(t)&\leq
-\mathbf{\zeta_\Gamma}^T S \mathbf{\zeta_\Gamma}
+E_{\eta}^\zeta(t) + \frac{1}{2}\left(||(\xi_1)_t||^2_{\Omega_1}+\gamma_MJ_1((\xi_1)_t,(\xi_1)_t)
+\frac{\gamma_A}{2}J_0(\xi_1,\xi_1)\right)\notag\\
&+\frac{\eta}{2}\left(||(\xi_2)_t||^2_{\Omega_2}+\gamma_MJ_1((\xi_2)_t,(\xi_2)_t)
+\frac{\gamma_A}{2}J_0(\xi_2,\xi_2)\right)\notag\\
&+\frac{1}{2}\left(\sum_{e\in\me_{h,1}}|a_1| \|\xi_1\|_e^2+\eta\sum_{e\in\me_{h,2}}|a_2| \|\xi_2\|^2_e\right)\notag\\
&+|a_1| \|\xi_1\|_{\Omega_1}\|(\zeta_1)_x\|_{\Omega_1}
+|1-\lambda_1| |a_1|\|\xi_{1}\|_{\Gamma} \|\zeta_{1}\|_{\Gamma}
+|\lambda_1||a_2|\|\xi_{2}\|_{\Gamma}\|\zeta_{1}\|_{\Gamma}\notag\\
&+\eta|a_2|\|\xi_2\|_{\Omega_2}\|(\zeta_2)_x\|_{\Omega_2}
+\eta|1+\lambda_2||a_2|\|\xi_{2}\|_{\Gamma}\|\zeta_{2}\|_{\Gamma}
+\eta|\lambda_2||a_1|\|\xi_{1}\|_{\Gamma}\|\zeta_{2}\|_{\Gamma}.
\label{ieq:error:rhs0}
\end{align}
Note that the approximation properties of $\pi_{h,i}$ (equation \eqref{interpolation} with $k=0,1$),  together with the trace inequalities \eqref{eq:traceH1}-\eqref{eq:traceH1G}, and the stability of the extension operator \eqref{eq:stabextension} yields
 \begin{align}\label{eq:projectionerror}
\|\xi_i\|_{\Omega_i}^2 &\lesssim h^{2r+2}\|u_i\|_{r+1,\Omega_i}^2,\\
\sum_{e\in\me_{h,i}} \|\xi_i\|_e^2 &\lesssim h^{2r+1}\|u_i\|_{r+1,\Omega_i}^2, \quad
\|\xi_i\|_\Gamma^2 \lesssim h^{2r+1}\|u_i\|_{r+1,\Omega_i}^2,
 \label{eq:projectionerror2} \\
J_s(\xi_i,\xi_i)&\lesssim h^{2r+1+s}||u_i||^2_{{r+1,\Omega_i}},  \quad i=1,2. \label{eq:errorJ0}
\end{align}
Using \eqref{eq:projectionerror}, Young's inequality, and the inverse inequality \eqref{inverse:discrete1} we have
\begin{align}\label{ieq:error:xi}
&\|\xi_i\|_{\Omega_i}\|(\zeta_i)_x\|_{\Omega_i}
\lesssim h^{-2}||\xi_i||_{\Omega_i}^2+h^{2}||(\zeta_i)_x||_{\Omega_i}^2
\lesssim
h^{2r}||u_i||^2_{r+1,\Omega_i}
+||\zeta_i||^2_{\Omega_i},\, i=1,2.
\end{align}
Using the trace inequality \eqref{eq:traceH1G}, the inverse inequality \eqref{inverse:discrete1}, Young's inequality, and \eqref{eq:projectionerror2} we have  for $i,j=1,2$,
\begin{align}
\|\xi_{i}\|_{\Gamma }\|\zeta_{j}\|_{\Gamma }
\lesssim
\|\xi_{i}\|_{\Gamma }h^{-\frac{1}{2}}\|\zeta_{j}\|_{\Omega_i}
\lesssim  h^{-1}||\xi_i||^2_{\Gamma}+||\zeta_j||^2_{\Omega_j}
\lesssim  h^{2r}||u_i||^2_{r+1,\Omega_i}+||\zeta_j||^2_{\Omega_j}.
\end{align}
Furthermore, since by assumption $(u_i)_t\in H^{r+1}({\Omega_i})$, we have similar estimates as \eqref{eq:projectionerror} and \eqref{eq:errorJ0} for $(\xi_i)_t$ and hence
\begin{align}
||(\xi_i)_t||^2_{\Omega_i}+\gamma_MJ_1((\xi_i)_t,(\xi_i)_t)&\lesssim h^{2r+2} ||(u_i)_t||^2_{r+1,\Omega_i}.
\label{ieq:projectionerror3}
\end{align}
Therefore, combining the inequality \eqref{ieq:error:rhs0} with the inequalities \eqref{eq:errorJ0}-\eqref{ieq:projectionerror3}, and using that $S$ is positive semi-definite, we have
\begin{align}
\frac{d}{dt}E_{\eta}^\zeta(t) \lesssim E_{\eta}^\zeta(t)+h^{2r}.
\label{ieq:error:dE}
\end{align}
Similar to the analysis in \cite{fu2021high} we also have that the initial error $E_{\eta}^\zeta(0) \lesssim h^{2r}$. Then, using Gr\"{o}nwall's inequality  we have  $E_{\eta}^\zeta(t) \lesssim C_th^{2r}$, where $C_t$ denotes a constant depending on time $t$. Using the definition of $E_{\eta}^\zeta(t)$ we have
$$
\min\{1,\eta\}\sum_{i=1}^2\left(\|\zeta_i\|^2_{\Omega_i}
+\gamma_MJ_1(\zeta_i,\zeta_i)\right)
\lesssim E_{\eta}^\zeta(t) \lesssim C_th^{2r}.
$$
Finally,  applying the triangle inequality, using the estimate \eqref{eq:projectionerror} and the bound above for  $||\zeta_i||^2_{\Omega_i}$,  we have the error estimate
\begin{align}
||u-u_h||^2_{\Omega_1 \cup \Omega_2}=\sum_{i=1}^2||\xi_i+\zeta_i||^2_{\Omega_i} \lesssim  \sum_{i=1}^2\left(||\xi_i||^2_{\Omega_i}+||\zeta_i||^2_{\Omega_i}\right) \lesssim h^{2r}.
\end{align}

Note that the error estimate depends on the parameter $\eta$, which is used in the stability analysis to ensure that the matrix $S$ is positive semi-definite and the scheme is stable. We point out that the estimate we have shown is suboptimal, but in the numerical computations we get optimal accuracy.

\end{proof}

\end{appendix}

\renewcommand\refname{Reference}
\bibliographystyle{abbrv}
\bibliography{CutDG}
\end{document}